\documentclass{article}
\usepackage[utf8]{inputenc}
\usepackage{fullpage}
\usepackage{fancyhdr}
\usepackage{amsmath,amssymb,amsthm}
\usepackage{tikz}
\usepackage{wrapfig} %
\usepackage{graphicx} %
\usepackage{tikz-cd}
\usepackage[margin=1.2in]{geometry} %
\usepackage[backend=biber]{biblatex}
\usepackage{tipa} %
\usepackage{mathtools} %
\usepackage{caption}
\usepackage{thmtools, thm-restate} %
\usepackage{xspace}
\usepackage[capitalize]{cleveref}

\usetikzlibrary{decorations.pathreplacing} %

\graphicspath{graphics/}

\DeclareMathOperator{\Ba}{\mathcal{B}}
\DeclareMathOperator{\Ca}{\mathcal{C}}
\DeclareMathOperator{\Da}{\mathcal{D}}

\DeclareMathOperator{\Ga}{\mathcal{G}}
\DeclareMathOperator{\Ha}{\mathcal{H}}

\DeclareMathOperator{\La}{\mathcal{L}}
\DeclareMathOperator{\Ma}{\mathcal{M}}

\DeclareMathOperator{\Oa}{\mathcal{O}}

\DeclareMathOperator{\Sa}{\mathcal{S}}

\DeclareMathOperator{\Ua}{\mathcal{U}}
\DeclareMathOperator{\Va}{\mathcal{V}}

\DeclareMathOperator{\Xa}{\mathcal{X}}
\DeclareMathOperator{\Ya}{\mathcal{Y}}

\DeclareMathOperator{\Cb}{\mathbb{C}}
\DeclareMathOperator{\Db}{\mathbb{D}}

\DeclareMathOperator{\Kb}{\mathbb{K}}

\DeclareMathOperator{\Nb}{\mathbb{N}}

\DeclareMathOperator{\Pb}{\mathbb{P}}
\DeclareMathOperator{\Qb}{\mathbb{Q}}
\DeclareMathOperator{\Rb}{\mathbb{R}}
\DeclareMathOperator{\Sb}{\mathbb{S}}
\DeclareMathOperator{\Tb}{\mathbb{T}}

\DeclareMathOperator{\Zb}{\mathbb{Z}}

\newcommand{\Dc}{\mathfrak{D}}

\DeclareMathOperator{\im}{\mathsf{im}}
\DeclareMathOperator{\Hom}{\mathsf{Hom}}
\DeclareMathOperator{\Spec}{\mathsf{Spec}}

\newcommand{\ep}{\varepsilon}

\DeclareMathOperator{\op}{^\text{op}}
\DeclareMathOperator{\inv}{^{-1}}
\DeclareMathOperator{\id}{\mathsf{id}}
\renewcommand{\ker}{\mathsf{ker}}

\DeclareMathOperator{\Aut}{\mathsf{Aut}}
\DeclareMathOperator{\BAut}{\mathsf{BAut}}

\DeclareMathOperator{\colim}{\mathsf{colim}}

\DeclareMathOperator{\Prop}{\textbf{Prop}}
\DeclareMathOperator{\Type}{\textbf{Type}}

\DeclareMathOperator{\refl}{\mathsf{refl}}
\DeclareMathOperator{\fib}{\mathsf{fib}}

\DeclareMathOperator{\fst}{\term{fst}}
\DeclareMathOperator{\snd}{\term{snd}}

\DeclareMathOperator{\tr}{\term{tr}}
\DeclareMathOperator{\ap}{\term{ap}}
\DeclareMathOperator{\at}{\term{\, at\,}}

\newcommand{\sslash}{\mathbin{/\mkern-6mu/}} %
\newcommand{\modal}{{\Diamond}}

\DeclareMathOperator{\Set}{\textbf{Set}}

\newcommand{\yo}{\term{y}}

\newcommand{\dprod}[1]{({#1}) \to }           %
\newcommand{\dsum}[1]{({#1}) \times }       %
\newcommand{\type}[1]{\mathsf{{#1}}}
\newcommand{\term}[1]{\mathsf{{#1}}}

\newcommand{\trunc}[1]{\left\lVert#1\right\rVert}       %
\newcommand{\B}{\textbf{B}}
\newcommand{\pt}{\term{pt}}
\newcommand{\ord}[1]{\underline{\type{#1}}}

\newcommand{\shape}{\mathbin{\textup{\textesh}}}

\newcommand{\exemplar}{exemplar\xspace}
\newcommand{\exemplars}{exemplars\xspace}

\newcommand{\twisted}[1]{^{\circlearrowright {#1}}}

\newcommand{\xto}[1]{\xrightarrow{#1}}

\newcommand{\pto}{\,\cdot\kern-.1em{\to}\,}

\makeatletter
\providecommand*{\xmapstofill@}{%
  \arrowfill@{\mapstochar\relbar}\relbar\rightarrow
}
\providecommand*{\xmapsto}[2][]{%
  \ext@arrow 0395\xmapstofill@{#1}{#2}%
}

\makeatletter
\def\slashedarrowfill@#1#2#3#4#5{%
  $\m@th\thickmuskip0mu\medmuskip\thickmuskip\thinmuskip\thickmuskip
   \relax#5#1\mkern-7mu%
   \cleaders\hbox{$#5\mkern-2mu#2\mkern-2mu$}\hfill
   \mathclap{#3}\mathclap{#2}%
   \cleaders\hbox{$#5\mkern-2mu#2\mkern-2mu$}\hfill
   \mkern-7mu#4$%
}
\def\rightslashedarrowfill@{%
  \slashedarrowfill@\relbar\relbar\mapstochar\rightarrow}
\newcommand\xslashedrightarrow[2][]{%
  \ext@arrow 0055{\rightslashedarrowfill@}{#1}{#2}}
\makeatother

\tikzset{
    vert/.style={anchor=south, rotate=90, inner sep=.5mm}
} %

\newcommand{\latticegen}[1]{\Zb\!{#1} \oplus \Zb}

\newtheorem{thm}{Theorem}[subsection]

\theoremstyle{definition}
\newtheorem{defn}[thm]{Definition}
\newtheorem{informal}[thm]{Informal Definition}

\newtheorem{ex}[thm]{Example}
\newtheorem{axiom}{Axiom}

\newtheorem{rmk}[thm]{Remark}
\newtheorem{warn}[thm]{Warning}

\newtheorem*{acknowledgements}{Acknowledgements}

\newtheorem{lem}[thm]{Lemma}

\newtheorem{cor}[thm]{Corollary}
\newtheorem{prop}[thm]{Proposition}

\begin{document}

\title{Orbifolds as Microlinear Types in Synthetic Differential Cohesive Homotopy Type Theory}
\author{David Jaz Myers}
\date{\today}
\maketitle

\begin{abstract}
  Informally, an orbifold is a smooth space whose points may have finitely many internal symmetries. Formally, however, the notion of orbifold has been presented in a number of different guises --- from Satake's $V$-manifolds to Moerdijk and Pronk's proper \'etale groupoids --- which do not on their face resemble the informal definition. The reason for this divergence between formalism and intuition is that the points of spaces cannot have internal symmetries in traditional, set-level foundations. The extra data of these symmetries must be carried around and accounted for throughout the theory. More drastically, maps between orbifolds presented in the usual ways cannot be defined pointwise, and instead must use intermediary spaces.

  In this paper, we will put forward a definition of orbifold in synthetic differential cohesive homotopy type theory: an orbifold is a microlinear type for which the type of identifications between any two points is properly finite. In homotopy type theory, a point of a type may have internal symmetries, and we will be able to construct examples of orbifolds by defining their type of points directly. All maps between these orbifolds may be defined pointwise; the mapping space between two orbifolds is merely the type of functions.

We will justify this synthetic definition by proving, internally, that every proper \'etale groupoid is an orbifold. In this way, we will show that the synthetic theory faithfully extends the usual theory of orbifolds. Along the way, we will investigate the microlinearity of higher types such as \'etale groupoids, showing that the methods of synthetic differential geometry generalize gracefully to higher analogues of smooth spaces. We will also investigate the relationship between the Dubuc-Penon and open-cover definitions of compactness in synthetic differential geometry, and show that any discrete, Dubuc-Penon compact subset of a second-countable manifold is subfinitely enumerable.
  \end{abstract}

\tableofcontents

\section{Introduction}

\subsection{What are orbifolds, and what could they be?} Informally, an orbifold is a smooth space whose points may have finitely many internal symmetries.

\begin{informal}\label{informal:orbifold}
An \emph{orbifold} is a smooth space $X$ whose points $x$ have finite groups $\Aut_{X}(x)$ of internal symmetries, known as their \emph{isotropy groups}.
\end{informal}

A paradigmatic example of an orbifold is the quotient of a manifold by the action of a finite group; the smooth structure comes from the manifold, and we may think of the stabilizer group of a point as its internal symmetries.

A common way to model the notion of orbifold uses \emph{proper \'etale groupoids}, which are groupoids internal to the category of smooth manifolds (Lie groupoids) satisfying certain properties \cite{Moerdijk-Pronk:Orbifolds}. However, the correct notion of sameness for these sorts of orbifolds is not equivalence of groupoids, but a separate notion of \emph{Morita equivalence}. Worse, to get all the morphisms between two such orbifolds we may need to replace one of them by a Morita equivalent orbifold first. While this sort of situation is standard fair for categorical homotopy theory, it does not directly capture the intuitive idea of an orbifold as a smooth space whose points have internal symmetries. The issue is that in the usual set theoretic foundations, the elements of sets \emph{cannot} have internal symmetries, and therefore we must carry the data of these symmetries around and account for them at every step of our theory.

In this paper, we will investigate the notion of orbifold in the setting of homotopy type theory, where points can be non-trivially self identified on the foundational level. Homotopy type theory is a novel foundation of mathematics which is based on the intuition that each mathematical object is a certain \emph{type} of mathematical object. For example, $3$ is an integer, $\pi$ is a real number, and $\textbf{GL}_{n}(\Rb)$ is a Lie group. For any two objects $x$ and $y$ of the same type --- say, any two vector spaces --- we may consider what it means to \emph{identify} $x$ with $y$ --- for vector spaces, we would identify $x$ with $y$ by giving a linear isomorphism between them. Because an identification between mathematical objects is another mathematical object, we also get a type of all identifications between $x$ and $y$ which we write as $(x = y)$.

  Identification between mathematical objects is the only form of equality available in homotopy type theory --- the more traditional \emph{proposition} of equality just occurs in the case that there is at most one way to identify two objects of a given type. For example, there is at most one way to identify two natural numbers: when they're equal, we may trivially identify them, and otherwise we can't. We call types where identification is the proposition of equality --- types where there is at most one way to identify two elements --- \emph{sets}. Just as in set theory, the elements of sets cannot have any non-trivial self-identifications. However, there are types in which there can be multiple ways to identify two different objects, such as the type of vector spaces (since there are many linear isomorphisms between two vector spaces).

  In these higher types, objects $x : X$ may have nontrivial self-identifications in the type $\Aut_{X}(x) :\equiv (x = x)$.\footnote{To distinguish between identifications $(=)$ and definitional equalities, we use the symbol $\equiv$ for a definitional equality. We will also put a colon next to the symbol --- as in $:\equiv$ --- to show that the left hand side is defined to be the right hand side. For more on notation, see \cref{sec:notation}.} That is, homotopy type theory gracefully generalizes from set-level mathematics to groupoid-level mathematics and higher. Because the elements of types can have non-trivial self identifications --- the type $(x = x)$ can have more than one element --- we can quite directly formalize \cref{informal:orbifold}.
\begin{defn}\label{defn:orbifold}
An \emph{orbifold} is a \emph{microlinear type} (\cref{defn:microlinear}) $X$ whose types of identifications $(x = y)$ are \emph{properly finite} for all points $x, y : X$.
\end{defn}

The notion of microlinearity is a good formalization of ``smooth space''. The technical notion of ``properly finite'' is needed due to a quirk of constructive mathematics --- saying that the identification types were simply finite would be much too strong. We will discuss these notions further in this introduction.

Our main goal in this paper will be to justify \cref{defn:orbifold}. We will do this in two main ways. First, in \cref{sec:orbifold.examples}, we will construct explicit examples of orbifolds by saying what their points are. For example, in \cref{defn:M11} we will construct the moduli space $\Ma_{1,1}$ of complex elliptic curves as the type of pairs $(\La, \Lambda)$ where $\La$ is a $1$-dimensional complex vector space and $\Lambda \subseteq \La$ is a lattice in it --- the elliptic curve itself is the torus $\La / \Lambda$. The symmetries of a point $(\La, \Lambda) : \Ma_{1,1}$ may be identified with $\Cb$-linear automorphisms of $\La$ that fix the lattice $\Lambda$. This will let us show in \cref{prop:M11.homotopy.quotient} that $\Ma_{1,1}$ may be equivalently defined as the quotient $\mathfrak{h} \sslash \type{SL}_{2}(\Zb)$ of the upper half plane $\mathfrak{h}$ by the action of $\type{SL}_{2}(\Zb)$ via M{\:o}bius transformations.

Second, in \cref{thm:ordinary.proper.etale.groupoid.is.orbifold}, we will prove that any crisp, ordinary proper \'etale groupoid is an orbifold in the sense of \cref{defn:orbifold}. The extra adjectives ``crisp'' and ``ordinary'' are just there to say that we mean proper \'etale groupoids in the ordinary, external sense. This theorem shows that our type theoretic definition of orbifold subsumes the standard definition.

This introduction is structured as an outline of the paper. The remaining introductions in this introduction introduce Sections 2-6 of this paper. We conclude in section 7 with a brief summary of our results.

\subsection{Quotients of group actions in homotopy type theory}

Homotopy type theory takes a novel perspective on the theory of groups which allows for the construction of quotients by group actions without using any colimits. This approach takes very seriously the idea that a group is to be considered as the type of symmetries of a given mathematical object. Instead of working with a group $G$ itself, we work instead with a type $\B G$ of \emph{exemplars} of $G$ --- mathematical objects whose group of symmetries is $G$, at least up to conjugation --- together with a canonical \exemplar $\pt_{\B G} : \B G$ whose self-identifications $(\pt_{\B G} =_{\B G} \pt_{\B G})$ we identify with $G$.

For example, we may take $\B\type{GL}_{n}(\Rb)$ to be the type of $n$-dimensional real vector spaces. That is, we take an \exemplar of the group $\type{GL}_{n}(\Rb)$ to be an $n$-dimensional real vector space. We have a canonical \exemplar of $\type{GL}_{n}(\Rb)$: the canonical $n$-dimensional real vector space $\Rb^{n} : \B \type{GL}_{n}(\Rb)$. We have a canonical identification of $\type{GL}_{n}(\Rb)$ with the type of identifications $(\Rb^{n}=_{\B\type{GL}_{n}(\Rb)} \Rb^{n})$, which can be proven to be equivalent to the type of automorphisms of $\Rb^{n}$ considered as an $n$-dimensional real vector space.

As another example, we may take $\B \Aut(\ord{n})$ to be the type of $n$-element sets. That is, we take an \exemplar of the symmetric group $\Aut(\ord{n})$ to be an $n$-element set. We again have a canonical \exemplar, namely the canonical $n$-element set $\ord{n} :\equiv \{0, \ldots, n-1\}$, and again the group $\Aut(\ord{n})$ is canonically identified with the automorphisms of $\ord{n}$ as an $n$-element set.

The HoTT approach to groups, reviewed in \cref{sec:higher.groups.intro}, works with the types $\B G$ of \exemplars --- known as \emph{deloopings} of the group $G$ --- rather than the group $G$ itself with its algebraic structure. This translation is lossless --- everything we might want to do with a group $G$ can be done in terms of a delooping $\B G$. In particular, as explained in \cref{sec:group.actions}, an action of $G$ on a type $X$ may be equivalently given by a function $X\twisted{-} : \B G \to \Type$ which sends any \exemplar $e : \B G$ of $G$ to a type $X\twisted{e}$ which we call ``$X$ twisted by $e$'', together with an identification $\pt_{X^{\twisted{-}}} : X\twisted{\pt_{\B G}} \simeq X$ of $X$ twisted by the canonical exemplar $\pt_{\B G}$ with $X$ itself.

For example, the action of $\type{GL}_{n}(\Rb)$ on the set $\Rb^{n}-\{0\}$ of non-zero vectors in $\Rb^{n}$ may be given by the function $V \mapsto V - \{0\} : \B\type{GL}_{n}(\Rb) \to \Type$, noting that this function sends the canonical \exemplar $\Rb^{n}$ to the set $\Rb^{n}-\{0\}$ we were trying to act on. Or, the action of the symmetric group $\Aut(\ord{n})$ on the vector space $\Rb^{n}$ given by permuting the coordinates may be equivalently given by the function $X \mapsto \Rb^{X} : \B\Aut(\ord{n}) \to \B \type{GL}_{n}(\Rb)$ sending a finite set $X$ to the vector space of real-valued functions on $X$, noting that we may canonically identify $\Rb^{\ord{n}}$ with $\Rb^{n}$.

This approach is to groups is as radical as it is elementary. It begins to pay major dividends in the construction of quotients by group actions. In \cref{sec:group.quotients}, we describe a construction of the quotient $X \sslash G$ of the type $X$ by the action of a group $G$ as the type of a pairs $(e, x)$ where $e : \B G$ is an \exemplar of $G$, and $x : X\twisted{e}$ is an element of $X$ twisted by the \exemplar $e$. For example, the configuration space $\Rb^{n} \sslash \Aut(\ord{n})$ of $n$ unlabelled points in $\Rb^{n}$ may be defined as the type of pairs $(X, v)$ where $X$ is an $n$-element set and $v : \Rb^{X}$ is a real-valued function on $X$. The quotient map itself is given by sending $x : X$ to $(\pt_{\B G}, x)$, remembering that we identify $X\twisted{\pt_{\B G}}$ with $X$. In our example, the quotient $\Rb^{n} \to \Rb^{n} \sslash \Aut(\ord{n})$ is given by sending $v : \Rb^{n}$ to the pair $(\ord{n}, v)$.

To see that this type of pairs constructs the quotient, let's consider an identification $p : (\ord{n}, v) = (\ord{n}, w)$ between the images of two vectors $v$ and $w : \Rb^{n}$ under the quotient map. By a few elementary HoTT lemmas concerning identifications, an identification $p$ is equivalently given by a pair of identifications $(\sigma, q)$ where $\sigma : \ord{n} = \ord{n}$ is a self-identification of the canonical $n$-element set with itself --- a permutation --- and where $q : v \circ \sigma = w$ is an identification of $v$ with $w$ \emph{relative} to the identification $\sigma$. That is, the type $(\ord{n}, v) = (\ord{n}, w)$ is equivalent to the set of all permuations $\sigma : \Aut(\ord{n})$ which send $v$ to $w$ under the action of permuting coordinates: $v \circ \sigma = w$.

As you can see, we have gotten something more out of this simple process of taking pairs than just the usual set-theoretic quotient. Yes, if we have a $\sigma$ so that $v \circ \sigma = w$, then we will get an identification $(\ord{n}, v) = (\ord{n}, w)$ of their images in the quotient, so that two vectors in the same orbit of this action are identified in the quotient. But furthermore, we remember exactly which permutations $\sigma$ send $v$ to $w$: the type $(\ord{n}, v) = (\ord{n}, w)$ is the \emph{set} of all such permutations. The quotients $X \sslash G$ constructed by taking pairs $\dsum{e : \B G} X\twisted{e}$ are often called \emph{weak quotients} or \emph{homotopy quotients}, though they are stronger than the usual set theoretic quotient in that they contain more information, and they have nothing in particular to do with continuous deformation.

Using this elementary construction of the quotients of types by the actions of groups, we will explicitly construct a number of examples of so-called \emph{good orbifolds} --- the orbifolds arising as quotients of smooth spaces by discrete groups --- in \cref{sec:orbifold.examples}. In general, to construct an orbifold knowing that it may be expressed as the quotient $X \sslash \Gamma$ invovles choosing a good notion of \exemplar for $\Gamma$ --- that is, judiciously choosing a $\B \Gamma$ --- so that the action of $\Gamma$ on $X$ takes a particularly nice form as a function $X\twisted{-} : \B \Gamma \to \Type$. In the end, the construction gives us an explicit definition of the orbifold in terms of its points: the points of $X \sslash \Gamma$ are pairs $(e, x)$ of an \exemplar $e : \B \Gamma$ of $\Gamma$ together with a point $x : X\twisted{e}$ of $X$ twisted by $e$.

\subsection{The homotopy theory of orbifolds via cohesion.}
In his open letter to the homotopy theory community \cite{Barwick:Letter}, Clark Barwick makes the bold claim that
\begin{quote}
Homotopy theory is not a branch of topology.
  \end{quote}
  If homotopy theory is not the study of \emph{homotopies} --- that is, continuous deformations of objects --- what is it? Homotopy type theory offers a striking formal answer to this question: homotopy theory is the study of the way mathematical objects may be \emph{identified}. For example, the homotopy circle $S^{1}$ may be freely generated as a type with a single point $\pt$ a self-indentification $\term{loop} : \pt = \pt$.

  However, if we intend to do algebraic topology --- that is, to use homotopy theoretic methods to study topological spaces and their higher cousins such as orbifolds --- then we will need to distinguish the actual circle
  \[
\Sb^{1} :\equiv \{x : \Rb^{2} \mid |x| = 1\}
\]
from the homotopy circle $S^{1}$. Even more, we should be able to prove that the homotopy circle $S^{1}$ is the type we end up with if we start with the actual circle $\Sb^{1}$ and identify points according to how they may be continuously deformed into each other. That is to say, $S^{1}$ should be the \emph{homotopy type} of $\Sb^{1}$.

In his paper ``Brouwer's fixed point theorem in real cohesive homotopy type theory'' \cite{Shulman:Real.Cohesion}, Shulman gives us the tools to do synthetic algebraic topology in homotopy type theory by adding a system of \emph{modalities} to HoTT which include the \emph{shape} modality $\shape$ that sends a type $X$ to its homotopy type $\shape X$. The shape $\shape X$ of a type $X$ may be defined as the localization of $X$ at the type $\Rb$ of real numbers, so that any path $\gamma : \Rb \to X$ gives us an identification $p(\gamma) : \gamma(0)^{\shape} = \gamma(1)^{\shape}$ in $\shape X$. However, for this construction to behave right, we need the rest of Shulman's cohesion, which we will review in \cref{sec:review.cohesion}.

In \cref{sec:modal.covering.theory}, we will use the theory of modal fibrations and coverings developed in \cite{Jaz:Good.Fibrations} to compute the homotopy types of some orbifolds and hint at their general covering theory. For example, in \cref{thm:homotopy.type.of.M11} we will compute that
\[
\shape \Ma_{1,1} \simeq \B \type{SL}_{2}(\Zb)
\]
the homotopy type of the moduli stack of elliptic curves $\Ma_{{1,1}}$ is a delooping of the group $\type{SL}_{2}(\Zb)$. We'll use this modal covering theory in \cref{sec:orbifold.maps} to briefly investigate maps between orbifolds.

\begin{rmk}\label{rmk:finite.covering}
It is precisely the fluency of covering theory in cohesive homotopy type theory which forces us to use a technical notion of ``properly finite'', rather than ``finite'', in our definition of orbifolds (\cref{defn:orbifold}). A map $f : X \to Y$ whose fibers are finite is necessarily a covering map by the ``good fibrations'' trick of \cite{Jaz:Good.Fibrations} (see also Remark 9.9 of \textit{ibid.}). If for every $x : X$, the type of automorphisms $(x = x)$ of $x$ were finite, then the projection from the inertia orbifold $X^{{S^{1}}} \to X$ (whose fiber over $x$ is $(x = x)$) would be a finite covering; but this would in particular imply that the cardinality of the isotropy group $(x = x)$ is constant on any connected component of $X$. This is problematic for most orbifolds seen in practice.
\end{rmk}

\subsection{Smooth spaces and synthetic differential geometry.}
Orbifolds are \emph{smooth spaces} whose points have internal symmetries, and while moving to homotopy type theory has given us direct access to types whose points have internal symmetries, we have not yet talked about the smooth structure. The formal system of homotopy type theory admits models in all $\infty$-toposes (\cite{Shulman:Models.of.HoTT}), so that a type gets interpreted as a stack of homotopy types, and an element of a type gets interpreted as a map between these stacks. Identifications between elements get interpreted as homotopies between the corresponding maps. We can therefore get the smooth structure we need on our types by interpreting our homotopy type theory in an $\infty$-topos of stacks on a suitable site --- say a site consisting of smooth manifolds.

But we would like to be able to work with this smooth structure from within homotopy type theory itself. To give our types smooth structure, we will use the axioms of \emph{synthetic differential geometry}, which we review in \cref{sec:SDG.axioms}. Synthetic differential geometry is an axiom system for doing differential geometry with nilpotent infinitesimals, first put forward by Lawvere and developed further by Dubuc, Kock, Bunge, Moerdijk, Reyes, and many others. While this axiom system is usually interpreted in $1$-toposes, it can be interpreted in $\infty$-toposes just as well. The Dubuc $\infty$-topos of stacks on a site of infinitesimally extended Euclidean spaces (with smooth maps between them) and the very similar $\infty$-topos which Schreiber calls the topos of formal smooth $\infty$-groupoids \cite{Schreiber:Differential.Cohomology.v2} are models of both cohesion and synthetic differential geometry. These will be our intended models for this paper.\footnote{The site for the Dubuc $\infty$-topos consists of (the opposite category of) $\Ca^{\infty}$-rings of the form $\Ca^{\infty}(\Rb^{n})/I$ where $I$ is a \emph{germ-determined ideal}: $f \in I$ if and only if for all $x \in \Rb^{n}$, the germ $f_{x}$ is in the ideal $I_{x}$ generated by the germs at $x$ of functions in $I$. The site for the topos of formal smooth $\infty$-groupoids has as its objects the $\Ca^{\infty}$-rings of the form $\Ca^{\infty}(\Rb^{n}) \otimes_{\Rb} W$ where $W$ is a \emph{Weil algebra} --- a finitely presented augmented $\Rb$-algebra with finitely generated and nilpotent augmentation ideal. Crucially, the category of euclidean spaces $\Rb^{n}$ and smooth maps embeds into both of these sites by $\Rb^{n} \mapsto \Ca^{\infty}(\Rb^{n})$. But these sites also have \emph{infinitesimal} spaces, such as the dual numbers $\Rb[x]/(x^{2})$, which enable us to work with infinitesimals in the homotopy type theory of these toposes. For a definition of these sites, see the standard reference \cite{Moerdijk-Reyes:SDG} where Dubuc's site is known as $\mathbb{G}$, and Section 6.5 of \cite{Schreiber:Differential.Cohomology.v2}.
}

In synthetic differential geometry, we axiomatize the \emph{smooth reals} $\Rb$ as an ordered field. Crucially, since we are working constructively, just because a number is \emph{not} non-zero does not imply that it is zero. The numbers which are not non-zero are known as \emph{infinitesimals}  (after Penon's \emph{Infinitesimaux et intuisionisme} \cite{Penon:Infinitesimals.and.Intuitionism}). The most crucial axiom of synthetic differential geometry which make these infinitesimals behave as we would like them to is the Kock-Lawvere axiom. As a special case of this axiom, we see that every function $f : \Db \equiv \{\ep : \Rb \mid \ep^{2} = 0\} \to \Rb$ from the set $\Db$ of nilsquare elements of $\Rb$ (the ``first-order infinitesimals'') to $\Rb$ is linear: that is, there is a unique $b : \Rb$ so that $f(\ep) = f(0) + b\ep$ for all $\ep^{2} = 0$.

As a corollary of this axiom, we may define the derivative $f'$ of a function $f : \Rb \to \Rb$ to be the unique function satisfying
\[
f(x + \ep) = f(x) + f'(x)\ep
\]
for all $x : \Rb$ and $\ep^{2} = 0$. This justifies calling $\Rb$ the ``smooth reals'' --- every function $f : \Rb \to \Rb$ is smooth.

In general, we may think of the type of functions $X^{\Db} \equiv ( \Db \to X )$ as the tangent bundle of $X$, with the projection $\pi : X^{\Db} \to X$ given by evaluation at $0$. The tangent space $T_{x} X$ of $x : X$ is therefore the type of functions $v : \Db \to X$ with $v(0) = x$. Because we can make this definition, there is a sense in which every type in synthetic differential geometry has a sort of differentiable structure. However, this structure isn't very much like a manifold's in general. For any type $X$, the tangent spaces $T_{x}X$ admit a scalar multiplication by $\Rb$ defined by $(rv)(\ep) :\equiv v(r\ep)$, but $T_{x} X$ is not generally an $\Rb$-module.

The natural question is then: what are the smooth spaces in synthetic differential geometry? There are a number of answers. We could of course repeat the usual definition of smooth manifold. Or, we could look at spaces that are only \emph{infinitesimally} (and not necessarily locally) isomorphic to Euclidean space; this gives us Penon's notion of manifold. Or, we could look at types which are infinitesimally isomorphic to Euclidean space, but this time in the sense of being related to Euclidean spaces by a zig-zag of {\'e}tale maps; this gives us Schreiber's notion of manifold. For each of these possible definitions of manifold, the tangent spaces $T_{x} X$ will be $\Rb$-modules.

But there is a wider class of spaces that includes all of the above and which the synthetic differential geometry community has settled into as the ``right'' notion of smooth space suitable for proving theorems: \emph{microlinear} spaces. Microlinear spaces have all the infinitesimal linear properties that $\Rb$ does, in a sense which we will make precise in \cref{sec:microlinear}.  And, since microlinear spaces may be defined by lifting uniquely on the right against a given class of maps (\cref{lem:microlinear.lifting.property}), they have good closure properties.

As a further advantage, the definition of microlinearity applies just as well to higher types as to sets. In particular, a standard theorem in synthetic differential geometry proves that the tangent spaces of microlinear sets are $\Rb$-modules; this is likely the reason these spaces are called ``microlinear''. In \cref{thm:inf.linear.R.module}, we will prove this fact in such a way that it applies not only to sets but also groupoids and general higher types. That is, if $X$ is a microlinear type, not necessarily a set, then its tangent spaces admit fully coherent $\Rb$-module structures.

Of the three definitions of manifold given above, only Schreiber's generalizes to higher types; but this definition relies on a choice of atlas, whereas microlinearity is ``coordinate-free''.

Just as we introduced the shape modality $\shape$ to study the topology of orbifolds by trivializing it through the nullification of $\Rb$, we will introduce the \emph{crystaline} modality $\Im$ in \cref{sec:crystaline.modality} to study the diffeology of orbifolds by trivializing it through the nullification of the set $\Dc$ of infinitesimal real numbers. The modality $\Im$ was called the ``infinitesimal shape modality'' by Schreiber in \cite{Schreiber:Differential.Cohomology.v2}, and was studied in homotopy type theory by Cherubini in \cite{Cherubini:Thesis}.

We will mainly use the $\Im$ modality for its \'etale maps. A map is $\Im$-\'etale when its $\Im$-naturality square is a pullback:
\[
\begin{tikzcd}
	X & {\Im X} \\
	Y & {\Im Y}
	\arrow["f"', from=1-1, to=2-1]
	\arrow["{(-)^{\Im}}", from=1-1, to=1-2]
	\arrow["{(-)^{\Im}}"', from=2-1, to=2-2]
	\arrow["{\Im f}", from=1-2, to=2-2]
	\arrow["\lrcorner"{anchor=center, pos=0.125}, draw=none, from=1-1, to=2-2]
\end{tikzcd}
\]
In \cref{thm:microlinear.descends.along.etale}, we will show that microlinearity descends along surjective $\Im$-\'etale maps. That is, if $X$ is microlinear and $f : X \to Y$ is surjective and $\Im$-\'etale, then $Y$ is also microlinear. We will use this theorem together with the ``good fibrations'' trick of \cite{Jaz:Good.Fibrations} to show that quotients of microlinear spaces by discrete groups are themselves microlinear (\cref{thm:discrete.homotopy.quotient.is.microlinear}). This proves in particular that the good orbifolds constructed in \cref{sec:orbifold.examples} are microlinear.

\subsection{Smooth spaces are microlinear}

With \cref{thm:microlinear.descends.along.etale} in hand, we will show that all sorts of smooth spaces are microlinear. In particular, we will show that ordinary smooth manifolds are microlinear (\cref{sec:ordinary.manifold}), as are the synthetic manifolds of Penon (\cref{sec:penon.manifold}) and Schreiber (\cref{sec:Schreiber.manifold}). We will also compare the modal notion of $\Im$-\'etale map with the usual notion of local diffeomorphism, showing in \cref{lem:etale.means.etale.for.ordinary.manifolds} that these notions coincide between crisp, ordinary manifolds.

Most importantly, in \cref{sec:etale.groupoid} we will prove in \cref{thm:etale.groupoid.microlinear} that \emph{\'etale groupoids} are microlinear. An \'etale groupoid is --- roughly speaking --- a groupoid $\Ga$ for which the source map $s : \Ga_{1} \to \Ga_{0}$ sending a morphism of $\Ga$ to its source is $\Im$-\'etale. Classically, these are a class of locally discrete Lie groupoids which contain the proper \'etale Lie groupoids that present orbifolds. For this reason, \cref{thm:etale.groupoid.microlinear} is a major step on the way to proving that all proper \'etale groupoids are orbifolds in the sense of \cref{defn:orbifold}.

The proof of \cref{thm:etale.groupoid.microlinear} involves a number of colimit-preserving properties of the modality $\Im$ --- properties which $\Im$ shares with $\shape$. That $\Im$ commutes with crisp pushouts and colimits of sequences follows from the assumption that the type $\Dc$ of infinitesimal real numbers is \emph{tiny}. This is a crucial assumption of synthetic differential geometry which we do not fully explore in this paper. Rather, we push the definition of tiny type and the requisite lemmas to \cref{sec:tiny}.

The main lemma in the proof of \cref{thm:etale.groupoid.microlinear} is an \'etale descent theorem, \cref{thm:etale.descent}, which states that if the pullback of a crisp map $f$ along itself is $\Im$-\'etale, then $f$ is itself $\Im$-\'etale. We prove this descent theorem for any modality that commutes with crisp colimits, which also includes $\shape$.

In \cref{sec:deloopings.inf.linear}, we will investigate the microlinearity of deloopings $\B G$ of microlinear groups $G$, which include the Lie groups. While I was not able to prove that $\B G$ is microlinear, we can prove in \cref{thm:inf.linear.BG} that $\B G$ is \emph{infinitesimally linear} --- a weaker condition than microlinearity --- so that at least its tangent spaces are (higher) $\Rb$-modules. The tangent space $T_{\pt_{\B G}}\B G$ of $\B G$ at its canonical \exemplar $\pt_{\B G}$ is a delooping of the Lie algebra $\mathfrak{g} :\equiv T_{1}G$, and the map $e \mapsto T_{e} \B G : \B G \to \Type$ is a delooping of the adjoint action of $G$ on $\B \mathfrak{g} :\equiv T_{\pt_{\B G}}\B G$.

\subsection{Finiteness and compactness}

Finally, we turn our attention to proving \cref{thm:ordinary.proper.etale.groupoid.is.orbifold}. This theorem states that crisp ordinary proper \'etale groupoids are orbifolds in the sense of \cref{defn:orbifold}. An ordinary proper \'etale groupoid is a groupoid $\Ga$ for which the spaces $\Ga_{0}$ of objects and $\Ga_{1}$ of morphisms are both ordinary smooth manifolds, where the source map $s : \Ga_{1} \to \Ga_{0}$ is $\Im$-\'etale (which by \cref{lem:etale.means.etale.for.ordinary.manifolds} means that $s$ is a local diffeomorphism in the ordinary sense), and where the map $(s, t) : \Ga_{1} \to \Ga_{0} \times \Ga_{0}$ is \emph{proper}.

The usual definition of a proper map is that the inverse image of any compact set is compact. If we used the usual definition of compact --- that any cover admits a finitely enumerable subcover --- then we could prove that the fibers of any proper map are in fact finite sets. This is of course to be expected, but remember: being finite is a strong condition in cohesive homotopy type theory. Namely, if a map has finite fibers, then it is a covering map. This won't do, because that would imply that $(s, t) : \Ga_{1} \to \Ga_{0} \times \Ga_{0}$ is a finite cover, which would mean that the cardinality of the isotropy groups $\Ga(x, x) = (s, t)\inv(x, x)$ would be constant over any connected component of $\Ga_{0}$. This is almost never true of orbifolds in practice; for example, the quotient $\Rb^{2} \sslash C_{k}$ of the plane by rotation by $\frac{2\pi}{k}$ has non-trivial isotropy group $C_{k}$ only at the origin.

For this reason, we will need a different definition of proper map, which means a different definition of compact set, which ultimately relies on a different notion of ``finite''. In \cref{sec:finiteness}, we will introduce \emph{properly finite} sets: discrete subquotients of finite sets. We then relate this new notion of finiteness to an appropriate notion of compactness.

Luckily, Dubuc and Penon have already explored a beautifully creative definition of ``compact set'' in the setting of synthetic differential geometry. A set $K$ is Dubuc-Penon compact if universal quantification over $K$ commutes with logical or: that is, if for any proposition $A$ and predicate $B : K \to \Prop$, if for all $k$ it is the case that $A$ holds or $B(k)$ holds, then either $A$ holds or for all $k$, $B(k)$ holds:
\[
(\forall k : K.A \vee B(k)) \Rightarrow (A \vee \forall k : K.\, B(k)).
\]
This  is an intrinsic property of the set $K$. In \cite{Dubuc-Penon:Compact}, Dubuc and Penon prove that in the various toposes of interest, a sheaf represented by a smooth manifold is Dubuc-Penon compact if and only if that manifold is compact in the ordinary sense.

In \cref{sec:compact}, we will prove an internal version of Dubuc and Penon's theorem in \cref{prop:crisply.refinement.subcompact.is.compact.for.crisp.covers}: for any crisp, Dubuc-Penon compact subset $K$ of an ordinary manifold, any crisp open cover of $K$ admits a finite subcover. This result follows as a corollary of \cref{prop:db.compact.is.countably.compact}, which states that any Dubuc-Penon compact set $K$ is countably compact: any countably enumerable Penon open cover of $K$ admits a finitely enumerable subcover. This theorem is proven with a key lemma, \cref{prop:compact.open.relation.fatness}, which states that for any Dubuc-Penon compact set $K$ and any relation $r \subseteq K \times \Rb$, if $r(k, x)$ for all $k : K$, then there exists an $\ep > 0$ such that $r(k, y)$ for all $k : K$ and $y \in B(x, \ep)$ in the $\ep$-ball around $x$. This key lemma was extracted from the proof that Gago gives in his thesis \cite{Gago:Thesis} that any positive valued function $f : K \to (0, \infty)$ is bounded away from $0$ (\cref{lem:compact.bounded}). All of this relies crucially on the Covering Property, originally due to Bunge and Dubuc \cite{Bunge-Dubuc:Covering.Property}, which is assumed of the smooth reals: if $A \cup B = \Rb$, then for any $x : \Rb$ there is an $\ep > 0$ so that $B(x, \ep) \subseteq A$ or $B(x, \ep) \subseteq B$.

With this analysis of Dubuc-Penon compact subsets in hand, we begin \cref{sec:final.theorem}. We will prove in \cref{thm:discrete.compact.subset.of.manifolds.is.properly.finite} that discrete Dubuc-Penon compact subsets of ordinary manifolds are properly finite. To finish the proof of our main \cref{thm:ordinary.proper.etale.groupoid.is.orbifold}, then, it remains to show that if $(s, t) : \Ga_{1} \to \Ga_{0} \times \Ga_{0}$ is Dubuc-Penon proper and that $s : \Ga_{1} \to \Ga_{0}$ is $\Im$-\'etale, then the hom sets $\Ga(x, y)$ of this proper \'etale groupoid $\Ga$ are discrete. We accomplish this final lemma in \cref{lem:crystaline.subset.of.manifold.is.discrete}, showing that crystaline subsets of ordinary manifolds are discrete.

We may then conclude that all (crisp, ordinary) proper \'etale groupoids are orbifolds in the sense of \cref{defn:orbifold}, justifying that definition. In \cref{sec:global.quotient}, we show that the quotient of a microlinear set by the action of a finite group is an orbifold, and quickly prove that orbifolds are closed under pullback.

\begin{acknowledgements}
Special thanks go to Emily Riehl, whose neverending support, critique, and patience made this paper possible. Thanks also to Ben Dees for help with the most elementary analysis with both hands tied his back. I also want to thank Keaton Stubis for explaining the relationship between $\type{SL}_{2}(\Zb)$ and lattices in $\Cb$ to me. The author appreciates the support of ARO under MURI Grant W911NF-20-1-0082.
\end{acknowledgements}

\subsection{Notation and terminology}
\label{sec:notation}

In this paper, we will use the following Agda-inspired syntax for types of pairs (dependent sum types) and types of functions (dependent product types):
\begin{itemize}
\item $\dsum{a : A} B(a)$ is the type $\displaystyle \sum_{a : A} B(a)$ with elements $(a, b)$ for $a : A$ and $b : B(a)$;
\item $\dprod{a : A} B(a)$ is the type $\displaystyle \prod_{a : A}B(a)$ with elements the maps $a \mapsto f(a)$ with $f(a) : B(a)$.
        \end{itemize}

        We will use the colon $a : A$ to say that $a$ is an element of type $A$. We will write $\Type$ for a (univalent) universe of types, and so a type family will be a function $B : A \to \Type$.

We will always use $(a = b)$ to denote the type of identifications of an element $a : A$ with an element $b : A$. Sometimes, we will write $(a =_{A} b)$ to emphasize that $a$ and $b$ are being identified \emph{as elements of type $A$}. We will use the symbol $\equiv$ to denote definitional (also known as judgemental) equality: $a \equiv b$ means that $a$ is \emph{defined to be} $b$. We always refer to elements $p : a = b$ as ``identifications of $a$ with $b$'' and never as ``paths'' (a term which we reserve for maps $\gamma : \Rb \to A$ from the real line).

        We will use $\Prop$ to denote the type of propositions, and we will identify subtypes $S \subseteq A$ with their membership predicates $a \mapsto a \in S : A \to \Prop$. In this way, if $P : A \to \Prop$ is any predicate, we will define its corresponding subtype to be the type of pairs:
        \[
\{a : A \mid P(a)\} :\equiv \dsum{a : A} P(a).
\]
We will often identify an element $(a, p) : \{a : A \mid P(a)\}$ of a subtype with its first component $a$, leaving the witness $p : P(a)$ implicit.

        In the case that $A$ is a set, which is to say that for all $a,\, b : A$, the type $(a = b)$ of identifications is a proposition, we will simply say that $a$ and $b$ are \emph{equal} when there is a (necessarily unique) element $p : a = b$. We will almost always leave the elements of propositions implicit, and instead say that the proposition holds.

\section{Constructing Examples of Orbifolds}\label{sec:higher.groups}

We'll begin this paper by constructing examples of orbifolds. These constructions can be performed in standard (Book) homotopy type theory, without any extra assumptions. The goal here is to show how conceptually clean the constructions of particular orbifolds can be in homotopy type theory.

The main takeaway of this section is that we can define orbifolds by saying what their elements are. For example, we will define the moduli space of complex elliptic curves $\Ma_{1,1}$ as the type of pairs $(\La, \Lambda)$ where $\La$ is a $1$-dimensional complex vector space, and $\Lambda \subseteq \La$ is a lattice within it (see \cref{defn:M11}).\footnote{The elliptic curve itself would be the torus $\La/\Lambda$.} As another example, we can define the configuration space of $n$ unlabelled points in a space $X$ as the type of pairs $(F, x)$ where $F$ is an $n$-element set and $x : F \to X$ is a function picking out points in $X$ for every element of $F$ (see \cref{defn:configuration.space}).

 We will focus on constructing \emph{good} orbifolds, those given by a quotient
 $$X \sslash \Gamma$$
 of a manifold $X$ by the action of a discrete group $\Gamma$. Remarkably, in homotopy type theory the quotient of a type by the action of a group may be constructed as a type of pairs of elements --- no equivalence classes necessary. We will review this standard construction in \cref{defn:quotient}.

We will begin this section by reviewing the theory of (higher) groups in homotopy type theory in \cref{sec:higher.groups}. In homotopy type theory, one works with a \emph{delooping} of a group $G$, rather than the group itself. A delooping of $G$ is a type $\B G$ with a fixed element $\pt_{\B G} : \B G$ whose group of symmetries is $G$ --- that is we have an isomorphism $G \simeq (\pt_{\B G} = \pt_{\B G})$ --- and where every other element $e : \B G$ is somehow identifiable with $\pt_{\B G}$, though not canonically. As an example, we may take $\B \type{GL}_{n}(\Rb)$ to be the type of $n$-dimensional real vector spaces with $\pt_{\B \type{GL}_{n}(\Rb)}$ defined to be $\Rb^{n}$; by definition, $\type{GL}_{n}(\Rb)$ is the group of linear automorphisms of $\Rb^{n}$, and every $n$-dimensional vector space is isomorphic to $\Rb^{n}$, though of course not canonically since such isomorphisms are equivalent to a choice of basis. In \cref{defn:exemplar}, we will introduce terminology for the elements of deloopings of groups: we will call $e : \B G$ an \emph{\exemplar} of $G$, and we will call $\pt_{\B G}$ the \emph{canonical \exemplar}.\footnote{Far be it from the author to tell you how to pronounce your own words, but in my opinion, pronouncing
the final syllable of ``\exemplar'' as in the word ``exemplary'' makes it sound much less pompous than pronouncing it as the final syllable of ``templar''.} We'll give a number of examples of \exemplars of common groups to help this concept settle.

Next, in \cref{sec:group.actions}, we will review how actions of a group can be described in terms of a delooping $\B G$. In particular, an action of a group $G$ on a type $X$ is equivalently a function which assigns to any \exemplar $e : \B G$ of $G$ a type $X\twisted{e}$ --- called ``$X$ twisted by $e$'' --- in such a way that $X\twisted{\pt_{\B G}}$ is identified with $X$. Like with deloopings, this is best understood in terms of examples, and so we provide them.

Then, in \cref{sec:group.quotients}, we receive a delightful payout for this reformulation of group theory. We will see in \cref{defn:quotient} that the quotient of a type $X$ by the action of a group $G$ may be constructed as the type of pairs $(e, x)$ with $e : \B G$ is an \exemplar of $G$, and $x : X\twisted{e}$ is an element of the type $X$ twisted by $e$. The quotient map itself sends $x : X$ to the pair $(\pt_{\B G}, x)$, remembering that we identify $X$ with $X\twisted{\pt_{\B G}}$. We'll note that this quotient is the ``weak quotient'' or ``homotopy quotient'' of the action: an identification $p : (\pt_{\B G}, x) = (\pt_{\B G}, y)$ is equivalently given by an element $g : G$ such that $gx = y$ (see \cref{lem:identifications.in.homotopy.quotient} and \cref{rmk:quotient}). In particular, the automorphisms of the point $(\pt_{\B G}, x)$ may be identified with the stabilizer of $x$. In this way, the elements of quotients so constructed pick up non-trivial internal symmetries.

We will make use of this construction of quotients in \cref{sec:orbifold.examples} to provide example constructions of various specific orbifolds. We won't prove that these are orbifolds in the sense of \cref{defn:orbifold} yet --- that will require the theory developed in the rest of the paper. The aim of \cref{sec:orbifold.examples} is rather to show what particular orbifolds can look like in homotopy type theory.

\subsection{(Higher) Groups in homotopy type theory} \label{sec:higher.groups.intro}

In homotopy type theory, we take the maxim that ``a group is the group of symmetries of some mathematical object'' as a definition. A symmetry is a self-identification of this object, considered as an object of a given type. We might therefore think of defining a group as a pair $(X, x)$ of a type $X$ of objects and an object $x : X$ of this type. The group $G$ itself would then be the type of symmetries of this object (as an element of the type $X$):
\[
G \equiv (x =_{X} x).
\]
However, this definition keeps around too much baggage. For the pair $(X, x)$ to be uniquely determined by the group $G \equiv (x = x)$ that it represents, we would need to show that two such pairs $(X, x)$ and $(Y, y)$ are equivalent if and only if their associated groups of symmetries $(x =_{X} x)$ and $(y =_{Y} y)$ are equivalent. However, if $X$ has other elements $x'$ which are not somehow identifiable with $x$, then there is no hope for this. Conversely, though, if every element of $X$ is somehow identifiable with the chosen object $x$ --- if $X$ is $0$-connected --- then we can prove the following \emph{fundamental theorem of higher groups}.
\begin{thm}[Folklore]
  Let $X$ and $Y$ be \emph{pointed, $0$-connected} types. That is, suppose that $\pt_{X} : X$ and $\pt_{Y} : Y$, and that for any $x : X$ there is merely an identification of $x$ with $x_{0}$, and similarly for $y : Y$. That is, suppose we have
 \(
\dprod{x : X} \trunc{x = \pt_{X}}
  \)
  and similarly $\dprod{y : Y} \trunc{y = \pt_{Y}}$. Then any function $ f: X \to Y$ with $\pt_{f} : \pt_{Y} =  f(\pt_{X})$ is an equivalence if and only if the induced function
  \[\Omega f :\equiv p \mapsto \pt_{f} \bullet f_{\ast} p \bullet \pt_{f}\inv : (\pt_{X} = \pt_{X}) \to (\pt_{Y} = \pt_{Y})\]

    is an equivalence.
\end{thm}
\begin{proof}
If $f$ is an equivalence, it is straightforward to show that $\Omega f$ is as well. So, we prove the converse.

Suppose that $\Omega f$ is an equivalence. We will show that $f$ is by showing that its fiber over any $y : Y$ is contractible. Since contractibility is a proposition and $Y$ is $0$-connected, we may assume a $p : \pt_Y = y$, which gives us an equivalence $\fib_f(\pt_Y) \simeq \fib_f(y)$. So, it will suffice to show that $\fib_f(\pt_Y)$ is contractible. Now,
\begin{align*}
  \fib_f(\pt_Y) &:\equiv \dsum{x : X} (\pt_Y = f(x))\\
                 &\phantom{:}\simeq \dsum{x : X} (f(\pt_X) = f(x))
  \end{align*}
by $\pt_f : \pt_Y = f(\pt_X)$. Now, $\Omega f = \pt_f \cdot f_{\ast} \cdot \pt_f\inv$ is an equivalence, so its conjugate $f_{\ast} : (\pt_X = \pt_X) \to (f(\pt_X) = f(\pt_X))$ is an equivalence. But we would like for $f_{\ast} : (\pt_X = x) \to (f(\pt_X) = f(x))$ to be an equivalence, because if it is, then
\begin{align*}
  \fib_f(\pt_Y) &\simeq \dsum{x : X} (f(\pt_X) = f(x)) \\ &\simeq \dsum{x : X} (\pt_X = x) \\ &\simeq \ast.
  \end{align*}
Luckily, $f_{\ast} : (\pt_X = x) \to (f(\pt_X) = f(x))$ being an equivalence is also a proposition, and since $X$ is $0$-connected we may assume a $q : \pt_X = x$. Then we have a commuting square
\[
\begin{tikzcd}
	{(\pt_X = \pt_X)} & {(f(\pt_X) = f(\pt_X))} \\
	{(\pt_X = x)} & {(f(\pt_X) = f(x))}
	\arrow["{\bullet q}"', from=1-1, to=2-1]
	\arrow["{f_\ast}", from=1-1, to=1-2]
	\arrow["{f_\ast}"', from=2-1, to=2-2]
	\arrow["{\bullet f_\ast q}", from=1-2, to=2-2]
\end{tikzcd}
\]by the functoriality of $f_{\ast}$.
In this square, the top map and vertical maps are equivalences,  Therefore, the bottom map is an equivalence, which proves the theorem.
\end{proof}

With this theorem in hand, we can make the following definition of (higher) group in homotopy type theory.

\begin{defn}[\cite{Buchholtz-vanDoorn-Rijke:Higher.Groups}]
  A higher group is a type $G$ identified with the type
  $$\Omega \B G :\equiv (\pt_{\B G} = \pt_{\B G})$$
  of self-identifications of the base point of a pointed, $0$-connected type $\B G$. We refer to $\B G$ as a \emph{delooping} of $G$. We say that $G$ is an \emph{$n$-group} if $\B G$ is $n$-truncated, or equivalently if $G$ itself is $(n-1)$-truncated.
\end{defn}

The basic theory of higher groups is developed in \cite{Buchholtz-vanDoorn-Rijke:Higher.Groups}, with a further development forthcoming the a textbook \cite{SymmetryBook}. By way of summary, in homotopy type theory we work with deloopings as pointed, $0$-connected types, rather than with groups as algebraic structures. This change of perspective has deep ramifications for concrete calculations, which we will try to describe now.

If $G$ is an ordinary ($1$-)group, then we can always deloop it by taking $\B G$ to be the type of $G$-torsors.
\begin{prop}[\cite{SymmetryBook}]
Let $G$ be a $1$-group. A $G$-torsor is a free, transtive, and inhabited (left) action of $G$ on a set. The type $\type{Tors}_{G}$ of $G$-torsors, pointed at $G$ acting on itself on the left, is a delooping of $G$.
\end{prop}

  In order to make this definition a bit more concrete, here is the full definition of the type of $G$ torsors:
  \begin{align*}
    (X : \Type) &\times (\dprod{x,\, y : X} \dprod{p,\,q : x = y} (p = q)) \\
    &\times (\alpha : G \times X \to X) \\
                &\times (\dprod{x : X} (\alpha(1, x) = x)) \\
                &\times (\dprod{x : X} \dprod{g,\, h : G} (\alpha(gh, x) = \alpha(g, \alpha(h, x)))) \\
                &\times \left(\dprod{x,\, y : X}
                  \left\{
                  \begin{aligned}
                    &((g, p) : \dsum{g : G} (\alpha(g, x) = y)) \\
                    \times (&\dprod{(h, q) : \dsum{h : G} (\alpha(h, x) = y)} ((g, p) = (h, q))))
                  \end{aligned}
                              \right.\right) \\
    &\times \trunc{X}
  \end{align*}
  The point of writing this type out in full is to show how a definition of an algebraic structure satisfying certain axioms can be written out using a few basic type constructors. Note that this is a tuple ($\times$) consisting of many functions ($\to$), some of which land in types of identifications ($=$).

  The first pair $\dsum{X : \Type} (\dprod{x,\, y : X} \dprod{p,\,q : x = y} (p = q))$ has elements the \emph{sets}, as defined in homotopy type theory. A set is a type where the type $x = y$ of identifications between two elements is a proposition --- namely, the proposition that $x$ and $y$ are \emph{equal}. A proposition is a type where any two elements may be identified; any witness to the truth of a proposition is as good as any other.

  The next element $\alpha : G \times X \to X$ is the action map itself, and it is followed by the two axioms which define a group action. The second to last element witnesses that this action is a torsor. It says that for any $x$ and $y$ in $X$, there is a unique $g$ in $G$ for which $\alpha(g, x) = y$. Since we assumed that $X$ was a set, the rest of this data after the action map $\alpha$ is a proposition. The last element says that $X$ is inhabited.

  Using common type-theoretical shorthand, we could write this type more succinctly and clearly as:
  \[
    \type{Tors}_{G} :\equiv
    \left\{
  \begin{aligned}
    (X : \Set) &\times (\alpha : G \times X \to X) \\
               &\times (\forall x : X,\, \alpha(1, x) = x) \times (\forall x : X, \forall g,\, h : G,\, \alpha(gh, x) = \alpha(g, \alpha(h, x))) \\
               &\times (\forall x,\, y : X,\, \exists! g : G,\, \alpha(g, x) = y) \\
               &\times \trunc{X}.
\end{aligned}
\right.
\]
In homotopy type theory, once you know the definition of a type of object, you also have constructed the stack which classifies bundles whose fibers are that type of object: the classifying stack is just the type itself. In particular, the type $\type{Tors}_{G}$ of $G$-torsors classifies $G$-principal bundles.

Depending on what we are trying to do with our group, it might be useful to have different constructions of its delooping. In any case, we will need some special terminology for the elements of a particular delooping $\B G$, since we will be using these elements to do all our work with the group $G$.
  \begin{defn}\label{defn:exemplar}
Let $G$ be a higher group and let $\B G$ be a delooping of $G$. We refer to the elements of $\B G$ as \emph{\exemplars} of $G$. We refer to the base point $\pt_{\B G} : \B G$ as the \emph{canonical \exemplar} of $G$ in $\B G$. We may refer to $\B G$ itself as a \emph{type of \exemplars} for $G$, and we note that there may be many different (though equivalent) types of \exemplars for a given group $G$.
\end{defn}

This terminology is best explained through examples. If our (higher) group of interest is the group of symmetries $\Aut_{X}(x) :\equiv (x = x)$ of an object $x : X$, then we can always take $x$ to be a canonical \exemplar and define an \exemplar to be an element of $X$ which is \emph{identifiable} with $x$.
\begin{defn}[Standard]
  Let $x : X$ be an element. Then
  $$\B \Aut_{X}(x) :\equiv \dsum{y : X} \trunc{y = x}$$
  is the type of all elements $y : X$ which are identifiable with $x$. Pointed at $(x, |\refl|)$, this type deloops the automorphism group $\Aut_{X}(x) :\equiv (x = x)$.
  \end{defn}

  We might also deloop a group $G$ by giving a \emph{categorical} definition of an object whose group of symmetries is $G$ --- a definition such that any two instances are identifiable.

\begin{ex}
  Let $\Sigma_{n}$ denote the symmetric group on $n$ elements. We can deloop $\Sigma_{n}$ with the type
  \begin{align*}
\B \Sigma_{n} &:\equiv \B \Aut_{\Set}(\ord{n}) \\
    &\phantom{:}\equiv \dsum{F : \Set} \trunc{F = \ord{n}},
\end{align*}
   since $\Sigma_{n}$ is the group of automorphisms of the standard $n$-element set $\ord{n} :\equiv \{0, \ldots, n-1\}$. Note that we can also see $\B \Sigma_{n}$ as the type of $n$-element sets --- those sets which admit some bijection with the standard $n$-element set. In other words, we define an \exemplar of $\Sigma_{n}$ to be an $n$-element set, and take the canonical \exemplar to be $\ord{n}$.
  \end{ex}

  \begin{ex}
  Let $U(1) :\equiv \{z : \Cb \mid z\bar{z} = 1\}$ be the unit circle in the complex plane, considered as a group under multiplication. We can deloop $U(1)$ with the type $\B U(1)$ of $1$-dimensional Hermitian vector spaces, pointed at $\Cb$. Explicitly, a Hermitian vector space is a vector space $V$ over $\Cb$ equipped with a Hermitian inner product $\langle -,- \rangle : V \times V \to \Cb$ which is linear in the first component, conjugate symmetric, and for which $\langle x, x \rangle > 0$ for non-zero $x$.\footnote{In the setting of this paper --- specifically the axioms of synthetic differential geometry found in \cref{sec:SDG.axioms} --- it is appropriate to ask that $x$ be non-zero. However, in pure homotopy type theory with no classical assumptions, we should ask instead that $x$ be \emph{apart} from $0$, meaning that there is some positive rational $\ep$ with either $x < \ep$ or $x > \ep$.} Any $1$-dimensional Hermitian vector space $\La$ is identifiable with $\Cb$ by some unitary isomorphism --- if $\ell : \La$ gives a basis for $\La$, then the map $1 \mapsto \frac{\ell}{\langle \ell, \ell \rangle}$ gives a unitary isomorphism of $\Cb$ with $\La$.

  In general, $\B U(n)$ may be defined to be the type of Hermitian vector spaces identifiable with $\Cb^{n}$ with its standard inner product. In other words, we define an \exemplar of $U(n)$ to be an $n$-dimensional Hermitian vector space with positive definite inner product and take the canonical \exemplar to be $\Cb^{n}$ with its standard inner product.
    \end{ex}

    \begin{ex}\label{ex:matrix.groups}
      Of course, other matrix groups work in a similar way. We can deloop $\type{GL}_{n}(\Rb)$ with the type $\B \type{GL}_{n}(\Rb)$ of $n$-dimensional real vector spaces, pointed at $\Rb^{n}$. That is, we define an \exemplar of $\type{GL}_{n}(\Rb)$ to be an $n$-dimensional real vector space. Note that $\B \type{GL}_{n}(\Rb)$ classifies real vector bundles of rank $n$ in a truly immediate way: the vector bundle $\pi : E \to B$ is classified by the map $b \mapsto \fib_{\pi}(b) : B \to \B \type{GL}_{n}(\Rb)$ sending every point to the vector space sitting over it in the bundle.

      We could take an \exemplar of $\type{SL}_{n}(\Rb)$ to be an $n$-dimensional real vector space $V$ equipped with a non-zero\footnote{Again, without classical assumptions this must instead mean ``apart from zero''.} element of the exterior power $\Lambda^{n}V$ --- or equivalently a non-trivial alternating $n$-form on $V$. The canonical \exemplar is $\Rb^{n}$ equipped with the element $e_{1} \wedge \cdots \wedge e_{n}$.

      We could take an \exemplar of $\type{O}(n)$ to be an $n$-dimensional real vector space equipped with an inner product. The canonical \exemplar is $\Rb^{n}$ with its standard inner product.

      We could take an \exemplar of the symplectic group $\type{Sp}(2n, \Rb)$ to be a $2n$-dimensional real vector space equipped with a non-degenerate alternating $2$-form. The canonical \exemplar is $\Rb^{2n}$ equipped with its standard symplectic form
      \[
\omega(v, w) :\equiv \sum_{i = 1}^{n}x^{i}(v)y^{i}(w) - y^{i}(v)x^{i}(w)
      \]
      where $\{x_{1}, \ldots, x_{n}, y_{1}, \ldots, y_{n}\}$ is the standard basis of $\Rb^{2n}$  and $x^{i}$ and $y^{i}$ are the associated conjugate basis of $(\Rb^{2n})^{\ast}$.

    \end{ex}

    \begin{ex}
      If $V$ is a real vector space considered as an addititive group, then we can take $\B V$ to be the type of affine spaces whose difference vectors land in $V$. This is not so different than defining $\B V$ to be the type of $V$-torsors. In other words, we define an \exemplar of $V$ to be an affine space over $V$, with the canonical \exemplar being $V$ itself.

      If we want to deloop the full affine group of $\Rb^{n}$, we can take an \exemplar to be a pair consisting of an $n$-dimensional real vector space $V$ and an affine space over it. The canonical \exemplar is $\Rb^{n}$ paired with itself. That is,
      $$\B \type{Affine}(\Rb^{n}) :\equiv \dsum{ V : \B \type{GL}_{n}(\Rb^{n}) } \B V.$$
    \end{ex}

    No matter what we use to deloop our group $G$, there is always an equivalence $\B G \xto{\sim} \type{Tors}_{G}$ associating a $G$-torsor to any \exemplar of $G$.
    \begin{prop}[\cite{SymmetryBook}]\label{prop:associated.torsor}
      Let $\B G$ be a delooping of a $1$-group $G$. Then for any \exemplar $t : \B G$, the type $(t = \pt_{\B G})$ of identifications of $t$ with the canonical \exemplar is a $G$-torsor, and the function
      \[
t \mapsto (t = \pt_{\B G}) : \B G \to \type{Tors}_{G}
      \]
      sending the \exemplar $t$ to its \emph{associated torsor} $(t = \pt_{\B G})$ is an equivalence.
      \end{prop}

      \begin{rmk}
        A similar theorem would work for $n$-groups for $n > 1$ so long as an appropriate notion of torsor could be defined. For example, a $2$-group is equivalently a monoidal groupoid where for every $g : G$, $g \otimes -$ is an equivalence. We could then define a torsor for a $2$-group as a action $\alpha : G \times X \to X$ of this monoidal groupoid (an ``actegory'') on an inhabited groupoid $X$ for which the map $\alpha(-, x) : G \to X$ is an equivalence for every $x : X$. However, a careful proof of this would require a good deal of work.

        Until a suitable theory of simplicial types can be developed in HoTT (a famous open problem), we will likely not be able to give a general theorem along the lines of \cref{prop:associated.torsor} for general higher groups.
    \end{rmk}

      \begin{rmk}\label{rmk:frame.bundle}
        In the case that $G \equiv  \type{GL}_{n}(\Rb)$, then the associated torsor of an $n$-dimensional vector space $V : \B \type{GL}_{n}(\Rb)$ is the type $V = \Rb^{n}$ of linear isomorphisms of $V$ with $\Rb^{n}$. A linear isomorphism with $\Rb^{n}$ is the same thing as a basis of $V$ --- a \emph{frame} --- since $V$ was assumed to be $n$-dimensional. Therefore, we see that the torsor associated to $V$ is its space $\type{Frame}(V)$ of frames.

        If $E : B \to \B \type{GL}_{n}(\Rb)$ classifies a vector bundle, then the composite $B \xto{E} \B \type{GL}_{n}(\Rb) \xto{\sim} \type{Tors}_{G}$ classifies the frame bundle of that vector bundle.
    \end{rmk}

      As a sanity check, note that the associated torsor of a torsor is itself.
      \begin{lem}\label{lem:torsor.associated.torsor.itself}
The associated torsor of a $G$-torsor $T$ is $T$ itself. Explicitly, $T$ is equivalent as a $G$-torsor (and therefore also as a type) to the torsor $(G =_{\type{Tors}_{G}} T)$ of its identifications with $G$ as a $G$-torsor.
        \end{lem}
        \begin{proof}
An identification of $G$ with $T$ is determined by the image of $1 : G$ by equivariance.
          \end{proof}

          \subsection{Homomorphisms and actions} \label{sec:group.actions}

    We can define homomorphisms between higher groups using just their deloopings.
    \begin{defn}[\cite{Buchholtz-vanDoorn-Rijke:Higher.Groups}]
    A homomorphism $\varphi : G \to H$ between higher groups is a pointed map $\B \varphi : \B G \pto \B H$ between their deloopings.

    In other words, a homomorphism $G \to H$ is a function assigning \exemplars of $G$ to \exemplars of $H$, together with an identification of the image of the canonical \exemplar of $G$ with that of $H$. The map $\varphi : G \to H$ itself is defined to be
\[
    \Omega \B \varphi(g) :\equiv \pt_{\B \varphi}\inv \cdot \,(B \varphi)_{\ast} g \cdot \pt_{\B \varphi}.
\]
    \end{defn}

  We can always deloop a homomorphism between ordinary groups into a map between their types of torsors by tensoring up along the homomorphism.
  \begin{prop}[\cite{SymmetryBook}]
    Let $G$ and $H$ be $1$-groups, and let $\varphi : G \to H$ be a homomorphism. Then the map
    $$\B \varphi :\equiv T \mapsto H \otimes_{G} T : \type{Tors}_{G} \to \type{Tors}_{H}$$
    given by sending a $G$-torsor $T$ to the $H$-torsor $H \otimes_{G} T$ defined by
    \[
      H \otimes_{G} T :\equiv \frac{H \times T}{(h \varphi(g), t) \sim (h, gt)}    \]
    deloops the homomorphism $\varphi$ when it is pointed at the equivalence
    $$\pt_{\B \varphi} : H \otimes_{G} G = H$$
    given by $h \otimes g \mapsto h \varphi(g)$.
\end{prop}

\begin{ex}
  If $\varphi : G \to H$ is a \emph{surjective} homomorphism, then we may deloop its kernel by defining $\B \ker \varphi :\equiv \fib_{\B\varphi}(\pt_{\B H})$ to be the fiber of any delooping $\B \varphi : \B G \pto \B H$. We need surjectivity for the fiber of $\B \varphi$ to be $0$-connected; this is an if-and-only-if, since $\B \varphi$ is $0$-connected if and only if its fiber is and if and only if its delooping $\varphi$ is surjective (which is by definition means $-1$-connected).

  For example, we can reconstruct the delooping of $\type{SL}_{n}(\Rb)$ given in \cref{ex:matrix.groups} by seeing $\type{SL}_{n}(\Rb)$ as the kernel of the determinant $\det : \type{GL}_{n}(\Rb) \to \type{GL}_{1}(\Rb)$. The determinant may be delooped by the function $V \mapsto \Lambda^{n}V : \B \type{GL}_{n}(\Rb) \to \B \type{GL}_{1}(\Rb)$ , pointed at the identification $\Lambda^{n} \Rb^{n} = \Rb$ induced by the standard basis element $e_{1} \wedge \cdots \wedge e_{n} : \Lambda^{n}\Rb^{n}$ (where $e_{i}$ are the standard basis vectors of $\Rb^{n}$). Therefore, the kernel of $\det$ may be delooped by the fiber of the function $\Lambda^{n} : \B \type{GL}_{n}(\Rb) \to \B \type{GL}_{1}(\Rb)$, which is the type of $n$-dimensional vector spaces $V$ equipped with a linear isomorphism $\Lambda^{n}V = \Rb$.
    \end{ex}

We may describe an action of a group $G$ on an object $x : X$ as a homomorphism $\alpha : G \to \Aut_{X}(x)$, which is the same as a pointed map $\B \alpha : \B G \pto \B \Aut_{X}(x)$. In other words, we can see an action of the group $G$ on an object $x$ as a way of taking \exemplars $t : \B G$ of $G$ to objects $\B \alpha(t)$ identifiable with $x$, together with an identification $\pt_{\B \alpha} : \B \alpha(\pt_{\B G}) = x$ of the image of the canonical \exemplar with $x$ itself.

\begin{defn}
  An action of a higher group $G$ on a object $x : X$ is a pointed map $x\twisted{(-)} : \B G \pto \B \Aut_{X}(x)$.  An action $X\twisted{(-)} : \B G \to \B \Aut_{X}(x)$ takes an \exemplar $t : \B G$ of $G$ to the object $x\twisted{t} : X$ which is identifiable with $x$; we say that $x\twisted{t}$ is $x$ \emph{twisted} by $t$.
\end{defn}

  The action of a higher group on a type is itself given by transport in the type family.
  \begin{lem}\label{lem:group.action.definition}
    Let $X\twisted{(-)} : \B G \to \Type$ be an action of a higher group $G$ on a type $X :\equiv X\twisted{\pt_{\B G}}$. Then for $g : G$ and $x : X$, we have
    $$gx = \tr(X\twisted{(-)}, g)(x).$$
  \end{lem}
  \begin{proof}
    The action $gx$ is by definition given by $(X\twisted{(-)})_{\ast}g (x)$, so the desired identification follows from the fact that transporting over an identification in a type family is the same as applying the type family to the identification.
  \end{proof}

We can always deloop an action of a group $G$ on a set $X$ by twisting the action with a torsor.
\begin{prop}[\cite{SymmetryBook}]\label{prop:delooping.group.action.torsors}
  Let $\alpha : G  \to \Aut(X)$ be an action of a group $G$ on a set $X$. Then the map
  \[
T \mapsto T \otimes_{G} X : \type{Tors}_{G} \to \Set
\]
sending a $G$-torsor $T$ to the tensor product $T \otimes_{G} X$ defined by
\[
T \otimes_{G} X :\equiv \frac{T \times X}{(t, gx) \sim (g\inv t, x)}
\]
deloops the action $\alpha$ when pointed at the identification
\[
\pt_{\B \alpha} : G \otimes_{G} X = X
\]
given by $g \otimes x \mapsto g\inv x$.
  \end{prop}
  \begin{rmk}
The appearance of the inverses in \cref{prop:delooping.group.action.torsors} is due to our choice to use \emph{left} $G$-torsors and \emph{left} actions. If we used \emph{right} $G$-torsors and left actions then the tensor product formulas would not need any inverses.
\end{rmk}

A \emph{representation} of a group $G$ is therefore a pointed map $\B G \pto \B \type{GL}_{n}(\Rb)$; that is, it is a way of turning \exemplars of $G$ into $n$-dimensional vectors spaces in such a way that the canonical \exemplar gets turned into $\Rb^{n}$.

\begin{ex}
  A representation of the symmetric group $\Sigma_{k}$ is a function $\B \Sigma_{k} \to \type{Vect}_{\Rb}$ which sends a $k$-element set to a vector space.

  For example, we have the canonical representation of $\Sigma_{k}$ on $\Rb^{k}$ which is given by the function $\B \rho \equiv X \mapsto \Rb^{X}$ sending a $k$-element set $X$ to the vector space $\Rb^{X}$ which is free on it, together with the identification $\pt_{\B \rho} : \Rb^{\ord{k}} = \Rb^{k}$ which we may as well take as definitional.
\end{ex}

\begin{ex}\label{ex:cyclic.action.on.plane}
  Since actions and representations of groups are given as functions of their \exemplars, it is sometimes useful to come up with a new type of \exemplars for the group in order to easily define an action. For example, suppose we are trying to define the action of the cyclic group $C_{n}$ on the plane by rotation (for $n \geq 1$). Which delooping should we use?

First, we can think of $C_{n}$ as the group $\mu_{n}$ of $n^{\text{th}}$ roots of unity acting on the complex plane $\Cb$ by multiplication. We can imagine an \exemplar of $\mu_{n}$ as a set of $n$ points equi-distantly arranged on a circle in a $1$-dimensional complex vector space. So, we can start with a $1$-dimensional Hermitian vector space $\La$, and equip this with $n$-element subset of its unit circle $\Sb^{\La}$ of equidistantly placed points. It is somewhat difficult to say that the points are equidistantly placed without an ordering on them; instead, we can say that the angle between any two of them evenly divides the circle, and that rotating by an $n$-th root of unity keeps us within $C$.

  \begin{defn}
    Let $\La$ be a $1$-dimensional Hermitian vector space. A \emph{cycle of $n$ elements} in $\La$ is a subset $C \subseteq \Sb^{\La}$ of its unit circle such that for any $x,\, y \in C$, we have $\langle x, y \rangle^{n} = 1$ and for any $n^{{\text{th}}}$-root of unity $\zeta \in \mu_{n}$ and $x \in C$, the rotation $\zeta x \in C$ is also in $C$.
    \end{defn}

    We will take a pair $(\La, C)$ of a $1$-dimensional Hermitian vector space and a cycle of $n$ elements in it as an \exemplar of $\mu_{n}$. That is, we define
    \[
\B \mu_{n} :\equiv \dsum{\La : \B U(1)} \type{Cycle}_{n}(\La).
    \]

    As a canonical \exemplar, we take the subset $\mu_{n} \subseteq U(1) \subseteq \Cb$ of $n^{{\text{th}}}$ roots of unity. It remains to show that $\B \mu_{n}$ is $0$-connected, and that it deloops $\mu_{n}$.

    Suppose that $C \subseteq \La$ is a cycle of $n$-elements, seeking to show that this is identifiable with $\mu_{n} \subseteq \Cb$. We can identify $\La$ with $\Cb$ via a unitary transformation $U : \La = \Cb$, which gives us a cycle $U_{\ast}C$ of $n$-elements in $\Cb$ (since $U$ is unitary). Now since $n \geq 1$, there is some element in $U_{\ast}C$, say $Ux \in U_{\ast} C$. Then the unitary transformation $(U x)\inv U : \La = \Cb$ sends $C$ to $\mu_{n}$. Consider $(Ux)\inv Uy \in (Ux)\inv U_{\ast} C$; then $((Ux)\inv Uy)^{n}= \langle Uy, Ux \rangle^{n} = \langle x, y\rangle^{n} = 1$, so that $(Ux)\inv Uy$ is in $\mu_{n}$. Conversly, if $\zeta\in \mu_{n}$, then we have $\zeta x \in C$ and so $\zeta = (Ux)\inv U(\zeta x) \in (Ux)\inv U_{\ast} C$.

Finally, $\B \mu_{n}$ does actually deloop $\mu_{n}$ since an unitary automorphism of $\Cb$ which fixes $\mu_{n}$ setwise is given by multiplication by an element of $\mu_{n}$.

  Now, to define the representation of $C_{n}$ on the plane is easy; we can send an \exemplar $(\La, X) : \B \mu_{n}$ to $\La$ considered as a $2$-dimensional real vector space. We point this operation at the canonical identification $\Cb = \Rb^{2}$ given by $1 \mapsto e_{1}$ and $i \mapsto e_{2}$.

  \end{ex}

  Note that in this example, we tailored the delooping of $C_{n}$ to the task of constructing its action on the plane by considering a delooping of $U(1)$ which easily described its action on the plane, and then restricting this delooping to $C_{n}$ by equipping the exemplars of $U(1)$ --- the complex lines $\La$ --- with cycles of $n$-elements.  Another way to describe a cycle of $n$ elements in a complex line $\La$ is as a $\mu_{n}$-torsor which is a subaction of the $\mu_{n}$ action on $\La$ by scalar multiplication. This recipe gives us a general way to restrict a delooping of a group to a subgroup.

  \begin{defn}
Let $G$ be a group, $X$ a $G$-action, and $\Gamma$ a subgroup of $G$. A $\Gamma$-subtorsor of $X$ is a $\Gamma$-subaction of the $G$-action of $X$ restricted to $\Gamma$ which is a $\Gamma$-torsor in its own right --- that is, it is free, transitive, and inhabited as a $\Gamma$-action. We denote the type of $\Gamma$-subtorsors of $X$ by $\type{Subtors}_{\Gamma}(X)$.
    \end{defn}

    \begin{lem}\label{lem:subtorsor.lemma}
Let $g : G$ and let $X$ be a $G$-action. Then $\tr(\type{Subtors}_{\Gamma}(X\twisted{-}), g) : \type{Subtors}_{\Gamma}(X) \to \type{Subtors}_{\Gamma}(X)$ sends $T$ to $g\inv T :\equiv \{x : X \mid \exists t : T.\, x = g\inv t\}$.
    \end{lem}
    \begin{proof}
We note that a $\Gamma$-subtorsor $T$ of $X$ is in particular a subset $T \subseteq X$ of $X$ which satisfies a property. For this reason, we only need to think about how transporting subsets works: $\tr(X\twisted{-} \to \Prop, g) : \type{Subset}(X) \to \type{Subset}(X)$ is given by taking the inverse image under $\tr(X\twisted, g) : X \to X$, which is the action by $g$. This is because subsets are equivalently described by the property of being in that subset, $\type{Subset}(X) = (X \to \Prop)$, and transport in the latter is given by precomposition. Therefore, $\tr(X\twisted{-} \to \Prop, g)(T) = \{x : X \mid gx \in T\} = g\inv T$.
    \end{proof}

 \begin{prop}\label{prop:subtorsor.delooping}
   Let $G$ be a group and $\Gamma$ a subgroup. Suppose that $\B G$ is a delooping of $G$ giving us a type of \exemplars for $G$. We may then define an \exemplar of $\Gamma$ to be an \exemplar of $G$ equipped with a $\Gamma$-subtorsor of its associated torsor:
   $$\B \Gamma :\equiv \dsum{t : \B G} \type{Subtors}_{\Gamma}(t = \pt_{\B G}).$$
We take the canonical \exemplar to be the canonical \exemplar $\pt_{\B G}$ of $G$ paired with $\Gamma$ considered as a $\Gamma$-subtorsor of the assocaited $G$-torsor $(\pt_{\B G} = \pt_{\B G})$, identified with $G$ acting on itself.

   Furthermore, the inclusion $\Gamma \hookrightarrow G$ is delooped by the first projection $\fst : \B \Gamma \to \B G$.
  \end{prop}
  \begin{proof}
An identification $(\pt_{\B G}, \Gamma) = (\pt_{\B G}, \Gamma)$ is an identification $g : \pt_{\B G} = \pt_{\B G}$ so that $g\inv \Gamma = \Gamma$ by \cref{lem:subtorsor.lemma}. However, $\Gamma$ contains $1$, so we may conclude that $g\inv \in \Gamma$ and therefore $g \in \Gamma$. On the other hand, any $\gamma : \Gamma$ clearly sends elements in $\Gamma$ to elements in $\Gamma$, so we have a bijection between self identifications of the canonical \exemplar of $\B \Gamma$ with $\Gamma$.
\end{proof}

\subsection{Quotients as types of pairs}\label{sec:group.quotients}

Now we are ready to define the quotient of a type by the action of a higher group. A beautiful feature of homotopy type theory is that the quotient has a \emph{mapping-in} property, in addition to its usual mapping out property defining it as a quotient. What this means in practice is that we can define quotients by their elements, without forcing any equivalence relations or freely generating any structure. In fact, the construction couldn't be simpler: the quotient $X \sslash G$ of the action $X$ of a higher group $G$ is the type of pairs $(t, x)$ of an \exemplar $t : \B G$ and an element $x : X\twisted{t}$ of $X$ twisted by $t$.

\begin{defn}[\cite{SymmetryBook}]\label{defn:quotient}
  Given a action $X\twisted{(-)} : \B G \to \Type$ of the higher group $G$ on the type $X :\equiv X\twisted{\pt_{\B G}}$, define the quotient
  \[
    X \sslash G :\equiv \dsum{t : \B G} X\twisted{t}
  \]
  to be the type of pairs of an \exemplar $t : \B G$ and an element of $X\twisted{t}$. The quotient map
  \[
[-] : X \to X \sslash G
\]
is given by pairing with the canonical \exemplar: $x \mapsto (\pt_{\B G}, x)$.
\end{defn}

This definition is justified by an elementary lemma about identifications in types of pairs. As a corollary, we may deduce that the symmetries of an element in the quotient are precisely its stabilizer.
\begin{lem}\label{lem:identifications.in.homotopy.quotient}
  Let $X\twisted{(-)} : \B G \to \Type$ be an action of a higher group $G$ on a type $X$. For $x, y : X$, we have an equivalence
  \[
([x] = [y]) \simeq \dsum{g : G} (g x  = y)
\]
\end{lem}
\begin{proof}
  This follows immediately from \cref{lem:group.action.definition} and Theorem 2.7.2 of the HoTT Book \cite{HoTTBook} which characterizes identifications in pair types:
  \[
((\pt_{\B G}, x) = (\pt_{\B G}, y)) \simeq \dsum{g : \pt_{\B G} = \pt_{\B G}} (\tr(X\twisted{(-)}, g)(x) = y).
  \]
  \end{proof}

  \begin{cor}
    Let $X\twisted{(-)} : \B G \to \Type$ be an action of a higher group on a type $X$. For $x : X$, the self-identifications of $x$ in the quotient is the stabilizer of $x$:
    \[
\Aut_{X \sslash G}([x]) \simeq \dsum{g : G} (gx = x) \equiv \type{Stab}(x).
    \]
  \end{cor}

  \begin{rmk}\label{rmk:quotient}
It is worth emphasizing that this construction of $X \sslash G$ given in \cref{defn:quotient} (which is  entirely standard in homotopy type theory) constructs what is usually known as the ``homotopy'' quotient, but it constructs it \emph{on the nose}, and not ``up to homotopy''. The terminology here really gets in the way, but the point is that it is \emph{up to identification} --- which is the only sort of equality in HoTT --- and not up to continuous deformation that the quotient $X \sslash G$ behaves like it is supposed to.
\end{rmk}

  We can also characterize the homotopy quotient maps $q : X \to X \sslash G$ quite simply, at least when $G$ is a $1$-group: they are precisely those maps whose fibers are $G$-torsors. Or, in other words, the homotopy quotients by $G$ are precisely the $G$-principal bundles.
  \begin{thm}\label{thm:recognize.homotopy.quotient}
Let $G$ be a group $q : X \to Y$ be a map. Suppose that for any $y : Y$ we have a $G$-torsor structure on $\fib_{q}(y)$. Then $G$ acts on $X$ and $q$ is equivalent to the homotopy quotient $X \to X \sslash G$ by this action.
    \end{thm}
    \begin{proof}
      We will construct an action $\alpha$ of $G$ on $X$ by constructing its delooping $\B \alpha : \type{Tors}_{G} \pto \BAut(X)$. We define
      $$\B \alpha(T) :\equiv \dsum{y : Y} (T = \fib_{q}(y))$$
      where the identification is taken as $G$-torsors. Then
      \begin{align*}
        \B\alpha(G) &\equiv \dsum{y : Y} (G = \fib_{q}(y))\\ &\simeq \dsum{y : Y} \fib_{q}(Y) \\
        &\simeq X.
        \end{align*}
      gives us a pointing of $\B \alpha$. Furthermore, since $\type{Tors}_{G}$ is $0$-connected, this shows that $\B\alpha(T)$ is identifiable with $X$ for all $T : \type{Tors}_{G}$, so that $B\alpha$ really does land in $\B \Aut(X) \equiv \dsum{ Z : \Type }\trunc{Z = X}$. The middle equivalence follows from \cref{lem:torsor.associated.torsor.itself}, and the last equivalence is a general fact about any map --- it is always the sum of its fibers. Therefore, this equivalence identifies $x : X$ with $(q(x), (1 \mapsto (x, \refl)))$ where $(1 \mapsto (x, \refl)) : G = \fib_{q}(q(x))$ is the identification of $G$-torsors determined by sending $1$ to $(x, \refl)$ and the rest by $G$-equivariance.

      Now, we will give an equivalence $Y = X \sslash G$ which commutes with the quotient maps. This follows quickly by the substitution lemma:
      \begin{align*}
        X \sslash G &\equiv \dsum{T : \type{Tors}_{G}} \dsum{y : Y} (T = \fib_{q}(y)) \\
       &\simeq Y.
        \end{align*}
      Explicitly, this equivalence $Y = X \sslash G$ is given by $y \mapsto (\fib_{q}(y), y, \refl)$. It remains to show that for $x : X$, we have $(\fib_{q}(q(x)), q(x), \refl) = (G, q(x), (1 \mapsto (x, \refl)))$. Since we may take $(1 \mapsto (x, \refl))\inv : \fib_{q}(q(x)) = G$ and transporting by this sends $(1 \mapsto (x, \refl))$ to $\refl : G = G$, we have our desired identification.
    \end{proof}

\subsection{Good orbifolds}\label{sec:orbifold.examples}

Finally, we are ready to construct some orbifolds. In this section, we will focus on \emph{good} orbifolds --- those orbifolds which are the quotients of discrete groups acting on manifolds. The easy and concrete construction of quotients in homotopy type theory makes constructing good orbifolds a breeze.

\begin{rmk}
We will eventually be able to define \emph{{\'e}tale groupoids} (\cref{defn:etale.groupoid}), a notion which includes the presentations of orbifolds as proper {\'e}tale groupoids due to Moerdijk and Pronk \cite{Moerdijk-Pronk:Orbifolds}. We will, in \cref{thm:ordinary.proper.etale.groupoid.is.orbifold}, prove that (crisp, ordinary) proper \'etale groupoids are orbifolds in the sense of \cref{defn:orbifold}. But first we will need to develop the language of synthetic differential geometry.
  \end{rmk}

\begin{ex}\label{ex:elliptic.point}
  We can define an \emph{elliptic point} of order $n$ as $\Rb^{2} \sslash C_{n}$, with the action of $C_{n}$ on $\Rb^{2}$ constructed as in \cref{ex:cyclic.action.on.plane}. Explicitly, this means
  \[
\Rb^{2} \sslash C_{n} :\equiv \dsum{\La : \B U(1)} \dsum{C : \type{Cycle}_{n} (\La)}\La
\]
consists of triples $(\La, C, \ell)$ where $\La$ is a $1$-dimensional Hermitian vector space, $C \subseteq \Sb^{\La}$ is a cycle of $n$-elements in $\La$, and $\ell : \La$ is a point of $\La$.
\end{ex}

\begin{ex}\label{ex:coordinate.patch}
  Satake \cite{Satake:Orbifold} defined orbifolds as spaces locally modelled on a quotient of $\Rb^{n}$ by the action of a finite subgroup of $\type{O}(n)$. For any finite subgroup $\Gamma \subseteq \type{O}(n)$, we can construct the coordinate patch $\Rb^{n} \sslash \Gamma$ by
  \[
    \Rb^{n} \sslash \Gamma :\equiv \dsum{V : \B O(n)} \type{Subtors}_{\Gamma}(\type{Frame}(V)) \times V.
  \]
  We're making use of \cref{prop:subtorsor.delooping} and \cref{rmk:frame.bundle} to deloop $\Gamma$ as $\B \Gamma :\equiv \dsum{V : \B O(n)} \type{Subtors}_{\Gamma}(\type{Frame}(V))$.
\end{ex}

    \begin{ex}\label{defn:configuration.space}
      We can describe the configuration space $X \sslash n!$ of $n$ unlabeled points in a given type $X$ quite simply as a homotopy quotient. We may take $n$-element sets as our \exemplars of the symmetric group $\Aut(\ord{n})$, with the standard finite cardinal $\ord{n} :\equiv \{0, \ldots, n-1\}$ as our canonical \exemplar. That is, we define
      $$\B\Aut(n) :\equiv \dsum{F : \Type} \trunc{F \simeq \ord{n}}$$
      to be the type of $n$-element sets, which are equivalently types $F$ which are somehow identifiable with $n$. We may then act on the cartesian power $X^{n}$ by sending an $n$-element set $F$ to the type of functions $X^{F}$. This gives us the following construction of the configuration space of $n$ unlabled points as the type of pairs of an $n$-element set $F$ and an $F$-tuple of elements of $X$:
      $$X^{n} \sslash n! :\equiv \dsum{F : \B\Aut(\ord{n})} X^{F}.$$
    \end{ex}

  \begin{ex}
    We can describe the moduli stack $\Ma_{1,1}$ of elliptic curves over $\Cb$ as the homotopy quotient of the type $\type{Lattice}(\Cb)$ of lattices in $\Cb$ by the action of $\Cb^{\ast}$.

    To this end, we will define the notion of a lattice in a complex line. In fact, we might as well define the notion of lattice in $n$-dimensional (real) space.
    \begin{defn}\label{defn:M11}
      Let $V : \B \type{GL}_{n}(\Rb)$ be an $n$-dimensional real vector space. A \emph{lattice} in $V$ is a subset $\Lambda \subseteq V$ which is
      \begin{enumerate}
        \item an additive subgroup of $V$,
        \item metrically discrete, in that for any norm $\langle - ,-\rangle$ on $V$ there exists a (rational) $\ep > 0$ so that if $x \in \Lambda$ has norm $\langle x,\, x \rangle$ less than $\ep$, then $x = 0$.
        \item non-degenerate, in that it has rank $n$ as an abelian group.
      \end{enumerate}
    We denote by $\type{Lattice}(V)$ the type of lattices in $V$. To consider a lattice in a complex vector space, we first consider that vector space as a real vector space.

    We can then define $\Ma_{1,1}$ as the type of pairs consisting of a $1$-dimensional complex vector space and a lattice in it.
    \begin{defn}
We define $\Ma_{1,1}$ to be the type of pairs consisting of a $1$-dimensional complex vector space $L$, and a lattice $\Lambda$ within it.
    \[
\Ma_{1,1} :\equiv \type{Lattice}(\Cb) \sslash \type{GL}_{1}(\Cb) \equiv \dsum{\La : \B \type{GL}_{1}(\Cb)}\type{Lattice}(\La).
\]
    \end{defn}
    \end{defn}

It is not clear from this description that $\Ma_{1,1}$ is an orbifold, however. In order to do that, we will use the recognition theorem, \cref{thm:recognize.homotopy.quotient}, to show that $\Ma_{1,1}$ is the homotopy quotient of the upper half plane $\mathfrak{h} :\equiv \{a + b i : \Cb \mid b > 0\}$ by the action of $\type{SL}_{2}(\Zb)$ via M{\:o}bius transformations.

We may define a map $q : \mathfrak{h} \to \Ma_{1,1}$ by $q(\tau) \equiv (\Cb, \Zb \oplus \tau \Zb)$. We will show that the fibers of $q$ may be equipped with the structure of a $\type{SL}_{2}(\Zb)$-torsor; by \cref{thm:recognize.homotopy.quotient}, this will show that $q$ is a homotopy quotient and that $\Ma_{1,1} = \mathfrak{h} \sslash \type{SL}_{2}(\Zb)$.

\begin{prop}\label{prop:M11.homotopy.quotient}
  Let $q : \mathfrak{h} \to \Ma_{1,1}$ be the map
  $$q(\tau) \equiv (\Cb, \latticegen{\tau}).$$

  Then every fiber $\fib_{q}(\La, \Lambda)$ may be equipped with the structure of a $\type{SL}_{2}(\Zb)$-torsor with the action given by M{\"o}bius transformations. Consequently,
  \[
\Ma_{{1,1}} \simeq \mathfrak{h} \sslash \type{SL}_{2}(\Zb).
  \]
  \end{prop}
  \begin{proof}
   Let $\La$ be a $1$-dimensional complex vector space and $\Lambda \subseteq \La$ a lattice in it. By definition,
    $$\fib_{q}(\La, \Lambda) :\equiv \dsum{\tau : \mathfrak{h}} \big((\Cb, \latticegen{\tau}) = (\La, \Lambda)\big).$$
    By the calculation of identifications in pair types, this is equivalently
    $$\dsum{\tau : \mathfrak{h}} \dsum{p : \Cb = \La} (\latticegen{\tau} = p\inv(\Lambda)).$$
    Note that since lattices are subsets, equality between lattices is a proposition. Therefore, we are free to consider the fiber as a type of pairs $(\tau, p)$ which satisfy a proposition. We will describe a $\type{SL}_{2}(\Zb)$ action on this type and then prove that it is free and transitive.

Given a matrix $\begin{bmatrix}a & b \\ c & d  \end{bmatrix} : \type{SL}_{2}(\Zb)$, its associated M{\"o}bius transform is the function $f(z) = \frac{az + b}{cz + d}$ which acts on the upper half plane $\mathfrak{h}$. We define the action of $\type{SL}_{2}(\Zb)$ on $\fib_{q}(\La, \Lambda)$ by
\[
  \begin{bmatrix}
    a & b \\ c & d
  \end{bmatrix}
  (\tau, p) :\equiv \left(\frac{a \tau + b}{c \tau + d},\, (c \tau + d)p\right).
\]
We check that
\begin{align*}
  \latticegen{\left(\frac{a\tau + b}{c \tau + d}\right)} &= \frac{1}{c\tau + d}\left( \Zb(c \tau + d) \oplus \Zb(a\tau + b) \right) \\
                                                         &= \frac{1}{c \tau + d}\left(\latticegen{\tau}\right) \\
                                                         &=\frac{1}{c \tau + d} p\inv(\Lambda) \\
  &= ((c \tau + d) p)\inv(\Lambda).
  \end{align*}

  We also check that this is an action, which is to say that
  \[
  \begin{bmatrix}
    x & y \\ u & v
  \end{bmatrix}
  \begin{bmatrix}
    a & b \\ c & d
  \end{bmatrix}
  (\tau, p)  =
  \begin{bmatrix}
    xa +yc & xb +yd \\ ua + vc & ub + vd
  \end{bmatrix}
  (\tau, p) .
  \]
  That is, we need that
  \[
\left(\frac{x\left(\frac{a\tau + b}{c \tau + d}\right) + y}{u\left(\frac{a\tau + b}{c \tau + d}\right) + v}, \left(u\left(\frac{a\tau + b}{c \tau + d}\right) +v\right)(c\tau + d)p\right) = \left(\frac{(xa + yc)\tau + (xb + yd)}{(ua + vc)\tau + (ub + vd)}, ((ua + vc)\tau + (ub + vd))p\right)
  \]
  which amounts to a bit of algebra. Next, we show that this action is free and transitive. If $(\tau, p)$ and $(\sigma, q)$ are in the fiber, then we have that
  $$\latticegen{\tau} = p\inv(\Lambda) = p\inv(q(\latticegen{\sigma})) = z(\latticegen{\sigma})$$
  where $z \equiv p\inv q(1)$ is a non-zero complex scalar. This tells us that $z$ and $z\sigma$ are in $\latticegen{\tau}$, or, in other words, we have that
  \begin{align*}
    z &= a\tau + b &\mbox{for $a, b : \Zb$, and } \\
    z\sigma &= c \tau + d &\mbox{for $c, d : \Zb$, so that} \\
    \sigma &= \frac{a\tau + b}{c \tau + d}.
  \end{align*}
  We also know that $z\sigma$ and $z$ generate $\latticegen{\tau}$ (as an ordered basis), since multiplication by $z$ is an isomorphism of abelian groups. Therefore, the matrix
  $\begin{bmatrix} a & b \\ c & d \end{bmatrix}$
  must have determinant $1$, since it transforms one ordered basis of this rank 2 free abelian group into another.
    \end{proof}
    \end{ex}

  \begin{ex}
    Let $\Lambda$ be a lattice in a $1$-dimensional complex vector space $V$ --- that is, let $(V, \Lambda) : \Ma_{{1,1}}$. The pillowcase orbifold $\mathfrak{P}(\Lambda)$ associated to $\Lambda$ is $(V/\Lambda) \sslash \type{O}(1)$ with the action of $\type{O}(1) = \{-1, 1\}$ on the torus $V/\Lambda$ given by $[v] \mapsto [-v]$. We can describe the action of $O(1)$ on $V/\Lambda$ by acting on $V$ via the map $L \mapsto L \otimes_{\Rb} V : \B O(1) \to \B \Aut_{\B \type{GL}_{1}(\Ca)}(V)$, where the complex action on $L \otimes_{\Rb} V$ is induced by $i(\ell \otimes v) :\equiv \ell \otimes iv$. Therefore, we can construct the pillowcase orbifold as type of pairs consisting of a $1$-dimensional real inner product space $L$ and an element of the torus $((L \otimes_{\Rb} V))/(L \otimes_{\Rb} \Lambda)$:
    \[
\mathfrak{P}(\Lambda) :\equiv \dsum{L : \B O(1)} ((L \otimes_{\Rb} V)/(L \otimes_{\Rb} \Lambda)).
    \]
  \end{ex}

  The next few examples are described as quotients in Section 1.6 of \cite{Adem-Leida-Ruan:Orbifolds.and.Stringy.Topology}.

\begin{ex}\label{ex:kummer}
  The Kummer surface $K$ is the quotient $\Tb^{4} \sslash \type{Gal}(\Ca:\Rb)$ of a $4$-torus $\Tb^{4} :\equiv (U(1))^{4}$ with the action of $\type{Gal}(\Ca:\Rb) :\equiv \{1, \sigma\}$ given by complex conjugation: $\sigma(z_{1}, z_{2}, z_{3}, z_{4}) :\equiv (\bar{z}_{1},\bar{z}_{2},\bar{z}_{3},\bar{z}_{4})$.

  To describe the action of $\type{Gal}(\Cb:\Rb)$ on $\Tb^{4}$, we should first describe its action on $U(1)$, since the action on $\Tb^{4}$ is diagonal. To do this smoothly, we should choose a judicious delooping of $\type{Gal}(\Cb:\Rb)$. We can define an exemplar of $\type{Gal}(\Cb:\Rb)$ to be an algebraic closure of $\Rb$ --- or, to be a bit safer, a degree $2$ algebraic extension of $\Rb$. The canonical \exemplar is of course taken to be $\Cb$. If $\Kb$ is any degree $2$ algebraic extension of $\Rb$, then its Galois group $\type{Gal}(\Kb:\Rb)$ has at most two elements, one of which must be the identity. Call the other element $\sigma$; for $\Kb \equiv \Cb$ this is of course complex conjugation. We may then define
   \[
  \Sb^{1}(\Kb) :\equiv \left\{z : \Kb \,\middle|\, z \sigma(z) = 1\right\}
\]
and then $\Tb^{4}(\Kb) :\equiv \Sb^{1}(\Kb)^{4}$. This gives us the desired action of $\type{Gal}(\Cb:\Rb)$ on $\Tb^{4} \equiv \Tb^{4}(\Cb)$. We may therefore define the Kummer surface as
\[
K :\equiv \dsum{\Kb : \B\type{Gal}(\Cb: \Rb)} \Tb^{4}(\Kb).
\]
  \end{ex}

  \begin{ex}
    The teardrop orbifold $\Sb^{2}(n)$ for $n : \Nb$ ($n > 1$) may be constructed as $\Sb^{3}\sslash U(1)$ where $U(1)$ acts on $\Sb^{3} :\equiv \{z : \Cb^{2} \mid |z| = 1\}$ via $(z_{1}, z_{2}) \mapsto (\lambda z_{1}, \lambda^{n} z_{2})$. We can describe this action as the map which sends $L : \B U(1)$ to $\Sb(L \oplus L^{\otimes n})$, where we consider $L \oplus L^{\otimes n} : \B U(2)$ as a 2-dimensional Hermitian vector space and define the unit sphere $\Sb(V)$ for $V : \B U(n)$ by $\Sb(V) :\equiv \{v : V \mid \langle v, v \rangle = 1\}$. Therefore, we may construct the teardrop as the type of pairs of a $1$-dimensional Hermitian vector space $L$ and a unit length element of $L \oplus L^{\otimes n}$:
    \[
\Sb^{2}(n) :\equiv \dsum{L : \B U(1)} \Sb(L \oplus L^{\otimes n}).
    \]

    We can generalize this definition to the weighted projective spaces $W\Pb(n_{1}, \ldots, n_{k})$ (for a natural numbers $n_{1}, \ldots, n_{k}$ all coprime).
    \[
W\Pb(n_{1}, \ldots, n_{k}) :\equiv \dsum{L : \B U(1)} \Sb(L^{\otimes n_{1}} \oplus \cdots \oplus L^{\otimes n_{k}}).
    \]
\end{ex}

  \begin{ex}\label{ex:torus.mod.group}
A large class of orbifolds which appear in practice may be constructed as quotients $\Tb^{n}\sslash \Gamma$ where $\Gamma \subseteq \type{GL}_{n}(\Zb)$ is a finite subgroup of $\type{GL}_{n}(\Zb)$ acting on $\Tb^{n} = \Rb^{n}/\Zb^{n}$ via the action of $\type{GL}_{n}(\Zb)$ on $\Rb^{n}$ by matrix multiplication. The Kummer surface of \cref{ex:kummer} is one example of this sort of orbifold, as are the pillowcases.

Here is one general construction of this sort of orbifold. We may deloop $\type{GL}_{n}(\Zb)$ by noting that this is the type of symmetries of the vector space $\Rb^{n}$ which preserve the lattice $\Zb^{n} \subseteq \Rb^{n}$. This suggests
\[
\B\type{GL}_{n}(\Zb) :\equiv \dsum{V : \B \type{GL}_{n}(\Rb)} \type{Lattice}(V)
\]
pointed at $(\Rb^{n}, \Zb^{n})$. We need to check that this is $0$-connected, so let $(V, \Lambda)$ be a lattice in an $n$-dimensional real vector space. There is some isomorphism $p : V = \Rb^{n}$, and this identifies $\Lambda$ with the lattice $p(\Lambda)$. Now choose generators for $p(\Lambda)$; this gives us an isomorphism $q : p(\Lambda) = \Zb^{n}$ considered as abstract groups; however, we may extend $q$ to an automorphism $\bar{q} : \Rb^{n} = \Rb^{n}$ of $\Rb^{n}$ by noting that the standard generators of $\Zb^{n}$ form a basis for $\Rb^{n}$. The composite $\bar{q}\inv \circ p : V = \Rb^{n}$ is an identification of $V$ with $\Rb^{n}$ which sends $\Lambda$ to $\Zb^{n}$.

For a finite subgroup $\Gamma$ of $\type{GL}_{n}(\Zb)$, we can now define $\B \Gamma$ out of $\B \type{GL}_{n}(\Zb)$ using \cref{prop:subtorsor.delooping}. We can then define
\[
\Tb^{n} \sslash \Gamma :\equiv \dsum{(V, \Lambda, T) : \B \Gamma} (V/\Lambda).
\]
Explicitly, the points of $\Tb^{n} \sslash \Gamma$ consist of an $n$-dimensional vector space $V$, a lattice $\Lambda$ in $V$, a $\Gamma$-subtorsor $T$ of the space of frames of $V$, and an point on the torus $V / \Lambda$.
\end{ex}

\section{Cohesion and the Homotopy Theory of Orbifolds}

We will now move beyond bare homotopy type theory and into \emph{modal} homotopy type theory. In this section, we will briefly survey the homotopy theory of orbifolds, and in particular their covering theory. This means, in particular, defining the \emph{homotopy type} of an orbifold.

Defining the homotopy type of a type is luckily quite straightforward. We would like to be able to identify points by giving continuous deformations between them.  In other words, we should have a sort of quotient map $(-)^{{\shape}} : X \to \shape X$ sending any point in $X$ to its homotopy class in the homotopy type $\shape X$, and given a path $\gamma : \Rb \to X$ we should get an identification $\gamma(0)^{\shape} = \gamma(1)^{\shape}$ in $\shape X$. In other words, we want to nullify maps out of $\Rb$, or more precisely, we want to localize our type $X$ at the terminal map $\Rb \to \ast$. The theory of localizations in HoTT is developed in \cite{RSS:Modalities} and \cite{LOPS:Localization.in.HoTT}, and we may use this theory to define the $\shape$ modality.

\begin{defn}[Definition 9.6 \cite{Shulman:Real.Cohesion}]
We define the \emph{shape} modality $\shape$ to be localization at the type of real numbers $\Rb$.\footnote{For now, I will leave ambiguous which type of real numbers we are localizing at. In \cref{sec:SDG.axioms}, we will see axioms for the type of \emph{smooth reals} which will play the role of the real numbers in synthetic differential geometry.} The \emph{$n$-shape} $\shape_{n}$ modality is the localization at $\Rb$ and the homotopy $(n+1)$-sphere $S^{n+1}$, and its modal types are those types which are both $\shape$-modal and $n$-truncated. A definition of this localization can as Definition 9.6 of \cite{Shulman:Real.Cohesion} or (for a general localization) in Section 2.2 of \cite{RSS:Modalities}.
  \end{defn}

  This definition is short and sweet, but without a supporting apparatus it is unfortunately underspecified. That supporting apparatus is the \emph{cohesive homotopy type theory} which adds a \emph{comodality} $\flat$ that strips types of their spatial structure. For this reason, we begin this section with a review of cohesive homotopy type theory in \cref{sec:review.cohesion}.

  In \cref{sec:modal.covering.theory}, we will then review the cohesive covering theory developed in \cite{Jaz:Good.Fibrations}. We will use this covering theory to quickly compute the homotopy type of $\Ma_{1,1}$ in \cref{thm:covering.classification}: it is a $\B\type{SL}_{2}(\Zb)$.

  Then, in \cref{sec:orbifold.maps}, we will  take a brief look at maps between good orbifolds. With the modal approach to covering theory and the HoTT approach to group theory, we will see that maps into a configuration space $X^{n} \sslash n!$ correspond to maps out of $n$-fold covers (\cref{thm:map.into.config.space}), and that maps into a quotient $X \sslash \Gamma$ corresopnd to $\Gamma$-equivariant maps out of $\Gamma$-principal bundles (\cref{thm:map.into.quotient}).

  \subsection{A review of cohesive homotopy type theory.}\label{sec:review.cohesion}
  To understand what cohesion adds to type theory, let's think in terms of models for a bit. We intend to interpret our type theory in a topos of smooth stacks, such as the Dubuc $\infty$-topos. As with any $\infty$-topos, this topos of smooth stacks lives over the topos of homotopy types (stacks on the point) by its global sections functor.

  \[
    \begin{tikzcd}[row sep = large]
      \text{smooth stacks} \ar[d, shift left = 25, bend left = 20, "\mbox{global sections}"] \ar[d, hookleftarrow, "\mbox{locally constant stacks}" description] \ar[d, shift right = 25, bend right = 20, "\mbox{shape}"'] \\
      \text{homotopy types}
      \end{tikzcd}
  \]
  Left adjoint to the global sections functor is the functor sending a homotopy type $X$ the stack of locally constant sections valued in $X$. In our case, this functor really is the inclusion of the constant stacks --- it is fully faithful. Then there is a further left adjoint: this sends a stack to its \emph{shape} (in the sense of Lurie). If that stack is represented by a manifold, then its shape will be the homotopy type of that manifold. In the terminology of higher toposes, the topos of smooth stacks is $\infty$-connected and locally $\infty$-connected.\footnote{It is furthermore \emph{local}, in that the global sections functor also admits a \emph{right} adjoint. This right adjoint is the inclusion of codiscretes. We will only need this other adjoint modality briefly, for \cref{lem:crisp.Penon.manifold.disk}, and so we will not dwell on it here.}

  Cohesive HoTT \cite{Shulman:Real.Cohesion} formalizes this relationship between smooth stacks and homotopy types by adding \emph{crisp} variables to homotopy type theory. Every expression in HoTT occurs in a \emph{context}, which is a list of the free variables in the expression, together with the type these variables have. To define a function $f : X \to Y$, we construct $f(x) : Y$ in the context of a free variable $x : X$. Since any such function $f : X \to Y$ will be interpreted as a map of stacks --- and if these stacks are represented by manifolds, therefore potentially as a smooth map of manifolds --- the dependence of an expression $f(x)$ on its free variable $x : X$ implies a sort \emph{smoothness}. But not all dependencies in mathematics should be smooth; sometimes an expression $f(x)$ should vary \emph{discontinuously} in $x$.

  Shulman allows for discontinuous dependency with a new type of free variable declaration: \emph{crisp} variables $x :: X$. To say that $f(x) : Y$ for a \emph{crisp variable} $x :: X$ is to say that $f(x)$ depends on $x$ in a (possibly) discontinous way. If all of the free variables in an expression $f$ are crisp, we say that $f$ is crisp. Importantly, crisp variables must also have crisp type. In particular, every expression with no free variables is \emph{crisp}. For example, $\Zb$ and $\Rb$ are crisp types, and $0 : \Zb$ and $\pi : \Rb$ are crisp elements. While the expression $x^{2} + 1 : \Rb$ with $x : \Rb$ is not crisp, the function $(x \mapsto x^{2} + 1) : \Rb \to \Rb$ is crisp since the variable $x$ has been bound.

  One way to ensure that crisp variable behave discontinuously is the crisp law of excluded middle.
  \begin{axiom}[\cite{Shulman:Real.Cohesion}]\label{axiom:LEM}
For any crisp proposition $P :: \Prop$, either $P$ holds or $\neg P$ holds.
\end{axiom}

The crisp law of excluded middle lets us define functions of crisp variables by cases. For example, if $x :: \Rb$ is a crisp variable, then the proposition $(x > 0) : \Prop$ is also crisp (since every free variable in it is crisp). Therefore, either $(x > 0)$ or $\neg(x > 0)$. We can therefore define the real number
\[
  f(x) :\equiv \begin{cases} -1 &\mbox{if $x > 0$} \\ 1 &\mbox{otherwise} \end{cases}
\]
by cases. Clearly, $f(x)$ is a discontinuous function of $x$ which we were able to define using law of excluded middle. Without the law of excluded middle, it is impossible to define discontinuous functions $\Rb \to \Rb$.

An expression such as $f(x) : Y$ above depending on a crisp variable $x :: X$ can't give a function $X \to Y$, because $X \to Y$ is supposed to be the type of smooth functions (or, really, the mapping stack). To internalize the crisp variables, we can add a type $\flat X$ which is ``freely generated by the crisp variables of $X$'' in the sense that $f(x) : Y$ depending on $x :: X$ gives rise to a function $\flat X \to Y$. We can think of $\flat X$ as $X$ stripped of its smooth structure; in terms of stacks, $\flat X$ is the stack constant at the global sections of $X$.

More formally, for any \emph{crisp} type $X$ we have a type $\flat X$ and for every crisp $x :: X$ we have $x^{\flat} : \flat X$. We then have the following induction principle: if $C : \flat X \to \Type$ is any type family, and if for $x :: X$ we have $c(x) : C(x^{\flat})$, then for any $u : \flat X$ we have an element
\[
(\mbox{let $x^{\flat} :\equiv u$ in $c(x)$}) : C(u).
\]
Furthermore, if $u \equiv y^{\flat}$ for $y :: X$, we have
\[
(\mbox{let $x^{\flat} :\equiv y^{\flat}$ in $c(x)$}) \equiv c(y).
\]
We can define a counit $(-)_{\flat} : \flat X \to X$ by $u_{\flat} :\equiv (\mbox{let $x^{\flat} :\equiv u$ in $x$})$. Less formally, we might say that $(-)_{\flat}$ is defined by $(x^{\flat})_{\flat} :\equiv x$. Given any smooth function $f : X \to Y$, we can precompose by $(-)_{\flat} : \flat X \to X$ to get its underlying discontinuous function $f \circ (-)_{\flat} : \flat X \to Y$.

A type $X$ should be discrete when any discontinuous function out of it is already smooth. That is, a type $X$ should be discrete precisely when precomposition by $(-)_{\flat}$ gives an equivalence of $X \to Y$ with $\flat X \to Y$ for any type $Y$. This will happen precisely when $(-)_{\flat}$ is an equivalence.
\begin{defn}[\cite{Shulman:Real.Cohesion}]
A crisp type $X :: \Type$ is \emph{(crisply) discrete} if $(-)_{\flat} : \flat X \to X$ is an equivalence.
\end{defn}

In terms of stacks, a crisply discrete type is a constant stack --- that is, one for which the canonical map from the constant stack at its global sections into it is an equivalence.

In order to relate the liminal spatiality of crisp variables to the concrete topology of the reals, we will relate the discreteness of $\flat$ with a discreteness measured by $\Rb$.

\begin{axiom}[$\Rb\flat$: \cite{Shulman:Real.Cohesion}]
  A crisp type $X :: \Type$ is discrete if and only if the inclusion $\term{const} : X \to (\Rb \to X)$ of constant paths is an equivalence. That is, $X$ is discrete if and only if every path $\gamma : \Rb \to X$ is constant.
  \[
(\flat X \simeq X) \quad\iff\quad (X \simeq \shape X).
  \]
  \end{axiom}

This axiom justifies extending the definition of discreteness to types which aren't crisp. We say a type $X$ is discrete just when $\term{const} : X \to (\Rb \to X)$ is an equivalence, or when it is $\Rb$-null. By construction, this is precisely when $(-)^{\shape} : X \to \shape X$ is an equivalence, so we see that a type is discrete if and only if it is $\shape$-modal.

  Let's end this review of cohesive homotopy type theory by quoting Theorem 9.15 of \cite{Shulman:Real.Cohesion}.
  \begin{thm}[Theorem 9.15 \cite{Shulman:Real.Cohesion}]
    For any crisp types $X$ and $Y$, we have an equivalence
    \[
\flat (X \to \flat Y) \simeq \flat (\shape X \to Y)
    \]
    exhibiting the adjointness of $\shape$ and $\flat$.
\end{thm}

\subsection{Modal covering theory} \label{sec:modal.covering.theory}
Let's take a minute to recall the modal covering theory developed in Section 9 of \cite{Jaz:Good.Fibrations}. A covering map $\pi : C \to X$ satisfies the \emph{unique path lifting} property:
\[
\begin{tikzcd}
	\ast & C \\
	\Rb & X
	\arrow["0"', from=1-1, to=2-1]
	\arrow["\forall"', from=2-1, to=2-2]
	\arrow["\pi", from=1-2, to=2-2]
	\arrow["\forall", from=1-1, to=1-2]
	\arrow["{\exists!}", dashed, from=2-1, to=1-2]
\end{tikzcd}
\]
For any path $\gamma : \Rb \to X$ and $c : C$ over $\gamma(0)$, there is a unique lift $\tilde{\gamma} : \Rb \to C$ of $\gamma$. Furthermore, if $f : A \to B$ is any map which induces an equivalence on fundamental groupoids, then $\pi : X \to Y$ lifts uniquely on the right against $f$. We can use this property to define the notion of covering using the fundamental groupoid modality $\shape_{1}$, which is given by localization both at $\Rb$ and the homotopy circle $S^{1}$ and whose modal types are discrete groupoids.

We will use the notion of a \emph{modal \'etale} map, studied in \cite{Cherubini-Rijke:Modal.Descent}.
\begin{defn}[\cite{Cherubini-Rijke:Modal.Descent}]\label{defn:modal.etale}
  For a modality $\modal$, a map $f : X \to Y$ is $\modal$-{\'e}tale when the modal naturality square
  \[
\begin{tikzcd}
	X & {\modal X} \\
	Y & {\modal Y}
	\arrow["f"', from=1-1, to=2-1]
	\arrow["{(-)^{\modal}}", from=1-1, to=1-2]
	\arrow["{(-)^{\modal}}"', from=2-1, to=2-2]
	\arrow["{\modal f}", from=1-2, to=2-2]
\end{tikzcd}
  \]
  is a pullback.
\end{defn}

\begin{defn}[Definition 9.1 \cite{Jaz:Good.Fibrations}]
A map $\pi : C \to X$ is a \emph{covering} if it is $\shape_{1}$-{\'e}tale and its fibers are sets.
\end{defn}

The justification of these definitions comes from Theorem 7.2 of \cite{Cherubini-Rijke:Modal.Descent} which shows that $\modal$-equivalences and $\modal$-{\'e}tale maps form an orthogonal factorization system.
\begin{thm}[Theorem 7.2 \cite{Cherubini-Rijke:Modal.Descent}]
A map $f : X \to Y$ is $\modal$-{\'e}tale if and only if it lifts uniquely on the right against all $\modal$-equivalences --- maps $g : A \to B$ for which $\modal g$ is an equivalence. Furthermore, the $\modal$-{\'e}tale maps and the $\modal$-equivalences form the right and left classes respectively of an orthogonal factorization system.
\end{thm}

Finally, we need a theorem from \cite{Jaz:Good.Fibrations}. We can characterize coverings of $X$ by the monodromy action of fundamental groupoid $\shape_{1} X$ of $X$.
\begin{thm}[Theorem 9.2 \cite{Jaz:Good.Fibrations}]\label{thm:covering.classification}
  For a type $X$, let $\type{Cov}(X)$ denote the type of coverings of $X$. Then we have an equivalence
  \[
\type{Cov}(X) \simeq (\shape_{1} X \to \Type_{\shape_{0}})
\]
between coverings of $X$ and discrete set valued functions on the fundamental groupoid of $X$. Given such a map $E : \shape_{1} X \to \Type_{\shape_{0}}$, the associated covering is the first projection from $C :\equiv \dsum{x : X} E_{x^{{\shape_{1}}}}$.
\end{thm}
\begin{rmk}
Note that if $X$ is connected in the sense that its fudmamental groupoid $\shape_{1} X$ is $0$-connected, and $x : X$ is any point, then $\shape_{1} X$ pointed at $x^{\shape_{1}}$ is a delooping $\B \pi_{1} X$ of the fundamental group $\pi_{1}X$ (based at $x$). In this case, \cref{thm:covering.classification} specializes to the usual theorem that coverings of $X$ are equivalent to actions $\B \pi_{1} X$ on discrete sets, with the action given by monodromy.
\end{rmk}

\paragraph{Coverings of orbifolds.}
With all this review out of the way, we can discuss the homotopy theory of orbifolds. We can begin by calculating the homotopy type of $\Ma_{1,1}$.

\begin{thm}\label{thm:homotopy.type.of.M11}
The homotopy type of $\Ma_{1,1}$ (\cref{defn:M11}) is a $\B \type{SL}_{2}(\Zb)$, and $q : \mathfrak{h} \to \Ma_{1,1}$ is the universal cover of $\Ma_{1,1}$.
  \end{thm}
  \begin{proof}
    By \cref{prop:M11.homotopy.quotient}, $q : \mathfrak{h} \to \Ma_{1,1}$ has fibers which are $\type{SL}_{2}(\Zb)$ torsors and so is the fiber of a map $\fib_{q} : \Ma_{1,1} \to \type{Tors}_{\type{SL}_{2}(\Zb)}$. Since $\type{SL}_{2}(\Zb)$ is a crisply discrete group, $\type{Tors}_{\type{SL}_{2}(\Zb)}$ is also discrete by Theorem 5.9 of \cite{Jaz:Good.Fibrations}. Since $\mathfrak{h}$ is $\shape$-connected, we see that the map $\fib_{q} : \Ma_{1,1} \to \type{Tors}_{\type{SL}_{2}(\Zb)}$ is a $\shape$-connected map into a $\shape$-modal type, making it a $\shape$-unit.

To see that $q : \mathfrak{h} \to \Ma_{1,1}$ is the universal cover, note that it is the fiber of $\fib_{q} : \Ma_{1,1} \to \type{Tors}_{\type{SL}_{2}(\Zb)}$ over the canonical \exemplar. Since $\fib_{q}$ is a $\shape_{1}$-unit, this exhibits $q$ as the universal cover (see Definition 9.3 and Theorem 9.4 of \cite{Jaz:Good.Fibrations}).
    \end{proof}

The argument in \cref{thm:homotopy.type.of.M11} is completely general.
    \begin{thm}
      Let $X\twisted{-} : \B \Gamma \to \Type$ be an action of a discrete higher group $\Gamma$ (in the sense that $\B \Gamma$ is $\shape$-modal) on a $\shape$-connected type $X$. Then the first projection
      $$\fst : X \sslash \Gamma \to \B \Gamma$$
      is a $\shape$-unit.
      \end{thm}
      \begin{proof}
 The map $\fst : X \sslash \Gamma \to \B \Gamma$ is a map whose fibers are identifiable with the $\shape$-connected type $X$, and it is therefore a $\shape$-connected into the $\shape$-modal type $\B \Gamma$. Therefore, it is a $\shape$-unit.
        \end{proof}

        As a corollary, we see that
        $$\shape(\Rb^{n} \sslash \Gamma) \simeq \B \Gamma$$
        when $\Rb^{n}\sslash \Gamma$ is a coordinate patch for a finite subgroup $\Gamma \subseteq \type{O}(n)$ as in \cref{ex:coordinate.patch}. We can be even more general so long as our higher group is crisp.
        \begin{thm}
          Let $X\twisted{-}: \B G \to \Type$ be an action of a crisp higher group $G$ on a type $X$. Then there is a unique action of $\shape G$ on $\shape X$ so that
          \[
\shape(X \sslash G) \simeq \shape X \sslash \shape G.
          \]
          \end{thm}
          \begin{proof}
            Consider the map $t \mapsto \shape X\twisted{t} : \B G \to \Type_{\shape}$ sending an \exemplar $t$ of $G$ to the homotopy type of $X$ twisted by $t$. This lands in the $\shape$-separated type of discrete types $\Type_{\shape}$ and so factors uniquely through the $\shape$-separated unit $(-)^{\shape^{(1)}}: \B G \to \shape^{(1)}\B G$. But by the proof of Theorem 8.9 of \cite{Jaz:Good.Fibrations}, we see that the factorization $\shape^{(1)}\B G \to \shape\B G$ of the $\shape$-unit of $\B G$ is an equivalence, so that $\shape X\twisted{-} : \B G \to \Type_{\shape}$ factors through $\shape \B G$; we take this as our action of $\shape G$ on $\shape X$, since $\shape \B G$ deloops $\shape G$ by Theorem 8.9 of \cite{Jaz:Good.Fibrations}.

            It remains to show that $\shape X \sslash \shape G$ is $\shape(X \sslash G)$. For $t : \B G$, we have a $\shape$-unit $(-)^{\shape} : X\twisted{t} \to \shape X\twisted{t}$. We can assemble these into a map
            \[
(t, x) \mapsto (t^{\shape}, x^{\shape}) : X \sslash G \to \shape X \sslash \shape G.
\]
Since pair types of modal types are modal, $\shape X \sslash \shape G$ is $\shape$-modal; so it suffices to show that this map is $\shape$-connected. But it is the pairing of $\shape$-connected maps, so by Lemma 1.39 of \cite{RSS:Modalities}, it is $\shape$-connected.
          \end{proof}

          As a corollary, we can compute the homotopy types of a few more of our examples.
          \begin{cor}
            Let $\Gamma$ be a crisp finite subgroup of $\type{GL}_{n}(\Zb)$. Then
            \[
\shape(\Tb^{n} \sslash \Gamma) \simeq  \B \tilde{\Gamma}
\]
where $\tilde{\Gamma}$ is a \emph{crystallographic group} extending $\Gamma$:
\[
0 \to \Zb^{n} \to \tilde{\Gamma} \to \Gamma \to 0.
\]
        \end{cor}
        \begin{proof}
          Consider the fiber sequence
          \[
            \Rb^{n} / \Zb^{n} \to \Tb^{n} \sslash \Gamma \xto{\fst} \B \Gamma
          \]
          where we recall from \cref{ex:torus.mod.group} that $\Tb^{n} \sslash \Gamma :\equiv \dsum{(V, \Lambda, T) : \B \Gamma} (V/\Lambda)$ and that the canonical \exemplar of $\Gamma$ is $(\Rb^{n}, \Zb^{n}, \Gamma)$. By Theorem 7.7 of \cite{Jaz:Good.Fibrations}, the projection $\fst : \Tb^{n} \sslash \Gamma \to \B \Gamma$ is a $\shape$-fibration and therefore
          \[
\shape(\Rb^{n} / \Zb^{n}) \to \shape(\Tb^{n} \sslash \Gamma) \to \shape \B \Gamma
\]
is a fiber sequence. Since $\Gamma$ is a crisply discrete group, $\B \Gamma$ is crisply discrete by Theorem 5.9 of \cite{Jaz:Good.Fibrations}, and $\shape(\Rb^{n}/ \Zb^{n}) = \shape((\Sb^{1})^{n}) = \B \Zb^{n}$, so we have a fiber sequence
\[
\B \Zb^{n} \to \shape (\Tb^{n} \sslash \Gamma) \to \B \Gamma.
\]
Now, we may point $\shape(\Tb^{n} \sslash \Gamma)$ at $\pt :\equiv (\Rb^{n}, \Zb^{n}, \Gamma, [0])^{\shape}$; it remains to show that $\shape(\Tb^{n} \sslash \Gamma)$ is $0$-connected. Since $\Tb^{n} \sslash \Gamma$ is crisp, $\trunc{\shape(\Tb^{n} \sslash \Gamma)}_{0} = \shape_{0}(\Tb^{n} \sslash \Gamma)$ by Proposition 4.5 of \cite{Jaz:Good.Fibrations}. But $\B \Gamma$ is discrete and $0$-connected, so it is also $\shape_{0}$-connected; likewise, the torus $V / \Lambda$ is $\shape_{0}$-connected for any vector space $V$ and lattice $\Lambda$ in it since it is surjected by the $\shape_{0}$-connected type $V$. Therefore, $\Tb^{n} \sslash \Gamma$ is $\shape_{0}$-connected as the sum of $\shape_{0}$-connected types.

Defining $\tilde{\Gamma} :\equiv \pi_{1}(\Tb^{n} \sslash \Gamma)$, which in this case is equivalent to $\Omega(\shape(\Tb^{n} \sslash \Gamma), \pt)$, we see that $\shape(\Tb^{n} \sslash \Gamma)$ is a $\B \tilde{\Gamma}$ and we have an extension
\[
0 \to \Zb^{n} \to \tilde{\Gamma} \to \Gamma \to 0. \qedhere
\]
        \end{proof}

\subsection{Maps between orbifolds}\label{sec:orbifold.maps}

Another upside of working in homotopy type theory is that correct notion of map between orbifolds is simply a function, which can be defined by its action on points in the usual way. That is, if $\Xa$ and $\Ya$ are orbifolds, then the mapping space between them is the space of functions $\Xa \to \Ya$. In particular, since we have defined our orbifolds in terms of their points, it is fairly straightforward to understand what it means to map into them.

\begin{prop}\label{thm:map.into.config.space}
Let $X$ be a type, and consider the configuration space $X^{n} \sslash n!$. A function $f : A \to X^{n} \sslash n!$ is equivalently an $n$-fold cover $\pi : C_{f} \to A$ together with a map $\tilde{f} : C_{f} \to X$.
\end{prop}
\begin{proof}
  We may calculate directly:
  \begin{align*}
    (A \to X^{n} \sslash n!) &\simeq (A \to \dsum{F : \B \Aut(\ord{n})} X^{F}) \\
                             &\simeq \dsum{C : A \to \B \Aut(\ord{n})} (\dprod{a : A} X^{C_{a}}) \\
                             &\simeq \dsum{(C, \pi) : \type{Cov}(A)} \dsum{\dprod{a : A} \trunc{\fib_{\pi}(a) \simeq \ord{n}}} (\dprod{a : A} X^{\fib_{\pi}(a)})\\
                             &\simeq \dsum{(C, \pi) : \type{Cov}(A)}\dsum{\dprod{a : A} \trunc{\fib_{\pi}(a) \simeq \ord{n}}} (\dsum{a : A} \fib_{\pi}(a) \to X) \\
                             &\simeq \dsum{(C, \pi) : \type{Cov}(A)}\dsum{\dprod{a : A} \trunc{\fib_{\pi}(a) \simeq \ord{n}}} (C \to X).\\
\end{align*}
We made use of \cref{thm:covering.classification} and the fact that finite sets are discrete.
\end{proof}

We can prove something slightly more general but along the same lines.
\begin{prop}\label{thm:map.into.quotient}
Suppose that a (higher) group $\Gamma$ acts on a type $X$. Then maps $f : A \to X \sslash \Gamma$ correspond to $\Gamma$-principal bundles $\pi : P \to A$ together with a $\Gamma$-equivariant map $P \to X$.
\end{prop}
\begin{proof}
  We may calculate directly:
  \begin{align*}
    (A \to X \sslash \Gamma) &\equiv (A \to \dsum{T : \type{Tors}_{\Gamma}} X\twisted{T}) \\
                             &\simeq (\dsum{P : \Type}\dsum{\pi : P \to A}\dsum{\dprod{a : A} \type{Tors}_{\Gamma}(\fib_{\pi}(a))}(\dprod{a : A} X\twisted{\fib_{\pi}(a)})) \\
                             &\simeq\left\{ \begin{aligned}
                                 {( P : \Type )}&\times\dsum{\pi : P \to A}({\dprod{a : A} \type{Tors}_{\Gamma}(\fib_{\pi}(a))})\\
                                 &\times(\dprod{(T, a, p) : \dsum{T : \type{Tors}_{\Gamma}} \dsum{a : A} (\fib_{\pi}(a) = T)} X\twisted{\fib_{\pi}(a)}) \end{aligned}
    \right.\\
                             &\simeq\left\{ \begin{aligned}
                                 {( P : \Type )}&\times\dsum{\pi : P \to A}({\dprod{a : A} \type{Tors}_{\Gamma}(\fib_{\pi}(a))})\\
                                 &\times(\dprod{T : \type{Tors}_{\Gamma}} \dprod{(a, p) : \dsum{a : A} (\fib_{\pi}(a) = T)} X\twisted{\fib_{\pi}(a)})) \end{aligned}
    \right.\\
                             &\simeq\left\{ \begin{aligned}
                                 {( P : \Type )}&\times\dsum{\pi : P \to A}({\dprod{a : A} \type{Tors}_{\Gamma}(\fib_{\pi}(a))})\\
                                 &\times\type{Hom}_{\Gamma}(P, X) \end{aligned}
    \right.\\
\end{align*}
The equivalence $\type{Hom_{\Gamma}}(P, X) \simeq (\dprod{T : \type{Tors}_{\Gamma}}\dprod{(a, p) : \dsum{a : A} (\fib_{\pi}(a) = T)} X\twisted{\fib_{\pi}(a)}) $ is either a definition (for general higher groups $\Gamma$ and types $X$) or a theorem (for ordinary groups $\Gamma$ and sets $X$). To understand this last step, first note that the action of $\Gamma$ on the total space $P$ of a $\Gamma$-bundle $\pi : P \to A$ may be described by the map
\[
(T : \type{Tors}_{\Gamma}) \mapsto P\twisted{T} :\equiv \dsum{a : A} ( \fib_{\pi}(a) = T).
\]
When we apply this function to $\Gamma : \type{Tors}_{\Gamma}$, we get may calcuate that $P\twisted{\Gamma}$ is
\begin{align*}
  ( \dsum{a : A} (\fib_{\pi}(a) = \Gamma) ) &\simeq (\dsum{a : A} \fib_{\pi}(a))\\
  &\simeq P,
\end{align*}since a $\Gamma$-equivariant identification with $\Gamma$ is determined by an element. Finally, a $\Gamma$-equivariant map $P \to X$ is equivalently a map $\dprod{T : \type{Tors}_{\Gamma}} (P \twisted{T} \to X\twisted{T})$, which is what appears at the end of the calculation above.
  \end{proof}

\section{Synthetic Differential Geometry}

Synthetic differential geometry began in 1967 with a series of lectures by Lawvere in which he attempted to give a topos-theoretic foundation (in the sense of a distillation of established practices) for the sorts of differential geometry used in engineering and physics that made explicit use of infinitesimals \cite{Lawvere:Chicago.Lectures,Lawvere:SDG}. Lawvere was influenced by the use of nilpotent infinitesimal elements appearing in non-reduced schemes in Grothendieck's reformulated algebraic geometry. The field of synthetic differential geomery was further developed by Kock, Dubuc, Bunge, Penon, Lavendhomme, Reyes, Moerdijk, and others. An introductory text is \cite{Bell:Smooth.Infinitesimal.Analysis}; reference texts are \cite{Kock:SDG} and \cite{Lavendhomme:SDG}. See also \cite{Bunge-Gago-SanLuis:Synthetic.Differential.Topology}. For topos theoretic models, see \cite{Moerdijk-Reyes:SDG}.

The main idea of synthetic differential geometry is to formalize the common arguments using numbers $\ep$ which are so small that their square $\ep^{2}$ is negligable. If we say that $\ep^{2} = 0$ is actually $0$, then such numbers are \emph{nilsquare infinitesimals}. The most crucial axiom of SDG, known as the Kock-Lawvere axiom, implies that for any function $f : \Rb \to \Rb$, there is a unique function $f' : \Rb \to \Rb$ so that for all $x : \Rb$ and $\ep^{2} = 0$, we have:

$$f(x + \ep) = f(x) + f'(x)\ep.$$
This axiom implies that every function $f : \Rb \to \Rb$ is smooth, and for that reason the real numbers $\Rb$ of SDG are known as the \emph{smooth reals}.

Of course, there cannot be any non-zero infinitesimals, at least if $\Rb$ is to be a field where non-zero elements are invertible. If $\ep$ was non-zero, then it would be invertible, and we could conclude that
\[1 = \ep^{2}\cdot\frac{1}{\ep^{2}} = 0 \cdot \frac{1}{\ep^{2}} = 0.\]
Since $1$ does not equal $0$, we may conclude that if $\ep^{2} = 0$, then $\ep$ is \emph{not non-zero}.

Classically, we could conclude from this that $\ep$ must be $0$; but this follows from a separate law of logic, \emph{double negation elimination}, which states that if a proposition is not false, then it is true. We do not have to take this law of logic as an axiom --- it does not follow from the rules of type theory. With our extra logical wiggle room, we can have a non-trivial theory of infinitesimal calculus. In fact, we will follow Penon in defining an \emph{infinitesimal} to be a number $\ep : \Rb$ which is not non-zero in \cref{defn:infinitesimal}.

So far, synthetic differential geometry has only been studied in $1$-topos theory. We will see that the same axioms (introduced in \cref{sec:SDG.axioms}), and the same definitions (such as that of microlinearity, \cref{defn:microlinear})  give us access to the differential structure of orbifolds and other higher types when interpreted in cohesive homotopy type theory.

The SDG literature has settled on the notion of \emph{microlinearity} (\cref{defn:microlinear}) as a good notion of ``smooth space'' for the purposes of proving theorems in SDG. Beautifully, the definition of microlinearity generalizes smoothly from sets to higher types. A type is microlinear when, roughly speaking, it has the same infinitesimal lifting properties as $\Rb$. In \cref{sec:microlinear}, we review the definition of microlinearity, and prove in \cref{thm:inf.linear.R.module} that the tangent spaces of microlinear types have the structure of $\Rb$-modules. While this theorem is standard for microlinear sets, we prove it in such a way that it generalizes to higher types. Specifically, we show that tangent spaces of these higher types --- which themselves may be higher types and not sets --- are models of the Lawvere theory of $\Rb$-modules and so have a $\Rb$-module structure which is coherent up to higher identifications.

In \cref{sec:crystaline.modality}, we will prove  our main theorem of this section concerning the descent of microlinearity along $\Im$-\'etale maps. \cref{thm:microlinear.descends.along.etale} states that if $X$ is microlinear and $f : X \to Y$ is surjective and $\Im$-\'etale, then $Y$ is also microlinear. This will allow us to give examples in \cref{sec:notions.of.smooth.space} of higher microlinear types, such as \'etale groupoids (\cref{thm:etale.groupoid.microlinear}) and the quotients of microlinear types by higher groups (\cref{thm:discrete.homotopy.quotient.is.microlinear}). Here, $\Im$ is the \emph{crystaline} modality given by localizing at the type of infinitesimals in $\Rb$ (\cref{defn:crystaline.modality}). The type $\Im X$ is sometimes known as the \emph{de Rham stack} of $X$, and a map $f : X \to Y$ is $\Im$-\'etale when the it's $\Im$-naturality square is a pullback:
\[
\begin{tikzcd}
	X & {\Im X} \\
	Y & {\Im Y}
	\arrow["f"', from=1-1, to=2-1]
	\arrow["{(-)^{\Im}}", from=1-1, to=1-2]
	\arrow["{(-)^{\Im}}"', from=2-1, to=2-2]
	\arrow["{\Im f}", from=1-2, to=2-2]
	\arrow["\lrcorner"{anchor=center, pos=0.125}, draw=none, from=1-1, to=2-2]
\end{tikzcd}
\]
This is a useful and entirely modal notion of local diffeomorphism. In \cref{prop:etale.means.iso.on.tangent.spaces}, we will see that any $\Im$-\'etale map $f : X \to Y$ induces an isomorphism $f_{\ast} : T_{x} X \to T_{fx} Y$ on tangent spaces, and later, in \cref{cor:crisp.open.is.etale}, we will see that this is an equivalent condition for $f$ to be $\Im$-\'etale so long as $X$ and $Y$ are manifolds.

\subsection{Axioms of synthetic differential geometry}\label{sec:SDG.axioms}

Synthetic differential geomtery proceeds by axiomatizing the \emph{smooth real line}, which we will denote by $\Rb$. We will use the naming convention of \cite{Bunge-Gago-SanLuis:Synthetic.Differential.Topology} (except for Postulates E and S, which do not appear there, and the Covering Property, which is due to Bunge and Dubuc \cite{Bunge-Dubuc:Covering.Property}).

\begin{axiom}\label{axiom:SDG}
The \emph{smooth real line} $\Rb$ is a ring satisfying the following axioms:
\begin{itemize}
  \item (Postulate K) $\Rb$ is a field in the sense of Kock: $0 \neq 1$ and for any $n : \Nb$ and $x : \Rb^{n}$, we have
        $$\neg \left( \bigwedge_{i = 1}^{n} (x_{i} = 0) \right) \to \bigvee_{i = 1}^{n}(x_{i} \mbox{ is invertible}).$$
        Taking the case $n \equiv 1$ tells us that if $x \neq 0$ then $x$ is invertible (and therefore the invertible elements coincide with the non-zero elements of $\Rb$). Taking the case $n \equiv 2$ tells us that $\Rb$ is a local ring in the sense that if $x + y$ is invertible, then one of $x$ or $y$ is invertible.
  \item (Postulate O) $\Rb$ is strictly ordered: there is a binary relation $<$ on $\Rb$ satisfying the following axioms:
        \begin{enumerate}
                \item $1 > 0$, and if $x > 0$ and $y > 0$, then $x + y > 0$ and $xy > 0$.
                \item It is never the case that $x > x$.
                \item If $x > y$, then either $x > z$ or $z > y$ for any $z$.
                \item If $x \neq 0$, then $x < 0$ or $x > 0$.\footnote{At this point in \cite{Bunge-Gabo-SanLuis:Synthetic.Differential.Topology}, the authors have the axiom $\neg\left(\bigwedge_{i = 1}^{n} x_{i} = 0\right) \to \bigvee_{i = 1}^{n}(x_{i} < 0 \vee x_{i} > 0)$, but in light of Postulate K the axiom we are using here is equivalent.}
                \item (Archimedean law) For any $x : \Rb$, there is an $n : \Nb$ with $x < n$.
        \end{enumerate}
        \item (Postulate Exp) There is an isomorphism of ordered groups $\exp : \Rb \simeq (0, \infty) : \log$ between the additive group of real numbers and the multiplicative group of positive real numbers.

        \item (Postulate E, the Covering Property) Let $A, B \subseteq \Rb$ be subsets of $\Rb$. If $A \cup B = \Rb$, then for every $x : \Rb$, either there is an $\ep > 0$ with $B(x, \ep) \subseteq A$, or there is an $\ep > 0$ with $B(x, \ep) \subseteq B$.
        \item (Principle of Constancy) Let $f : \Rb \to \Rb$. If for all $x : \Rb$ and $\ep : \Rb$ with $\ep^{2} = 0$, $f(x + \ep) = f(x)$, then $f$ is constant.
        \item (Postulate W) The crisp infinitesimal varieties (\cref{defn:infinitesimal.variety}) and the type $\Dc$ of infinitesimals in $\Rb$ (\cref{defn:infinitesimal}) are \emph{tiny} (\cref{defn:tiny}).
  \item (Postulate J) The Kock-Lawvere axiom: For every Weil algebra $W$ over $\Rb$, the evaluation map
        $$w \mapsto \varphi \mapsto \varphi(w) : W \to (\Spec_{\Rb} W \to \Rb)$$
        is an isomorphism. We will explain these terms and the consequences of this axiom shortly.
\end{itemize}
\end{axiom}

Let's explain Postulate J, which is the axiom which underlies the differential geometric aspects of synthetic differential geometry. To do this, we need to understand the notion of a \emph{Weil algebra}.

\begin{defn}[Standard, see \cite{Lavendhomme:SDG}]
Let $R$ be a ring. A \emph{Weil algebra} over $R$ is an augmented finitely presented $R$-algebra $\pi : W \to R$ whose augmentation ideal $\ker \pi$ is finitely generated and nilpotent. The category of Weil algebras $\type{Weil}$ is the full subcategory of the augmented $R$-algebras spanned by the Weil algebras.
\end{defn}
\begin{rmk}
Note that a Weil algebra is something entirely different from a Weyl algebra.
\end{rmk}

Any Weil algebra may be put in a standard form as the quotient of a polynomial algebra where the augmentation is given by evaluating at $0$.
\begin{lem}[Standard]\label{lem:weil.alg.standard.form}
Any augmented algebra with $W$ finitely presented $\pi : W \to \Rb$ is merely equivalent to a augmented algebra of the form $\term{ev}_{0} : \Rb[x_{1}, \ldots, x_{n}]/(f_{1}, \ldots, f_{m}) \to \Rb$ with augmentation given by sending each $x_{i}$ to $0$.
\end{lem}
\begin{proof}
This is a quick change of variables. By hypothesis, $W$ is finitely presented as a $\Rb$-algebra, so it is of the form $\Rb[y_{1}, \ldots, y_{n}]/(g_{1}, \ldots, g_{m})$. Let $\varphi : \Rb[x_{1}, \ldots, x_{n}] \to \Rb[y_{1}, \ldots, y_{n}]$ be the map given by $\varphi(x_{i}) :\equiv y_{i} - \pi(y_{i})$, which we note is an equivalence. Define $f_{i} :\equiv \varphi\inv(g_{i})$, which is to say that $f_{i}(x_{1}, \ldots, x_{n}) = g_{i}(x_{1} + \pi(y_{1}), \ldots, x_{n} + \pi(y_{n}))$. By construction, $\varphi$ descends to an equivalence $\Rb[x_{1}, \ldots, x_{n}]/(f_{1}, \ldots, f_{n}) = \Rb[y_{1}, \ldots, y_{n}]/(g_{1}, \ldots, g_{n})$, and since $\pi\varphi(x_{i}) = \pi(y_{i} - \pi(y_{i})) = 0$, this equivalence commutes with the augmentation.
\end{proof}

\begin{rmk}
The canonical example of a Weil aglebra is $R[x]/(x^{2})$ equipped with the augmentation $x \mapsto 0 : R[x]/(x^{2}) \to R$.
\end{rmk}

\begin{defn}[Standard, see \cite{Lavendhomme:SDG}]
  Let $A$ be an $R$-algebra. The \emph{synthetic spectrum $\Spec_{R}(A)$ of $A$ relative to $R$} is the set of $R$-algebra homomorphisms $A$ to $R$.
  $$\Spec_{R}(A) :\equiv \Hom_{R}(A, R).$$
\end{defn}

\begin{rmk}
  Note that if $A \equiv R[x_{1}, \ldots, x_{n}](f_{1}, \ldots, f_{m})$ is finitely presented $R$-algebra, then by the universal properties of quotients of polynomial algebras, the synthetic spectrum of $A$ over $R$ is the set of solutions to the equations $f_{1}, \ldots, f_{n}$:
  $$\Spec_{R}(A) = \{(r_{1}, \ldots, r_{n}) : R^{n} \mid \forall i.\, f_{i}(\vec{r}) = 0\} $$

\end{rmk}

  In particular, note that
  $$\Spec_{\Rb}(\Rb[x]/(x^{2})) \simeq \{\ep : \Rb \mid \ep^{2} = 0\}\equiv \Db.$$
  is the set of nilsquare infinitesimals $\Db$. The evaluation map $\Rb[x]/(x^{2}) \to (\Spec_{\Rb}(\Rb[x]/(x^{2})) \to \Rb)$ sends $a + bx$ to the function $\ep \mapsto a +  b\ep$. The Kock-Lawvere axiom (Postulate J) says that this map is an equivalence. In other words, every function $f : \Db \to \Rb$ is of the form $f(\ep) = a + b\ep$ for unique $a$ and $b$ in $\Rb$. Of course, plugging in $0$ for $\ep$ shows us that $a = f(0)$, so we see that there is a unique $b : \Rb$ for which $f(\ep) = f(0) + b\ep$ for all $\ep^{2} = 0$.

This axiom is valid in every context; that is, we can make use of it even when there are other free variables floating around. In particular, it gives us the following lemma.
  \begin{lem}[Standard, see \cite{Lavendhomme:SDG}]
    Let $f : \Rb \to \Rb$ be a function. Then there is a unique function $f' : \Rb \to \Rb$ such that for all $x : \Rb$ and $\ep^{2} = 0$, we have
    \[
f(x + \ep) = f(x) + f'(x) \ep.
\]
We refer to $f'$ as the \emph{derivative} of $f$.
  \end{lem}
  \begin{proof}
Given $x : \Rb$, define $g_{x}(\ep) :\equiv f(x + \ep)$ and note that $g_{x} : \Db \to \Rb$. Therefore, there is a unique $b_{x} : \Rb$ for which $g_{x}(\ep) = g_{x}(0) + b_{x}\ep$. We may therefore define $f'(x) :\equiv b_{x}$.
  \end{proof}

  In general, a function $v : \Db \to X$ plays the role of a tangent vector in $X$, based at $v(0) : X$. In particular, a function $v : \Db \to \Rb$ can be considered as a tangent vector based at $v(0)$, and we see that the type of all such vectors is equivalent to $\Rb$ by the association of $b$ with $w(\ep) :\equiv v(0) + b\ep$.

  \begin{defn}[Standard, see \cite{Lavendhomme:SDG}]
    Let $X$ be a type, and $x : X$ an element. We define the tangent space $T_{x} X$ of $X$ based at $x$ to be the type of pointed function $v : \Db \pto X$ sending $0$ to $x$.
    \[
T_{x} X :\equiv \dsum{v : \Db \to X} (v(0) = x).
    \]

    The tangent bundle is the projection $\fst : TX :\equiv \dsum{x : X} T_{x}X \to X$.
  \end{defn}

  Here's an example of how we might compute a tangent space. Specifically, we will show that the Lie algebra of $U(1)$ is $\Rb$.
  \begin{lem}[Standard]
The tangent space $T_{1}U(1)$ of $U(1) :\equiv \{z : \Cb \mid z\bar{z} = 1\}$ at $1$ is identifiable with the set $1 + i\Rb$ of numbers of the form $1 + bi$ in $\Cb$ --- which is itself identifiable with $\Rb$.
    \end{lem}
    \begin{proof}
      Let $v : \Db \to U(1)$ be a tangent vector at $1$, so that $v(0)  = 1$. We can write $v(\ep) = a(\ep) + b(\ep)i$, and then note that $1 = v(0) = a(0) + b(0)i$ so that $a(0) = 1$ and $b(0) = 0$. We can further expand $v(\ep)$ as
      \[
v(\ep) = a(\ep) + b(\ep)i = (1 + a'(0)\ep) + (b'(0)\ep)i
      \]
      However, we still know that $v(\ep) \overline{v(\ep)} = 1$, so in particular
      \[
1 = (1 + a'(0)\ep)^{2} + (b'(0)\ep)^{2} = 1 + 2a'(0)\ep
      \]
      from which we may conclude that $a'(0) = 0$. Therefore, we see that $v(\ep) = 1 + b'(0)\ep i$ for a unique element $b'(0) : \Rb$, which proves the lemma.
      \end{proof}

      \begin{rmk}
        Note that if $X$ is a higher type (such as an orbifold), then the isotropy group $\Aut_{X}(x)$ of $x : X$ acts on $T_{x} X$. This action is easy to define as a function $T_{x}X\twisted{-} : \B \Aut_{X}(x) \to \Type$, namely:
        \[
T_{x}X\twisted{y} :\equiv T_{y}X.
\]
      \end{rmk}

      Just because every type has a tangent bundle doesn't mean that every type is smooth. While we can always define a scalar action of $\Rb$ on $T_{x} X$ by $rv :\equiv \ep \mapsto v(r\ep)$, this action does not in general extend to the structure of an $\Rb$-module on $T_{x}X$. That is, we can't necessarily add tangent vectors. One pass at a definition of ``smooth type'' would be a type for which the tangent spaces are $\Rb$-modules. But we don't just want the first order algebraic structure of infinitesimals from $\Rb$, we want the higher order structure as well: we want all the algebraic structure of higher order infinitesimals in $\Rb$ to be present in ``smooth types''. This leads us to the notion of \emph{microlinear} types.

      \subsection{Microlinear types}\label{sec:microlinear}

      In this section, we will review the notion of \emph{microlinear} types, and characterize them in terms of the algebraic theory of \emph{finite order algebras}. The synthetic differential geometry community has settled on the notion of microlinearity as the correct notion of ``smooth space'' in the context of SDG. The main theme of this half of the paper will be that all reasonable notions of manifold and orbifold give rise to microlinear types. This means that by naively extending this notion of smoothness to higher types, we correctly pick up the intuitively smooth higher types such as orbifolds.

In order to define microlinear types, we will need the notion of \emph{infinitesimal variety}.
\begin{defn}\label{defn:infinitesimal.variety}
  An \emph{infinitesimal variety} $V$ is the spectrum of a Weil algebra. More formally, a pointed type $V$ is an infinitesimal variety if there merely exists a Weil algebra $W$ for which $V \simeq \Spec_{\Rb}(W)$ as pointed types, where $\Spec_{\Rb}(W)$ is pointed by the augmentation of $W$.

  The category $\type{InfVar}$ of infinitesimal varieties is the full subcategory of pointed sets spanned by the infinitesimal varieties.
\end{defn}

The walking tangent vector $\Db = \type{Spec}_{\Rb}(\Rb[x]/(x^{2}))$ is an example of an infinitesimal variety. In fact, the category of infinitesimal varieties is dual to the category of Weil algebras by the Kock-Lawvere axiom.
\begin{lem}\label{lem:inf.var.is.weil.alg}
  We have an equivalence of categories:
  \[
      \begin{tikzcd}
\type{InfVar}^{\op} \ar[r, bend left, "\Rb^{(-)}"] \ar[r, bend right, leftarrow, "\type{Spec}"'] & \type{WeilAlg}.
      \end{tikzcd}
  \]
\end{lem}
\begin{proof}
  First, let's note that this is a contravariant adjunction. That is, we have the unit $\eta :\equiv w \mapsto [\varphi \mapsto \varphi(w)] : W \to \Rb^{\type{Spec}(W)}$ and counit $\ep :\equiv v \mapsto [f \mapsto f(v)] : V \to \type{Spec}(\Rb^{V})$, both given by evaluation. The Kock-Lawvere axiom (Postulate J) says that the unit $\eta$ is an isomorphism for Weil algebras $W$.
   Therefore, the left adjoint $\type{Spec} :  \type{WeilAlg} \to \type{InfVar}\op $ is fully faithful. Since it is by definition essentially surjective, this concludes our proof.
\end{proof}

We can justify the name ``infinitesimal variety'' if we have a good definition of ``infinitesimal'' (due to Penon \cite{Penon:Infinitesimals.and.Intuitionism}).
\begin{defn}[\cite{Penon:Infinitesimals.and.Intuitionism}]\label{defn:infinitesimal}
  A real number $x : \Rb$ is \emph{infinitesimal} if it is not non-zero. More generally, for $x, y : X$ of any set $X$, define the \emph{neighbor} relation $x \approx y$ by
  \[
(x \approx y) :\equiv \neg \neg (x = y).
\]
An infinitesimal is $x$ such that $x \approx 0$. We denote the set of infinitesimals by
\[
\mathfrak{D} :\equiv \{x : \Rb \mid x \approx 0\}.
\]
If $x : X$, then we may define $\mathfrak{D}_{x}X :\equiv \{y : X \mid y \approx x\}$.
\end{defn}

\begin{rmk}
In the Dubuc topos, the type $\Dc$ of infinitesimals is representable by the $\Ca^{\infty}$-algebra $\Ca^{\infty}_{0}(\Rb)$ of germs of smooth functions on $\Rb$ at $0$. This is Proposition 11.5 of \cite{Bunge-Gago-SanLuis:Synthetic.Differential.Topology}.
\end{rmk}

\begin{rmk}
Note that by the Archimedian property (Postulate O.5), a number $x$ is infinitesimal if and only if $x < \frac{1}{n}$ for all $n : \Nb$.
\end{rmk}

\begin{lem}[\cite{Penon:Infinitesimals.and.Intuitionism}]\label{lem:neighbor.relation.functorial}
  Any function $f : X \to Y$ preserves the neighbor relation, in that we have a map
  $$f_{\ast} : (x \approx y) \to (fx \approx fy).$$
  As a corollary, for any point $x : X$ there is a pushforward
  $$f_{\ast} : \mathfrak{D}_{x}X \to \mathfrak{D}_{fx}Y.$$
  \end{lem}
  \begin{proof}
We apply $\neg\neg$ functorially to $\ap\, f$.
  \end{proof}

The infinitesimal neighborhoods of $0$ in $\Rb^{n}$ consist of the points with infinitesimal coordinates.
  \begin{lem}\label{lem:neighborhood.of.0.in.R.n}
    For any $n : \Nb$, we have an equality of subsets of $\Rb^{n}$:
    $$\Dc_{0}(\Rb^{n}) =  \mathfrak{D}^{n}.$$
  \end{lem}
  \begin{proof}
    We prove both inclusions. Going from the left hand side to the right hand side is straightforward. Suppose that $\vec{x} \approx 0$, and let $x_{i}$ be its $i^{\text{th}}$ coefficient. If $x_{i} \neq 0$, then $\vec{x} \neq 0$ since $\vec{x} = 0$ if and only if all its coefficients are $0$; therefore, we conclude that $x_{i}$ is not non-zero, which is to say that $x_{i} \approx 0$.

    It's the other direction that requires Postulate K. Suppose that each coefficient of $\vec{x}$ is not non-zero. Suppose that $\vec{x}$ were non-zero; since $\vec{x} = 0$ means precisely that all of its coefficients are $0$, we are assuming
    \[
\neg \left(\bigwedge_{i = 1}^{n} (x_{i} = 0)\right).
\]
By Postulate K, we may therefore conclude that one of the $x_{i}$ is invertible, which in particular means that it is non-zero. But this contradicts our assumption that each of the $x_{i}$ are not non-zero, so we conclude that $\vec{x}$ is not non-zero, which is to say that $\vec{x} \approx 0$.
  \end{proof}
  \begin{rmk}
If we assume that the smooth reals $\Rb$ are a field in the sense that every non-zero element is invertible, we can see \cref{lem:neighborhood.of.0.in.R.n} as a reformulation of Postulate K.
  \end{rmk}

Infinitesimal varieties are the zero locuses of functions on infinitesimals.
\begin{lem}\label{lem:inf.variety.is.zero.locus}
Let $V$ be an infinitesimal variety. Then there is merely a (polynomial) function $f : \Dc^{n} \to \Dc^{m}$ which send $0$ to $0$ and for which $V \simeq \{x : \Dc^{n} \mid f(x) = 0\}$, identifying the base point of $V$ with $0$.
\end{lem}
\begin{proof}
Let $V = \Spec(W)$ be the spectrum of a Weil algebra in standard form $W = \Rb[x_{1}, \ldots, x_{n}]/(f_{1}, \ldots, f_{m}) \xto{\term{ev}_{0}} \Rb$, which we may assume by \cref{lem:weil.alg.standard.form}. The polynomials $f_{1}, \ldots, f_{m}$ assemble into a function $f : \Rb^{n} \to \Rb^{m}$. We note that since $\term{ev}_{0}(f_{j}) = 0$, we have that $f(0) = 0$. We may restrict these functions to $\{a : \Rb^{n} \mid a \approx 0\} \to \{b : \Rb^{m} \mid b \approx 0\}$. We will show that $V$ is equivalent to the fiber over $0$ of this restricted function.

To show that $V \simeq \{a : \Dc^{n} \mid f(a) = f(0)\}$, we'll construct an explicit equivalence. Let $a : \Dc^{n}$ so that $f(a) = 0$, and define $\varphi_{a} : W \to \Rb$ by $\varphi_{a}(x_{i}) = a_{i}$. For this to be well defined, we need to know that $f_{j}(a) = 0$, but this was presumed. This gives a map $\{a : \Dc^{n} \mid f(a)= 0\} \to \Spec(W)$ which is evidently injective. To show surjectivity, suppose $\psi : W \to \Rb$ is a $\Rb$-homomorphism; we will show that $(\psi(x_{1}), \ldots, \psi(x_{n})) \in \{a : \Dc^{n} \mid f(a) = 0\}$, splitting the map $a \mapsto \varphi_{a}$. First, we note that $f(\psi(x_{1}), \ldots \psi(x_{n})) = \psi(f)(x_{1}, \ldots, x_{n}) = 0$. It remains to show that $\psi(x_{i}) \approx 0$. Note that since $\term{ev}_{0}(x_{i}) =0$, $x_{i}$ is in the augmentation $\ker(\term{ev}_{0})$ which was assumed to be nilpotent. Therefore, $x_{i}$ is nilpotent, and so $\psi(x_{i})$ is nilpotent, so it cannot be non-zero.
\end{proof}

We can now define microlinear types, though it will take a bit more work to explain why they are a useful class of types.
  \begin{defn}
A square as on the left is said to be an $X$-pushout (for a given type $X$) if the square on the right given by precomposition is a pullback:
\[
\begin{tikzcd}
	A & C && {X^A} & {X^C} \\
	B & D && {X^B} & {X^D}
	\arrow[from=1-1, to=2-1]
	\arrow[from=2-1, to=2-2]
	\arrow[from=1-1, to=1-2]
	\arrow[from=1-2, to=2-2]
	\arrow[from=1-5, to=1-4]
	\arrow[from=2-5, to=1-5]
	\arrow[from=2-5, to=2-4]
	\arrow[from=2-4, to=1-4]
\end{tikzcd}
\]
  \end{defn}

  \begin{defn}
An \emph{infinitesimal $\Rb$-pushout} is a \emph{crisp} commuting square of pointed maps between infinitesimal varieties which is an $\Rb$-pushout.
  \end{defn}

  \begin{defn}[Standard, see \cite{Lavendhomme:SDG}]\label{defn:microlinear}
    A type $X$ is \emph{microlinear} if every infinitesimal $\Rb$-pushout is also an $X$-pushout.
  \end{defn}

We will show that the tangent spaces of microlinear types have a canonical $\Rb$-module structure. We will do this by showing that for any $x : X$, the functor $V \mapsto \dsum{v : V \to X}(v(0) = x) : \type{InfVar}\op \to \Type$ restricts to a product preserving functor $\type{fgFreeMod}_{\Rb}^{\op} \to \Type$ sending $\Rb$ to $T_{x}X$. This shows that $T_{x}X$ is a model of the algebraic theory of $\Rb$-modules.

First, we begin by defining the first order infinitesimal patches of the origin in $\Rb^{n}$.
\begin{defn}[Standard, see \cite{Lavendhomme:SDG}]
  The first order infinitesimal patch of the origin $\Db(n)$ in $\Rb^{n}$ is
  \[
\Db(n) :\equiv \{x : \Rb^{n} \mid \forall i,j,\, x_{i}x_{j} = 0\} = \{x : \Rb^{n} \mid xx^{T} = 0\}.
  \]
  Note that $\Db(1) \equiv \Db$ is the set of first order infinitesimals.
\end{defn}

\begin{lem}
There is a fully faithful functor $\Db : \type{fgFreeMod})_{\Rb} \to \type{InfVar}$ sending $\Rb^{n}$ to $\Db(n)$.
\end{lem}
\begin{proof}
  First, we will show that the object assignment $\Rb^{n} \mapsto \Db(n)$ is functorial simply by restricting a linear function $f : \Rb^{n} \to \Rb^{m}$ to $\Db(n)$. If $x \in \Db(n)$, we will show that $f(x) \in \Rb^{m}$. By the Kock-Lawvere axiom, we have that
  \[
f(x) = f(0) + \sum_{i} x_{i}A_{i}
\]
for unique vectors $A_{i} : \Rb^{m}$. By linearity, $f(0) = 0$, so we see that $f(x) = Ax$ where $A$ is the $m \times n$ matrix with columns $A_{i}$. Finally,
\[
 (Ax)(Ax)^{T} = Axx^{T}A^{T} = 0.
\]
because by hypothesis $xx^{T} = 0$.

To show that this assignment is fully faithful, we need to show that any pointed function $f : \Db(n) \to \Db(m)$ extends to a unique linear map $\Rb^{n} \to \Rb^{m}$. The key is again to use the Kock-Lawvere axiom to see that $f(x) = f(0) + Ax$ for a unique matrix $A$, and since we are only looking at pointed functions we know that $f(0) = 0$. Therefore we have a matrix $A$, which gives us a linear map $\Rb^{n} \to \Rb^{m}$.
\end{proof}

\begin{lem}[Proposition 2.2.6 of \cite{Lavendhomme:SDG}]\label{lem:linear.inf.pushout}
  The square
\begin{equation}\label{diag:linear.inf.pushout}
\begin{tikzcd}
	\ast & {\Db(m)} \\
	{\Db(n)} & {\Db(n + m)}
	\arrow["0"', from=1-1, to=2-1]
	\arrow["{x \mapsto (x, 0)}"', from=2-1, to=2-2]
	\arrow["0", from=1-1, to=1-2]
	\arrow["{y \mapsto (0, y)}", from=1-2, to=2-2]
\end{tikzcd}
\end{equation}

 is an infinitesimal $\Rb$-pushout. As a corollary, the functor $\Db : \type{fgFreeMod}_{\Rb} \to \type{InfVar}$ sends coproducts to infinitesimal $\Rb$-pushouts.
\end{lem}
\begin{proof}
  Suppose $f : \Db(n) \to \Rb$ and $g : \Db(m) \to \Rb$ are such that $f(0) = g(0)$. By the Kock-Lawvere axiom, we have that $f(x) = f(0) + a \cdot x$ and $g(y) = g(0) + b \cdot y$ for unique vectors $a : \Rb^{n}$ and $b : \Rb^{m}$.

  Define $\langle f, g \rangle : \Db(n + m) \to \Rb$ by $\langle f, g \rangle(x, y) :\equiv f(0) + a \cdot x + b \cdot y$. By definition, $\langle f, g \rangle(x, 0) = f(x)$ and $\langle f, g \rangle(0, y) = g(y)$. The uniqueness part of the Kock-Lawvere axiom shows that this extension is unique.

  Finally, we note that every coproduct diagram in $\type{fgFreeMod}_{\Rb}$ is given by the inclusion of axis as in the above square, at least up to isomorphism.
\end{proof}

\cref{lem:linear.inf.pushout} shows that for any microlinear type $X$, the square
\[
\begin{tikzcd}
	{X^{\Db(n + m)}} & {X^{\Db(n)}} \\
	{X^{\Db(m)}} & X
	\arrow[from=1-1, to=2-1]
	\arrow[from=1-1, to=1-2]
	\arrow[from=1-2, to=2-2]
	\arrow[from=2-1, to=2-2]
	\arrow["\lrcorner"{anchor=center, pos=0.125}, draw=none, from=1-1, to=2-2]
\end{tikzcd}
\]
is a pullback. This condition is weaker than microlinearity, but useful in its own right. It appears as Definition 6.3 in \cite{Kock:SDG}.
\begin{defn}[\cite{Kock:SDG}]\label{defn:inf.linear}
A type $X$ is \emph{infinitesimally linear} if the squares
\[
\begin{tikzcd}
	{X^{\Db(n + m)}} & {X^{\Db(n)}} \\
	{X^{\Db(m)}} & X
	\arrow[from=1-1, to=2-1]
	\arrow[from=1-1, to=1-2]
	\arrow[from=1-2, to=2-2]
	\arrow[from=2-1, to=2-2]
	\arrow["\lrcorner"{anchor=center, pos=0.125}, draw=none, from=1-1, to=2-2]
\end{tikzcd}
\]
give by precomposing by the $\Rb$-pushout squares in \cref{lem:linear.inf.pushout} are pullbacks for all $n$ and $m$.
\end{defn}

\begin{rmk}
As an immediate corollary of \cref{lem:linear.inf.pushout}, we see that microlinear types are infinitesimally linear. It is somewhat unfortunate that these two notions have such similar names; it is really ``microlinearity'' which should be called something else, since it also involves higher order infinitesimals.
\end{rmk}

\begin{rmk}
While Kock \cite{Kock:SDG} and Lavendhomme \cite{Lavendhomme:SDG} use ``microlinear'' and ``infinitesimally linear'' in the same sense that we are using here, Bunge, Gago, and San Luis \cite{Bunge-Gago-SanLuis:Synthetic.Differential.Topology} use the term ``infinitesimally linear'' to refer to what we here call ``microlinear''.
\end{rmk}

\begin{thm}\label{thm:inf.linear.R.module}
If $X$ is infinitesimally linear (in particular if $X$ is microlinear), then for all $x : X$, the tangent space $T_{x} X$ has the (coherent) structure of an $\Rb$-module --- even if $X$ is not a set.
\end{thm}
\begin{proof}
  If $X$ is microlinear, then the functor
  \[
V \mapsto (\type{ev}_{\pt} : X^{V} \to X) : \type{InfVar}\op \to \Type_{/X}
\]
sends infinitesimal $\Rb$-pushouts to pullbacks. The functor $\type{fib}_{(-)}(x) : \Type_{/X} \to \Type$ taking the fiber over $x$ preserves pullbacks, and so the composite functor $V \mapsto \dsum{v : V \to X}(v(0) = x)$ sends infinitesimal $\Rb$-pushouts to pullbacks. Since the functor $\Db : \type{fgFreeMod}_{\Rb} \to \type{InfVar}$ sends coproducts to infinitesimal $\Rb$-pushouts, the composite functor $\Rb^{n} \mapsto \dsum{v : \Db(n) \to X}(v(0) = x)$ preserves products. This makes $T_{x}X :\equiv \dsum{ v : \Db(1) \to X }(v(0) = x)$ into a model of the Lawvere theory of $\Rb$-algebras, which was to be shown.
\end{proof}

\subsection{The crystaline modality $\Im$}\label{sec:crystaline.modality}

Just as we used the shape modality $\shape$, defined by nullifying the reals $\Rb$, to study the homotopy theory of orbifolds, we can use the \emph{crystaline} modality $\Im$, defined by nullifying the infinitesimals $\mathfrak{D}$, to study the differential structure of orbifolds.

In this section, we will introduce the $\Im$ modality (and in particular the  $\Im$-{\'e}tale maps). We will then prove a crucial result: microlinearity descends along surjective $\Im$-{\'e}tale maps. In the next section, we will use this theorem to show that all reasonable notions of ``smooth space'' are microlinear.

\begin{defn}\label{defn:crystaline.modality}
  We define the \emph{crystaline} modality $\Im$ to be nullification at $\Dc :\equiv \{x : \Rb \mid x \approx 0\}$. We refer to the $\Im$-modal types as \emph{crystaline}.
\end{defn}

\begin{rmk}
The crystaline modality $\Im$ appears in \cite{Schreiber:Differential.Cohomology} as the \emph{infinitesimal shape} modality. The type $\Im X$ is sometimes called the ``de Rham stack'' of $X$. It was studied in homotopy type theory in \cite{Cherubini:Thesis} (see also the appendix to \cite{Cherubini-Rijke:Modal.Descent}). I have decided to call it the \emph{crystaline} cohomology because a map $E : \Im X \to \type{Vect}$ is a \emph{crystal} on $X$, in the sense of \cite{Lurie:Crystal}.

In the setting of \cite{Schreiber:Differential.Cohomology}, $\Im$ may be equivalently defined as the nullification of all crisp infinitesimal varieties. But in the Dubuc $\infty$-topos, it is not clear that these two modalities coincide; we have gone for the definition which works best in the Dubuc $\infty$-topos.
\end{rmk}

Arguing externally, we could use the tinyness of $\Dc$ to prove that $\Im$ has an external right adjoint. However, in our intended models, $\Im$ also has an external left adjoint, and so in particular preserves crisp pullbacks. We take this preservation property as an axiom.
\begin{axiom}
  The $\Im$ modality preserves crisp pullbacks.
\end{axiom}
\begin{rmk}
  In the intendend models, $\Im$ is in fact part of an (external) adjoint triple
  \[
\type{R} \dashv \Im \dashv \&
  \]
  which Schreiber refers to as \emph{infinitesimal cohesion}. Here, both $\type{R}$ and $\&$ are \emph{comodalities}, or comonadic modalities. To add nontrivial comodalities requires modifying the underlying type theory, so we refrain from that here. The comodality $\&$ behaves very much like $\flat$, and adjunction $\Im \dashv \&$ mirrors that of $\shape \dashv \flat$. For this reason, a modification of Shulman's crisp type theory to accomodate $\&$ should be rather straightforward.
  But the comodality $\type{R}$ is not lex (even externally), and so a different method would be necessary to incorporate it into the type theory. We make use of neither here.
  \end{rmk}

The $\Im$-connected types --- those types $X$ for which $\Im X$ is contractible --- are a good class of \emph{infinitesimal types}. We can show that the crisp infinitesimal varieties are $\Im$-connected, using the crisp lexness of $\Im$.
\begin{lem}\label{lem:inf.variety.is.im.connected}
Every crisp infinetismal variety is $\Im$-connected.
\end{lem}
\begin{proof}
By \cref{lem:inf.variety.is.zero.locus}, every infinitesimal variety is the zero-locus of a map between powers of $\Dc$. Since $\Im$ is lex for crisp diagrams and powers of $\Dc$ are $\Im$-connected, the fiber of a crisp map $f :: \Dc^{n} \to \Dc^{m}$ over the crisp element $0 :: \Dc^{m}$ is $\Im$-connected.
\end{proof}

We will mainly use the modality $\Im$ for its {\'e}tale maps. Recall from \cref{defn:modal.etale} that a map $f : X \to Y$ is $\Im$-\'etale if and only if the $\Im$-naturality square is a pullback:
\[
\begin{tikzcd}
	X & {\Im X} \\
	Y & {\Im Y}
	\arrow["f"', from=1-1, to=2-1]
	\arrow["{(-)^{\Im}}", from=1-1, to=1-2]
	\arrow["{(-)^{\Im}}"', from=2-1, to=2-2]
	\arrow["{\Im f}", from=1-2, to=2-2]
	\arrow["\lrcorner"{anchor=center, pos=0.125}, draw=none, from=1-1, to=2-2]
\end{tikzcd}
\]
The $\Im$-{\'e}tale maps $f : X \to Y$ (see ) are a reasonable class of ``local diffeomorphisms'' because they lift uniquely against the base points $\pt : \ast \to V$ for any crisp infinitesimal variety $V$:
\[
\begin{tikzcd}
	\ast & C \\
	V & X
	\arrow["0"', from=1-1, to=2-1]
	\arrow["\forall"', from=2-1, to=2-2]
	\arrow["\pi", from=1-2, to=2-2]
	\arrow["\forall", from=1-1, to=1-2]
	\arrow["{\exists!}", dashed, from=2-1, to=1-2]
\end{tikzcd}
\]
In particular, an $\Im$-{\'e}tale map $f : X \to Y$ induces an equivalence $T_{x}X \xto{\sim} T_{fx}Y$ on tangent spaces for all $x : X$.
\begin{prop}\label{prop:etale.means.iso.on.tangent.spaces}
Let $f : X \to Y$ be $\Im$-{\'e}tale. Then the pushforward $f_{\ast} : T_{x}^{V}X \to T_{fx}^{V}Y$ on the $V$-tangent space for any crisp infinitesimal variety $V$ and $x : X$ is an equivalence.
\end{prop}
\begin{proof}
 Let $x : X$ and $(v, w) : T_{fx}^{V}Y$. Consider the following diagram:
 \begin{equation}\label{diag:etale.lifting.square}
\begin{tikzcd}
	\ast & C \\
	V & X
	\arrow["0"', from=1-1, to=2-1]
	\arrow["v"', from=2-1, to=2-2]
	\arrow["\pi", from=1-2, to=2-2]
	\arrow["x", from=1-1, to=1-2]
	\arrow["{\exists!}", dashed, from=2-1, to=1-2]
\end{tikzcd}
\end{equation}

The square commutes by the witness $w : v(0) = fx$ that $v$ sends $0$ to $fx$. The type of fillers to this square is
\[
\dsum{\tilde{v} : V \to X} \dsum{\tilde{w} : \tilde{v}(0) = x} \dsum{q : f \circ \tilde{v} = v} ((q\inv \at 0)\bullet f_{\ast}(\tilde{w}) = w).
\]
This is equivalent to $\fib_{f_{\ast}}(v, w)$ of $f_{\ast} : T_{x}^{V} X \to T_{fx}^{V}Y$ over $(v, w)$:
\begin{align*}
\fib_{f_{\ast}}(v, w) &\equiv \dsum{(\tilde{v},\tilde{w}) : T_{x}^{V}X} ((f \circ \tilde{v}, f_{\ast}\tilde{w}) = (v, w)). \\
                &= \dsum{(\tilde{v}, \tilde{w}) : T_{x}^{V}} \dsum{q : f \circ \tilde{v} = v} (\tr(\lambda z. z(0) = fx,\, q)(f_{\ast}\tilde{w}) = w).
  \end{align*}
By \cref{lem:inf.variety.is.im.connected}, the base point inclusion $\pt : \ast \to V$ is a $\Im$-equivalence. Therefore, the type of fillers to (\ref{diag:etale.lifting.square}) is contractible since $f$ is $\Im$-{\'e}tale. By the equivalence between $\fib_{f_{\ast}}(v, w)$ and the type of fillers, we see that $\fib_{f_{\ast}}(v, w)$ is contractible, showing that $f_{\ast} : T_{x}^{V}X \to T_{fx}^{V}Y$ is an equivalence.
\end{proof}

\begin{rmk}
It would be really useful to know that the $\Im$-\'etale maps are exactly those that lift on the right against the base point inclusion $0 : \ast \to \Dc$. But I do not know under which conditions the \'etale maps of a modality given by nullifying at a pointed type are defined by lifting against the inclusion of the base point. In Theorem 3.10 of \cite{Cherubini-Rijke:Modal.Descent}, Cherubini and Rijke show that this holds for the $n$-truncation modality, and ask the question in general.
\end{rmk}

Let's find some ways to construct $\Im$-{\'e}tale maps. Firstly, any covering map is $\Im$-{\'e}tale. Even more generally, every $\shape$-modal map is $\Im$-{\'e}tale.
\begin{lem}\label{lem:discrete.is.crystaline}
  Any discrete ($\shape$-modal) type  is crystaline ($\Im$-modal). As a corollary:
  \begin{enumerate}
    \item Any $\shape$-modal map is $\Im$-modal.
          \item Any $\shape$-{\'e}tale map is $\Im$-{\'e}tale.
    \item Any $\Im$-connected map is $\shape$-connected.
          \item Any $\Im$-equivalence is a $\shape$-equivalence.
\end{enumerate}
\end{lem}
\begin{proof}
This follows from the fact that $\Dc$ is $\shape$-connected, since it admits a multiplicative action by $\Rb$ giving us an explicit contraction $x \mapsto tx$ onto $0$. Therefore, if $X$ is discrete --- which is to say, $\shape$-modal --- then any map $\Dc \to X$ factors uniquely through the shape of $\Dc$, which is the point.

The corollaries follow from Theorem 3.17 of \cite{Jaz:Good.Fibrations}.
  \end{proof}

  \begin{rmk}
    Combining \cref{lem:discrete.is.crystaline} with \cref{prop:etale.means.iso.on.tangent.spaces} shows us that $q : \mathfrak{h} \to \Ma_{1,1}$ is $\Im$-{\'e}tale and that therefore the induces an isomorphism on tangent spaces. In particular, we can say that $\Ma_{1, 1}$ is a $2$-dimensional orbifold.
    \end{rmk}

  \begin{warn}\label{rmk:crystaline.modality.is.not.lex}
Thanks to \cref{lem:discrete.is.crystaline}, we can prove that $\Im$ is \emph{not} lex. Since all propositions are discrete (Lemma 8.8 of \cite{Shulman:Real.Cohesion}), all propositions are crystaline. This means that every embedding is $\Im$-modal. If $\Im$ were lex, this would mean that every embedding would be $\Im$-\'etale. But then the inclusion $\{0\} \hookrightarrow \Rb$ would be $\Im$-\'etale, which would imply that the inclusion $\Dc \hookrightarrow \Rb$ is constant at $0$ --- that is, every infinitesimal would be $0$. This obviously trivializes the theory.
  \end{warn}

Let's now turn our attention to the relationship between microlinearity and $\Im$-{\'e}tale maps. Since $\Im$-{\'e}tale maps are intuitively local diffeomorphisms and microlinearity is a local --- or even infinitesimal --- property, it stands to reason that if $f : X \to Y$ is $\Im$-{\'e}tale, then it should be the case that $X$ is microlinear if and only if $Y$ is. We will show that this is the case, as long as $f$ is surjective as well.

First we will show that microlinearity ascends along $\Im$-{\'e}tale maps. In order to prove this, we will need to re-express the notion of microlinearity as a lifting condition.
  \begin{lem}\label{lem:microlinear.lifting.property}
    A type $X$ is microlinear if and only if it lifts uniquely on the right against all codiagonal maps $\nabla_{V} : V_{2} +_{V_{1}} V_{3} \to V_{4}$ where
    \[
\begin{tikzcd}[sep = small]
	{V_1} && {V_3} \\
	& V \\
	{V_2} && {V_4}
	\arrow[from=1-1, to=3-1]
	\arrow[from=3-1, to=3-3]
	\arrow[from=1-1, to=1-3]
	\arrow[from=1-3, to=3-3]
\end{tikzcd}
    \]
    is an infinitesimal $\Rb$-pushout.
  \end{lem}
  \begin{proof}
The unique lifting condition says that pre-composition by $\nabla_{V}$ gives an equivalence between $X^{{V_{4}}}$ and $(V_{2} +_{V_{1}} V_{3} \to X)$. But by the universal property of the pushout, this latter type is equivalent to the pullback $X^{V_{2}} \times_{{X^{V_{1}}}} X^{{V_{3}}}$, and so $V$ is an $X$-pushout if and only if $X$ lifts uniquely against $\nabla_{V}$.
  \end{proof}

    \begin{lem}\label{lem:microlinear.ascends.along.etale}
If $Y$ is microlinear and $f : X \to Y$ is $\Im$-{\'e}tale, then $X$ is microlinear.
    \end{lem}
    \begin{proof}
Let $V$ be an infinitesimal $\Rb$-pushout. Then the codiagonal $\nabla_{V} : V_{2} +_{V_{1}} V_{3} \to V_{4}$ is a map between $\Im$-connected types, and is therefore an $\Im$-equivalence. Since $f$ is $\Im$-{\'e}tale, it therefore lifts against $\nabla_{V}$, making $f$ microlinear. Since the terminal map $X \to \ast$ is the composite $X \xto{f} Y \to \ast$ and by hypothesis $Y \to \ast$ was microlinear, $X \to \ast$ is microlinear as well (by closure under composition of maps defined by lifting on the right against a given class of maps).
    \end{proof}

Next, we will show that microlinearlity descends along surjective $\Im$-{\'e}tale maps. To do this we will need a general lemma.
      \begin{lem}\label{lem:modal.connected.closed.under.pushout}
For any modality $\modal$, the $\modal$-connected types are closed under pushout.
      \end{lem}
      \begin{proof}
        Consider a pushout square
        \[
          \begin{tikzcd}
            A \ar[r, "f"] \ar[d, "h"'] & B \ar[d, "k"] \\
            C \ar[r, "g"'] & D
          \end{tikzcd}
        \]
        in which $A$, $B$, and $C$ are $\modal$-connected. Then for any $\modal$-modal $X$, consider the following cube:
        \[
\begin{tikzcd}[sep=small]
	& X && X \\
	X && X \\
	& {X^D} & {} & {X^C} \\
	{X^B} && {X^A}
	\arrow[from=4-1, to=4-3]
	\arrow[from=3-4, to=4-3]
	\arrow[from=3-2, to=4-1]
	\arrow[from=3-2, to=3-4]
	\arrow[from=1-2, to=2-1, equals]
	\arrow[from=1-2, to=1-4, equals]
	\arrow[from=1-4, to=2-3, equals]
	\arrow[from=2-1, to=4-1]
	\arrow[from=1-2, to=3-2]
	\arrow[from=1-4, to=3-4]
	\arrow[from=2-1, to=2-3, crossing over, equals]
	\arrow[from=2-3, to=4-3, crossing over]
\end{tikzcd}
        \]
        Since the square we began with was a pushout, the bottom face is a pullback. The top face is also trivially a pullback. Since $A$, $B$, and $C$ are $\modal$-connected, the front three vertical edges are equivalences. Therefore, the back edge is also an equivalence, which proves that $D$ is also $\modal$-connected.
        \end{proof}

\begin{thm}\label{thm:microlinear.descends.along.etale}
Let $f : X \to Y$ be a surjective, $\Im$-{\'e}tale map. If $X$ is microlinear, then so is $Y$.
  \end{thm}
  \begin{proof}
    We will show that $Y$ lifts on the right against the gap maps $\nabla_{V} : V_{2} +_{V_{1}} V_{3} \to V_{4}$ of infinitesimal $\Rb$-pushouts. This means showing that for any $k : V_{2} +_{V_1} V_{3} \to Y$, the type $\dsum{g : V_{4} \to Y}(g \circ \nabla_{V} = k)$ is contractible. As we are seeking to prove a proposition and $f$ is surjective, we may assume we have $(x, p) : \fib_{f}(k(0))$. Then, since by \cref{lem:modal.connected.closed.under.pushout} the inclusion $0 : \ast \to V_{2} +_{V_{1}} V_{3}$ is a $\Im$-equivalence, we have a unique dashed filler in the following square:
    \[
\begin{tikzcd}
	\ast & X \\
	{V_2 +_{V_1} V_3} & Y
	\arrow["0"', from=1-1, to=2-1]
	\arrow["x", from=1-1, to=1-2]
	\arrow["k"', from=2-1, to=2-2]
	\arrow["f", from=1-2, to=2-2]
	\arrow["{\tilde{k}}", dashed, from=2-1, to=1-2]
\end{tikzcd}
\]
Call the bottom commutative triangle $\beta : f \circ \tilde{k} = k$.
    We will show that
    $\dsum{g : V_{4} \to Y}(g \circ \nabla_{V} = k)$ is equivalent to the type $\dsum{\tilde{g} : V_{4} \to X}(\tilde{g} \circ\nabla_{V} = \tilde{k})$. Since this latter type is contractible by the assumption that $X$ is microlinear, this will complete our proof.

    Define $\varphi : \dsum{\tilde{g} : V_{4} \to X}(\tilde{g} \circ\nabla_{V} = \tilde{k}) \to \dsum{g : V_{4} \to Y}(g \circ \nabla_{V} = k)$ by
    $$(\tilde{g}, \tilde{w}) \mapsto (f \circ \tilde{g}, (f \circ)_{\ast}\tilde{w} \bullet \beta).$$
    We will show that the fibers of this map are contractible. We begin with a calculation:
    \begin{align*}
      \fib_{\varphi}(g, w) &= \dsum{\tilde{g} : V_{4} \to X} \dsum{\tilde{w} : \tilde{g} \circ \nabla_{V} = \tilde{k}} ((f \circ \tilde{g}, (f \circ)_{\ast}\tilde{w} \bullet \beta)=(g, w))\\
      &=\dsum{\tilde{g} : V_{4} \to X} \dsum{\tilde{w} : \tilde{g} \circ \nabla_{V} = \tilde{k}} \dsum{z : f \circ \tilde{g} = g} (\tr\, (\lambda h. h \circ \nabla_{V} = k)\, z\, ((f\circ)_{\ast}\tilde{w}\bullet \beta) = w)\\
      &=\dsum{\tilde{g} : V_{4} \to X} \dsum{\tilde{w} : \tilde{g} \circ \nabla_{V} = \tilde{k}} \dsum{z : f \circ \tilde{g} = g} ((\circ\nabla_{V})_{\ast}z \bullet (f\circ)_{\ast}\tilde{w} \bullet \beta = w)\\
      &=\dsum{\tilde{g} : V_{4} \to X} \dsum{\tilde{w} : \tilde{g} \circ \nabla_{V} = \tilde{k}} \dsum{z : f \circ \tilde{g} = g} ((\circ\nabla_{V})_{\ast}z \bullet (f\circ)_{\ast}\tilde{w}= w \bullet \beta\inv)
    \end{align*}
    This final type is the type of fillers to the square:
    \[
\begin{tikzcd}
	{V_2 +_{V_1} V_3} & X \\
	{V_4} & Y
	\arrow["{\nabla_V}"', from=1-1, to=2-1]
	\arrow["{\tilde{k}}", from=1-1, to=1-2]
	\arrow["g"', from=2-1, to=2-2]
	\arrow["f", from=1-2, to=2-2]
	\arrow["{\tilde{g}}", dashed, from=2-1, to=1-2]
\end{tikzcd}
    \]
    where the underlying square commutes by $w \bullet \beta\inv : g \circ \nabla_{V} = f \circ \tilde{k}$. But since $\nabla_{V}$ is an $\Im$-equivalence and $f$ is $\Im$-{\'e}tale, there is a unique filler to this square. Therefore, $\fib_{\varphi}(g, w)$ is contractible.
  \end{proof}

  \section{Smooth Spaces are Microlinear}\label{sec:notions.of.smooth.space}

  In this section, we will use \cref{thm:microlinear.descends.along.etale} to show that most reasonable notions of smooth space in synthetic differential geometry are microlinear. Our main theorem will be \cref{thm:etale.groupoid.microlinear}, which proves that \'etale groupoids are microlinear.

  We will begin in \cref{sec:ordinary.manifold} by recalling the definition of a manifold in the ordinary sense, which we will call ordinary manifolds for emphasis. We will make use of ordinary manifolds in our final theorem, \cref{thm:ordinary.proper.etale.groupoid.is.orbifold}, which proves that crisp ordinary proper \'etale groupoids --- which are proper \'etale groupoids in the ordinary sense --- are orbifolds in the sense of \cref{defn:orbifold}.

  In \cref{thm:manifold.is.microlinear}, we will show that ordinary manifolds are microlinear. We will do this by showing that any ordinary manifold is a \emph{Penon manifold} (\cref{defn:Penon.manifold}), and then by showing that all Penon manifolds are microlinear --- a result which is likely folklore. We will also take the opportunity in \cref{thm:Penon.etale.is.im.etale} to show that Penon's notion of \'etale map coincides with that of $\Im$-\'etale maps, at least between crisp Penon manifolds.

  In \cref{sec:Schreiber.manifold}, we will recall Schreiber's notion of $V$-manifold and show in \cref{thm:Schreiber.manifold.microlinear} that any Schreiber $V$-manifold for microlinear $V$ is itself microlinear. Unlike ordinary manifolds and Penon manifolds, which must be sets, Schreiber manifolds can be higher types. This is therefore our first taste of higher microlinear types.

  We will fully turn our attention to higher microlinear types in \cref{sec:etale.groupoid}. In \cref{thm:discrete.homotopy.quotient.is.microlinear}, we will use \cref{thm:microlinear.descends.along.etale} to quickly prove that the quotient $X \sslash\, \Gamma$ of a microlinear type $X$ by the action of a crisply discrete higher group $\Gamma$ is microlinear. This is enough to show that all of our examples of good orbifolds from \cref{sec:orbifold.examples} are microlinear. But not all orbifolds are good orbifolds. In order to justify \cref{defn:orbifold} and show that any orbifold in the ordinary sense --- that is, any proper \'etale groupoid --- is microlinear, we prove in \cref{thm:etale.groupoid.microlinear} that crisp \'etale groupoids are microlinear. This provides one pillar of the upcoming \cref{thm:ordinary.proper.etale.groupoid.is.orbifold} which shows that all crisp, ordinary proper \'etale groupoids are orbifolds in the sense of \cref{defn:orbifold}.

  The proof of \cref{thm:etale.groupoid.microlinear} is not trivial. It relies on a form of \'etale descent, \cref{thm:etale.descent}, which proves that if $f :: X \to Y$ is crisp and surjective and its pullback along itself is $\Im$-\'etale, then it is $\Im$-\'etale. This makes essential use of the commutation of $\Im$ with crisp pushouts and colimits of crisp sequences, a fact which appears in this section as \cref{thm:Im.preserves.crisp.colimits} but whose proof is deffered to \cref{sec:tiny}.

Finally, we show in \cref{sec:deloopings.inf.linear} that a delooping of an infinitesimally linear group (such as a Lie group) is itself infinitesimally linear. This gives examples of infinitesimally linear higher types which are not locally discrete. However, the method we use to prove \cref{thm:inf.linear.BG} does not enable us to show that a delooping of a microlinear group is microlinear. I have not been able to prove this, nor have I come up with a counterexample.

Before even starting, let's observe a number of well known closure properties of microlinear types. This will suffice to show that a wide number of spaces encountered in practice are microlinear, without relying on any other notions of smooth space.
\begin{prop}
Microlinear types are closed under limits, and if $X$ is microlinear then $X^{A}$ is microlinear for any $A$.
\end{prop}
\begin{proof}
By \cref{lem:microlinear.lifting.property}, microlinearity is characterized by lifting uniquely on the right against a class of maps. These closure properties are the closure properties of the class of maps defined by such a lifting property.
\end{proof}

As a corollary, we see that the zero loci of (arbitrarily indexed) families of real valued functions are microlinear.
\begin{prop}
Let $f : \Rb^{N} \to \Rb^{M}$ be any map where $N$ and $M$ are any types. Then the zero-locus $Z_{f} :\equiv \{x : \Rb^{N} \mid f(x) = 0\}$ is microlinear.
\end{prop}
\begin{proof}
  As $\Rb$ is microlinear, $\Rb^{N}$ and $\Rb^{M}$ are microlinear. The result then follows by the closure of microlinear types under pullback.
\end{proof}

\subsection{Ordinary manifolds}\label{sec:ordinary.manifold}

First we begin by giving a formulation of the standard notion of manifold in homotopy type theory. By the standard notion of manifold, I mean a second countable Hausdorff topological space which is locally homeomorphic to $\Rb^{n}$. For emphasis, we will always call these ``ordinary manifolds'' in this paper. Classically, this would only define a continuous manifold; but we are using the \emph{smooth reals}, so any transition function between charts will already be smooth.

Since we are working in a constructive setting, we should take a bit of care about the topological points. Namely, instead of asking that an ordinary manifold be Hausdorff ($T_{2}$), we will ask that it be regular Hausdorff ($T_{3}$) --- that is, regular and $T_{0}$, which together imply Hausdorff.

\begin{defn}
  We recall the following topological definitions:
  \begin{enumerate}
  \item A topological space $X$ is \emph{regular} if for any open set $U$ and $x \in U$, there is an open neighborhood $V$ of $U$ and an open set $G$ disjoint from $V$ (that is, $G \cap V = \emptyset$) but complementary with $U$ (that is, $G \cup U = X$).
          \item A topological space is $T_{0}$ if for any pair of \emph{distinct} points $x$ and $y$ (that is, $x \neq y$), there is either an open set containing $x$ but not $y$, or an open set containing $y$ but not $x$.
          \item A topological space is regular Hausdorff if it is regular and $T_{0}$.
          \item A topological space $X$ is Hausdorff if for any pair of \emph{distinct} points $x$ and $y$, there are disjoint open sets $U_{x} \cap U_{y} = \emptyset$ with $x \in U_{x}$ and $y \in U_{y}$.
          \end{enumerate}
\end{defn}

Let's recall the simple proof that regular Hausdorff spaces as defined above are also Hausdorff.
\begin{lem}
A regular Hausdorff topological space is Hausdorff.
\end{lem}
\begin{proof}
Let $x$ and $y$ be distinct points of a regular Hausdorff space $X$. Since $X$ is $T_{0}$, there is (without loss of generality) an open set $U$ containing $x$ but not $y$. By regularity, there is then an open set $V$ containing $x$ and an open set $G$ disjoint from $V$ and with $G \cup U = X$. Therefore $y$ is either in $G$ or $U$; but we already know it is not in $U$, so it must be in $G$. Therefore, there are disjoint open neighborhoods separating $x$ and $y$.
  \end{proof}

There is a good reason for using regular Hausdorff instead of just Hausdorff: open sets in regular spaces are \emph{infinitesimally stable} in the sense that if $x$ is in an open set $U$ and $x \approx y$, then $y$ is also in $U$. We will use this property in \cref{prop:ordinary.is.penon} to show that manifolds have the correct \emph{infinitesimal} structure in addition to the local structure which we assume.
\begin{lem}\label{lem:regular.is.inf.stable}
Let $X$ be a regular topological space. Then open sets $U$ of $X$ are infinitesimally stable in the sense that if $x \in U$ and $y \approx x$, then $y \in U$.
\end{lem}
\begin{proof}
  Let $x \in U$ and suppose that it is not the case that $x \neq y$. Since $X$ is regular, for $x \in U$ there is an open $V$ containing $x$ and an open $G$ disjoint from $V$ and complementary to $U$. Therefore, $y \in G$ or $U$. But $G$ is disjoint from $V$ and $x \in V$; if $y$ were in $G$, then $y$ would be distinct from $x$ (we would have $x \neq y$). Therefore, $y \in U$.
  \end{proof}

  Before going forward to define manifold, we should equip our Euclidean spaces with an appropriate topology and check that it is regular Hausdorff. We will use the metric topology.
  \begin{defn}
    A subset $U \subseteq \Rb^{n}$ is \emph{metrically open} if for every $x \in U$ there is a $\ep > 0$ (which may be chosen to be rational, by the Archimedian property) so that for all $y : X$, if
    $$\sum_{i = 1}^{n}(x_{i} - y_{i})^{2} < \ep^{2}$$
    then $y \in U$.

    Defining $B(x, \ep) :\equiv \{y : \Rb^{n} \mid \sum_{i = 1}^{n}(x_{i} - y_{i})^{2} < \ep^{2}\}$, this means that a set $U$ is metrically open if for all $x \in U$ there is an $\ep > 0$ for which $B(x, \ep) \subseteq U$.
  \end{defn}

  \begin{thm}
The metrically open subsets of $\Rb^{n}$ form a regular Hausdorff topology.
  \end{thm}
  \begin{proof}
     We explain the proof of the $1$-dimensional case to more clearly communicate the main ideas. First, we will show that $\Rb$ is $T_{0}$. Suppose that $x \neq y$ are real numbers; then either $x < y$ or $y < x$, so suppose the latter without loss of generality. Then $x - y \neq 0$ and is therefore invertible. Let $n$ be a natural number so that $\frac{1}{x - y} < n$; then $x - y < \frac{1}{n}$. Then the ball $B(x, \frac{1}{n})$ is a metrically open subset containing $x$ but not $y$.

Next, we will show that $\Rb$ is regular. Let $U$ be metrically open and let $x \in U$. Then there is an $\ep \in \Qb$ for which $B(x, \ep) \subseteq U$. Define $V :\equiv  B(x, \frac{\ep}{2})$. Define $G :\equiv \{y : \Rb \mid (x - y)^{2} > \left(\frac{\ep}{2}\right)^{2}\}$, and note that $G \cap V = \emptyset$. It remains to show that $G \cup U = \Rb$; it suffices to show that $G \cup B(x, \ep) = \Rb$. But this follows from Postulate O.3: since $\ep^{2} > \left(\frac{\ep}{2}\right)^{2}$, either $\ep^{2} >  (x - y)^{2}$ or $(x- y)^{2} >\left(\frac{\ep}{2}\right)^{2}$.
    \end{proof}

    Now we can define the notion of ``ordinary manifold''.
\begin{defn}
 A \emph{ordinary $n$-dimesional manifold} $M$ is a regular Hausdorff topological space which is a locally isomorphic to $\Rb^{n}$ in that for any point $p : M$, there is merely a \emph{chart} around $p$, an open subset $U \subseteq \Rb^{n}$ of the origin and an open embedding $\phi : U \hookrightarrow M$ so that $\phi(0) = p$. We will assume that ordinary manifolds are second countable in that they have a countable base of charts.
\end{defn}

In the next section, we will show that ordinary manifolds are Penon manifolds, and that Penon manifolds are microlinear.

\subsection{Penon manifolds}\label{sec:penon.manifold}

In his paper \emph{Infinitesimaux et Intuisionisme} \cite{Penon:Infinitesimals.and.Intuitionism}, Jacques Penon emphasizes that infinitesimal neighbors of a point $x$ of a space (for us, a set or a $0$-type) $X$ are the points $y$ which are \emph{not distinct} from $x$ in the sense that
$$\neg \neg (y = x).$$
This is the relaton $x \approx y$ of \cref{defn:infinitesimal}.

Accordinly, Penon suggests that a ``manifold'' should be a set which is infinitesimally isomorphic to $\Rb^{n}$ in the following sense.
\begin{defn}\label{defn:Penon.manifold}
  A \emph{Penon manifold} of dimension $n$ is a set $M$ so that for all $p : M$, the infinitesimal neighborhood $\Dc_{p}M \equiv \{x : M \mid x \approx p\}$ is identifiable with the infinitesimal neighborhood $\Dc^{n} \subseteq \Rb^{n}$ of the origin of real $n$-space as a pointed set. That is, for all $p$, we have
  \[
\trunc{\dsum{f : \Dc_{p}M = \Dc^{n}} (f(p) = 0)}.
  \]
  \end{defn}

  Ordinary manifolds are Penon manifolds.
  \begin{prop}\label{prop:ordinary.is.penon}
Every ordinary manifold is a Penon manifold.
  \end{prop}
  \begin{proof}
    Let $M$ be an ordinary manifold and let $p : M$ be a point, seeking to prove that $\Dc_p M$ is identifiable with $\Dc^{n} \subseteq \Rb^{n}$. Since we are seeking to prove a proposition, we may take as given a chart $\phi : U \hookrightarrow M$ around $p$ --- that is, with $\phi(0) = p$. Since $\phi$ is an open embedding, it's image $\phi(U)$ is open. Since $M$ is regular, $\Dc_p M \subseteq \phi(U)$ by \cref{lem:regular.is.inf.stable}. Since $\phi$ is an embedding, it restricts to an equivalence on its image; therefore, it also restricts to an equivalence $\phi : \phi\inv(\Dc_p M) \simeq \Dc_{p}$. It remains to show that $\phi\inv(\Dc_p M) = \Dc^{n}$. First, we note that $0 \in \phi\inv(\Dc_p M)$ since by assumption $\phi(0) = p$. Then, by \cref{lem:neighbor.relation.functorial}, the equivalence $\phi$ restricts to an equivalence $\Dc^{n}$ with $\Dc_p M$.
  \end{proof}

We will now show that Penon manifolds are microlinear by proving that microlinearity is a local property.
\begin{lem}\label{lem:microlinearity.is.local}
Let $X$ be a type. Then $X$ is microlinear if and only if for each $x : X$, the infinitesimal neighborhood $\Dc_{x}X :\equiv \dsum{y : X} (y \approx x)$ of $x$ is microlinear.
\end{lem}
\begin{proof}
  For any $x : X$, the projection $i : \Dc_{x}X \to X$ is an embedding since $(y \approx x)$ is a proposition. Furthermore, if $V_{4}$ is any infinitesimal variety, then since every $\ep : V_{4}$ is near $0$ --- $\ep \approx 0$ --- we have that $v(\ep) \approx v(0)$ for every $v : V_{4} \to X$. Therefore, we have a lift of $\nabla_{V}$ into $X$ for any infinitesimal $\Rb$-pushout if and only if we have a lift into $\Dc_{x}X$ where $x$ is the image of the base point $0$.
\end{proof}

\begin{thm}\label{thm:manifold.is.microlinear}
Every Penon manifold, and hence every ordinary manifold, is microlinear.
\end{thm}
\begin{proof}
Since for every point $p : M$ in a Penon manifold, the infinitesimal neighborhood $\Dc_{p}$ is identifiable with $\Dc^{n}$, it will suffice to show that $\Dc^{n}$ is microlinear by \cref{lem:microlinearity.is.local}. But $\Dc^{n} = \Dc_{0}( \Rb^{n} )$ and $\Rb^{n}$ is microlinear, so this also follows by \cref{lem:microlinearity.is.local}.
\end{proof}

In \cite{Penon:Infinitesimals.and.Intuitionism}, Penon gives a definition of \'etale map between his manifolds.
\begin{defn}
A map $f : X \to Y$ between sets is Penon \'etale if for all $x : X$, the induced map $f_{\ast} : \Dc_{x} \to \Dc_{fx}$ is an equivalence.
\end{defn}
We already have a notion of \'etale map: the $\Im$-\'etale maps. This is not to mention the ordinary notion of local diffeomorphism between manifolds: a map $f : X \to Y$ for which the pushforward $f_{\ast} : T_{x} X \to T_{fx} X$ is an isomorphism for all $x : X$. Luckily, all of these notions coincide where they are jointly defined, at least for crisp manifolds.

To prove this, we take a definition from Cherubini's \cite{Cherubini:Thesis}.
\begin{defn}[\cite{Cherubini:Thesis}]
  Let $X$ be a type and $x : X$ be an element. The $\Im$-disk around $x$ is defined to be the fiber of the unit $(-)^{\Im} : X \to \Im X$ over $x^{\Im}$:
  \[
D^{\Im}_{x}X :\equiv \fib_{(-)^{\Im}}(x^{\Im}) \equiv (\dsum{y : X} (x^{\Im} = y^{\Im})).
  \]
\end{defn}

Cherubini and Rijke show in Proposition 3.7 of \cite{Cherubini-Rijke:Modal.Descent} that if the modal unit $(-)^{\Im} : X \to \Im X$ is surjective, then a map $f : X \to Y$ is $\Im$-\'etale if and only if the induced map $f_{\ast} : D^{\Im}_{x}X \to D^{\Im}_{fx}Y$ is an equivalence. But as $\Im$ is given by localizing at a pointed type, all modal units are surjective.

We will show that in a crisp set $X$, the $\Im$-disk $D^{\Im}_{x}X$ around $x$ coincides with the infinitesimal disk $\Dc_{x}X$ around $x$. Together, this will show that a map between crisp Penon manifolds is Penon \'etale if and only if it is $\Im$-\'etale. However, we will need a few lemmas to do this. First, we need to know that for a crisp set $X$, $\Im X$ is also a set.
\begin{lem}\label{lem:im.of.crisp.set.is.set}
Let $X$ be a crisp set. Then $\Im X$ is a set.
\end{lem}
\begin{proof}
  A type $X$ is a set if and only if the square
  \[
\begin{tikzcd}
	X & X \\
	X & {X \times X}
	\arrow[from=1-1, to=2-1, equals]
	\arrow[from=1-1, to=1-2, equals]
	\arrow["\Delta"', from=2-1, to=2-2]
	\arrow["\Delta", from=1-2, to=2-2]
\end{tikzcd}
  \]
  is a pullback. If $X$ is a crisp set, then that square is a pullback, and since $\Im$ preserves crisp pullbacks (and binary products) so is the square
  \[
\begin{tikzcd}
	\Im X & \Im X \\
	\Im X & {\Im X \times \Im X}
	\arrow[from=1-1, to=2-1, equals]
	\arrow[from=1-1, to=1-2, equals]
	\arrow["\Delta"', from=2-1, to=2-2]
	\arrow["\Delta", from=1-2, to=2-2]
\end{tikzcd}
  \]
  which shows that $\Im X$ is a set.
\end{proof}

Next, we need to investigate the relationship between $\Im$ and the other modality of cohesion: the codiscrete modality $\sharp$ (see Section 3 of \cite{Shulman:Real.Cohesion}). We have not talked about $\sharp$ very much in this paper, and we won't need it for anything but this.
\begin{lem}\label{lem:codiscrete.is.crystaline}
Every codiscrete type is crystaline. As a corollary, the unit $(-)^{\sharp} : X \to \sharp X$ factors uniquely through the unit $(-)^{\Im} : X \to \Im X$, and the factor $\Im X \to \sharp X$ is itself a $\sharp$-counit.
\end{lem}
\begin{proof}
It suffices to show that $\Dc$ is $\sharp$-connected: $\sharp \Dc \simeq \ast$. For this, it suffices to show that $\flat \Dc \simeq \ast$ by Theorem 6.22 of \cite{Shulman:Real.Cohesion}. We note that $\flat \Dc$ is a set, and so by the crisp law of excluded middle, for any $u : \flat Dc$, either $u = 0^{\flat}$ or not. Suppose that $u \neq 0^{\flat}$; then, since $(-)_{\flat} : \flat \Dc \to \Dc$ is an embedding since $\Dc$ is a set by Theorem 8.21 of \cite{Shulman:Real.Cohesion}, it follows that $u_{\flat} \neq 0$. But every element of $\Dc$ is not distinct from $0$, so this is a contradiction and we may conclude that $u = 0^{\flat}$, so that $\flat \Dc$ is contractible.
\end{proof}

Finally, we can prove that infintesimal disks and $\Im$-disks coincide for Penon manifolds.
\begin{lem}\label{lem:crisp.Penon.manifold.disk}
  Let $X$ be a crisp Penon manifold. Then for $x : X$, we have
  \[
\Dc_{x}X \simeq D^{\Im}_{x}X
\]
over $X$.
\end{lem}
\begin{proof}
  By \cref{lem:im.of.crisp.set.is.set}, $\Im X$ is a set as well and so for any $x, y : X$, the type $(x^{\Im} = y^{\Im})$ is a proposition. Therefore, $D^{\Im}_{x}$ is a subset of $X$ for any $x : X$, so it suffices to show that for $y : X$ we have that $(x \approx y)$ if and only if $(x^{\Im} = y^{\Im})$.

  First, suppose that $x \approx y$, seeking to show the proposition $(x^{\Im} = y^{\Im})$. Since $X$ is a Penon manifold, choose a coordinate chart $\phi : \Dc^{n} \simeq \Dc_{x}X$, and note that the composite $\Dc^{n} \xto{\phi} \Dc_{x}X \hookrightarrow X \xto{(-)^{\Im}} \Im X$ is constant at $x^{\Im}$, which shows in particular that $(x^{\Im} = y^{\Im})$.

  On the other hand, suppose that $(x^{\Im} = y^{\Im})$. By \cref{lem:codiscrete.is.crystaline}, we have a map $\Im X \to \sharp X$, and this gives us a map $(x^{\Im} = y^{\Im}) \to (x^{\sharp} = y^{\sharp})$. By Theorem 3.7 of \cite{Shulman:Real.Cohesion}, we have an equivalence $(x^{\sharp} = y^{\sharp}) \simeq \sharp (x = y)$, and by Theorem 3.15 of \cite{Shulman:Real.Cohesion} we have an equivalence $\sharp(x = y) \simeq \neg\neg(x = y)$, which by definition was $(x \approx y)$. In total, we see that $(x^{\Im} = y^{\Im})$ implies $(x \approx y)$.
\end{proof}

As a corollary, we can finally deduce that Penon \'etale maps between crisp Penon manifolds are the same as $\Im$-\'etale maps.
\begin{thm}\label{thm:Penon.etale.is.im.etale}
Let $X$ and $Y$ be crisp Penon manifolds and let $f : X \to Y$.\footnote{Note, $f$ does not need to be crisp here, just the manifolds do.} Then $f$ is $\Im$-\'etale if and only if it is Penon \'etale.
\end{thm}
\begin{proof}
  By \cref{lem:crisp.Penon.manifold.disk}, for any $f : X \to Y$ between crisp Penon manifolds and $x : X$, we have a commuting square
  \[
\begin{tikzcd}
	{\Dc_xX} & {D^\Im_{x}X} \\
	{\Dc_{fx}Y} & {D^\Im_{fx}X}
	\arrow["\sim", from=1-1, to=1-2]
	\arrow["\sim"', from=2-1, to=2-2]
	\arrow["{f_\ast}"', from=1-1, to=2-1]
	\arrow["{f_\ast}", from=1-2, to=2-2]
\end{tikzcd}
  \]
  Therefore, the left vertical map is an equivalence if and only if the right vertical map is. The map $f$ is Penon \'etale if and only if the left vertical map is an equivalence, and the right vertical map is an equivalence if and only if $f$ is $\Im$-\'etale by Proposition 3.7 of \cite{Cherubini-Rijke:Modal.Descent} (noting that the modal units of $\Im$ are surjective, since it is given by nullifying a pointed type).
\end{proof}

\begin{cor}\label{cor:crisp.open.is.etale}
Let $S \subseteq X$ be a crisp subset of a (necessarily crisp) Penon manifold $X$. If $S$ is infinitesimally stable (meaning that if $x \in S$ and $x \approx y$, then $y \in S$) the inclusion $S\hookrightarrow X$ is $\Im$-\'etale.
 \end{cor}
 \begin{proof}
Infinitesimal stability implies that for any $x \in S$, the infinitesimal neighborhodd $\Dc_{x}S$ of $x$ in $S$ is equivalent to the infinitesimal neighborhood $\Dc_{x} X$ of $x$ in $X$. Therefore, the result follows by \cref{thm:Penon.etale.is.im.etale}.
  \end{proof}

Using the final proposition of \cite{Penon:Infinitesimals.and.Intuitionism}, we can prove as a corollary that the $\Im$-\'etale maps between the crisp ordinary manifolds are precisely the local diffeomorphisms in the usual sense. This does require an extra axiom which is considered by Penon (and proved to hold in the intended models by the same): the ``infinitesimal inverse function theorem''.
\begin{cor}\label{lem:etale.means.etale.for.ordinary.manifolds}
  Suppose that the ``infinitesimal inverse function theorem'' holds: for every $f : \Dc^{n} \to \Dc^{n}$ with $f(0) = 0$, if $f$ is invertible when restricted to the first order infinitesimals $\Db(n)$, then $f$ is invertible.

  Let $X$ and $Y$ be crisp Penon manifolds (or ordinary manifolds). Then $f : X \to Y$ is $\Im$-\'etale if and only if the square
  \[
\begin{tikzcd}
	TX & TY \\
	X & Y
	\arrow[from=1-1, to=2-1]
	\arrow["f"', from=2-1, to=2-2]
	\arrow[from=1-2, to=2-2]
	\arrow["Tf", from=1-1, to=1-2]
	\arrow["\lrcorner"{anchor=center, pos=0.125}, draw=none, from=1-1, to=2-2]
\end{tikzcd}
  \]
  is a pullback.
\end{cor}
\begin{proof}
Penon shows in the final proposition of \cite{Penon:Infinitesimals.and.Intuitionism} that such a map is Penon \'etale if and only if the square of tangent bundles is a pullback, and so the result follows by \cref{thm:Penon.etale.is.im.etale}.
\end{proof}

\subsection{Schreiber manifolds}\label{sec:Schreiber.manifold}

Schreiber describes his notion of manifold in Definition 5.3.88 of \cite{Schreiber:Differential.Cohomology.v2}. We will use a slightly more general definition allowing for multiple different sorts of coordinate spaces.
\begin{defn}[Definition 5.3.88 \cite{Schreiber:Differential.Cohomology.v2}]
  Let $V : I \to \Type$ be a fixed family of \emph{coordinate spaces} indexed by a discrete type $I$. A \emph{Schreiber $V$-manifold} is a type $M$ for which there merely exists a \emph{$V$-atlas}, which consists of:
  \begin{enumerate}
\item A family of types $U : A \to \Type$ indexed by a discrete type $A$.
          \item For every $a : A$, an $\Im$-\'etale map $i_{a} : U_{a} \hookrightarrow M$. We assume that these are jointly surjective: for every $p : M$ there is merely some $a : A$ and $u : U_{a}$ with $p = i_{a}(u)$.
          \item For every $a : A$, an index $ka : I$ and an $\Im$-{\'e}tale embedding $c_{a} : U_{a} \hookrightarrow V_{ka}$.
  \end{enumerate}
\end{defn}

A variant of Schreiber manifolds were studied in homotopy type theory by Cherubini \cite{Cherubini:Thesis}. As a special case, we can consider Satake's notion of orbifold as a space locally modelled on $\Rb^{n} \sslash\, \Gamma$ where $\Gamma$ is a finite subgroup of $\type{O}(n)$.
\begin{defn}\label{defn:Satake.Orbifold}
Let $I :\equiv \flat\{\mbox{finite subgroup of $\type{O}(n)$}\}$ be the type of crisp, finite subgroups of $\type{O}(n)$, and let $V : I \to \Type$ be defined by $V_{{\Gamma^{\flat}}} :\equiv \Rb^{n} \sslash\, \Gamma$. A \emph{Schreiber-Satake orbifold} is a Schreiber $V$-manifold for this choice of $V$.
\end{defn}

As a quick corollary of \cref{thm:microlinear.descends.along.etale}, we can prove that any Schreiber $V$-manifold where $V$ is a family of microlinear types is also microlinear. We need two quick lemmas first: crystaline sum of $\Im$-\'etale maps are $\Im$-\'etale and crystaline sum of microlinear types are microlinear.
\begin{lem}\label{lem:discrete.sum.of.etale.is.etale}
Let $U : A \to \Type$ and $X : I \to \Type$ be type families indexed by $\Im$-crystaline types $A$ and $I$. Let $k : A \to I$ be a map and let $i_{a} : U_{a} \to X_{ka}$ be $\Im$-{\'e}tale for all $a : A$. Then the map $\sum_{k}i : \dsum{a : A} U_{a} \to \dsum{i : I}X_{i}$ defined by $\sum_{k}i(a, u) :\equiv i_{a}(u)$ is $\Im$-{\'e}tale.
\end{lem}
\begin{proof}
  Consider the square
  \[
\begin{tikzcd}
	{\dsum{a : A}U_a} & {\dsum{a : A} \Im U_a} \\
	\dsum{i : I}X_{i} & \dsum{i : I}{\Im X_{i}}
	\arrow[from=2-1, to=2-2]
	\arrow["{\sum_k i}"', from=1-1, to=2-1]
	\arrow["{\sum_k \Im i}", from=1-2, to=2-2]
	\arrow[from=1-1, to=1-2]
\end{tikzcd}
  \]
  where $\sum_{k} \Im i(a, z) :\equiv (ka, \Im i_{a}(z))$. By Lemma 1.24 of \cite{RSS:Modalities} and the fact that $A$ and $I$ are $\Im$-modal, the horizontal maps are $\Im$-units and this square is an $\Im$-naturality square. Therefore, to show that $i \equiv \sum_{a} i_{a}$ is $\Im$-{\'e}tale, we just need to show that this square is a pullback. Consider a point $(i, x) : \dsum{i : I} X_{i}$ and the induced map
  \[
    (\dsum{(a, p) : \fib_{k}(i)}\fib_{i_{a}}(x)) \to (\dsum{(a, p) : \fib_{k}(i)} \fib_{\Im i_{a}}(x^{\Im}))
  \]
  on fibers of the vertical maps. Note that this map is the sum of the maps $\fib_{i_{a}}(x) \to \fib_{\Im i_{a}}(x^{\Im})$ induced by the $\Im$-naturality squares of $i_{a} : U_{a} \to X_{ka}$, which by hypothesis were equivalences. Therefore, this map is an equivalence, and $\sum_{k}i$ is $\Im$-{\'e}tale.
\end{proof}

\begin{lem}\label{lem:crystaline.sum.of.microlinear.is.microlinear}
Let $X : I \to \Type$ be a family of microlinear types indexed by a crystaline type $I$. Then the sum $\dsum{i : I}X_{i}$ is microlinear.
\end{lem}
\begin{proof}
  Let $V$ be an infinitesimal $\Rb$-pushout and consider a map $v : V_{2} +_{V_{1}} V_{3} \to \dsum{i : I} X_{i}$, seeking to show that the type of lifts $\dsum{\tilde{v} : V_{4} \to \dsum{i : I} X_{i}} (\tilde{v} \circ \nabla_{V} = v)$ is contractible. Now, since $V_{2} +_{V_{1}} V_{3}$ and $V_{4}$ are $\Im$-connected, any map from them into a crystaline type is constant and so
  \begin{align*}
    (V_{2} +_{V_{1}} V_{3} \to \dsum{i : I} X_{i}) &\simeq (\dsum{v^{1} : V_{2} +_{V_{1}} V_{3} \to I} (\dprod{\ep : V_{2} + _{V_{1}} V_{3}} X_{v^{1}\ep})) \\
    &\simeq \dsum{i : I} (V_{2} +_{V_{1}}V_{3} \to X_{i})
  \end{align*}
  and similarly $V_{4} \to \dsum{i : I} X_{i}$ is equivalent to $\dsum{i : I} (V_{4} \to X_{i})$. Let $v$ correspond to $(i_{0}, v')$ under this equivalence; the type of lifts is therefore equivalent to $\dsum{i : I} \dsum{\tilde{v}' : V_{4} \to X_{i}} (\tilde{v}' \circ \nabla_{V} = v') \times (i = i_{0})$. We can contract $i$ into $i_{0}$ so that this type is equivalent to $\dsum{\tilde{v}' : V_{4} \to X_{i_{0}}}(\tilde{v}'\circ\nabla_{V} = v')$ which was contractible by hypothesis.
\end{proof}

\begin{thm}\label{thm:Schreiber.manifold.microlinear}
Any Schreiber $V$-manifold for a family of microlinear types $V : I \to \Type$ is itself microlinear.
\end{thm}
\begin{proof}
  Consider an atlas of $X$ and the associated span
  \[
\begin{tikzcd}
	& {\dsum{a : A} U_a} \\
	\dsum{i : I} V_{i} && X
	\arrow[from=1-2, to=2-1, "\sum_{k}c"']
	\arrow["\sum i", two heads, from=1-2, to=2-3]
\end{tikzcd}
\]
given by summing up the maps $i_{a}  : U_{a} \to X$ and $c_{a} : U_{a} \to V_{ka}$
  By \cref{lem:discrete.sum.of.etale.is.etale}, these maps are are $\Im$-{\'e}tale, and by hypothesis the right leg $i$ is surjective. Therefore, by combining \cref{lem:microlinear.ascends.along.etale} and \cref{thm:microlinear.descends.along.etale}, we see that $X$ is microlinear when $\dsum{i : I}V_{i}$ is. And $\dsum{i : I}V_{i}$ is microlinear when each of the $V_{i}$ is by \cref{lem:crystaline.sum.of.microlinear.is.microlinear}.
\end{proof}

\subsection{{\'E}tale groupoids}\label{sec:etale.groupoid}
Given \cref{thm:microlinear.descends.along.etale}, it is easy to show that the quotient
 $X \sslash \Gamma$ of the action of a crisply discrete higher group $\Gamma$ on a microlinear type $X$ is microlinear.

\begin{thm}\label{thm:discrete.homotopy.quotient.is.microlinear}
Let $\Gamma$ be a crisply discrete group, and let $X\twisted{-} : \B\Gamma \to \Type$ be an action of $\Gamma$ on a microlinear type $X :\equiv X \twisted{\pt_{\B \Gamma}}$. Then the homotopy quoteint $X \sslash \Gamma$ is microlinear.
\end{thm}
\begin{proof}
By Theorem 7.7 of \cite{Jaz:Good.Fibrations}, the quotient map $q : X \to X \sslash \Gamma$ is $\shape$-{\'e}tale and therefore by \cref{lem:discrete.is.crystaline} is $\Im$-{\'e}tale. Since $q$ is surjective and $\Im$-{\'e}tale, by \cref{thm:microlinear.descends.along.etale}, we see that if $X$ is microlinear then $X\sslash \Gamma$ is as well.
\end{proof}

As a corollary, we can give a satisfying condition for the microlinearity of a crisp type.
\begin{cor}
Suppose that $X$ is a crisp, pointed type, and that $X$ is path connected in the sense that $\shape_{1} X$ is $0$-connected. If the universal cover $\tilde{X}$ of $X$ is microlinear, then $X$ is microlinear.
\end{cor}
\begin{proof}
Since $\shape_{1} X$ was presumed to be $0$-connected, it is a $\B \pi_{1}(X)$ (where we define $\pi_{1}(X) :\equiv \trunc{\Omega \shape_{1}X}_{0}$ to be the fundamental group of $X$). By definition, the universal cover $\tilde{X}$ is the fiber of the unit $(-)^{\shape_{1}} : X \to \shape_{1}X$. Therefore, we can see the map $t \mapsto \fib_{(-)^{\shape_{1}}}(t) : \B \pi_{1} X \to \Type$ as giving the monodromy action of $\pi_{1}(X)$ on the universal cover $\tilde{X}$. The homotopy quotient is therefore \(\dsum{t : \B \pi_{1}(X)}\fib_{(-)^{\shape_{1}}}(t)\), which is equivalent to $X$; in other words, we have $X \simeq \tilde{X}\sslash \pi_{1}(X) $. Since $\pi_{1}(X)$ is a crisp group, \cref{thm:discrete.homotopy.quotient.is.microlinear} then shows that if $\tilde{X}$ is microlinear, so is $X$.
\end{proof}

However, not every orbifold may be presented as the quotient of a smooth space by the action of a discrete group. A general way to present orbifolds is with \emph{proper \'etale groupoids}, as first defined by Moerdijk and Pronk \cite{Moerdijk-Pronk:Orbifolds}. In this section, we will show that the wider class of \'etale groupoids are microlinear; in the next section, we'll discuss the notion of compactness appropriate for synthetic differential geometry, define proper \'etale groupoids, and prove that they are orbifolds in the sense of \cref{defn:orbifold}.

First, let's recall the notion of \emph{pregroupoid}. In HoTT, a groupoid is best understood as a type which is $1$-truncated. However, the traditional definition of a groupoid as having a set of objects and a set of isomorphism between these objects can still be performed in HoTT; the resulting notion is that of a \emph{pregroupoid}. The terminology is by analogy with the relation between preorders and ordered sets.

\begin{defn}
A \emph{precategory} $\Ca$ consists of a type of objects $\Ca_{0}$, and for each two objects $x, y : \Ca_{1}$ a set $\Ca(x, y)$ of morphisms, together with an associative, unital composition of morphisms. A \emph{pregroupoid} is a precategory where every morphism is an isomorphism.

A precategory is a \emph{category} (or a \emph{univalent category}, for emphasis) if the map $\term{idtoiso} : (x = y) \to \type{Iso}_{\Ca}(x, y)$ defined by $\refl_{x} \mapsto \id_{x}$ is an equivalence for all objects $x, y : \Ca_{0}$.
\end{defn}

\begin{rmk}
A univalent pregroupoid $\Ga$ --- one for which the map $\term{idtoiso} : (x = y) \to \term{Iso}_{\Ga}(x, y)$ is an equivalence for all $x, y : \Ga_{0}$ --- carries no more information than its type $\Ga_{0}$ of objects, since  $\term{Iso}_{\Ga}(x, y) = \Ga(x, y)$. Furthermore, since by hypothesis there is a set of morphisms $\Ga(x, y)$, we find that there is a set of identifications $(x = y)$, making $\Ga_{0}$ into a groupoid in the sense of being a $1$-type. Therefore, we are free to identify univalent pregroupoids with their groupoids ($1$-types) of objects. We will therefore drop the subscript on $\Ga_{0}$ when talking about groupoids.
\end{rmk}

There is a universal groupoid generated by any pregroupoid: the \emph{Rezk completion}. For more, see Section 9.9 of the HoTT Book \cite{HoTTBook}.
\begin{defn}
The \emph{Rezk completion} $\term{r}\Ca$ of a precategory $\Ca$ is the essential image of the Yoneda embedding --- the full subcategory of the category $\hat{C}$ of presheaves on $\Ca$ spanned by the representable functors. Explicitly $\term{r}\Ca$ has objects those presheaves $F$ for which there merely exists an $x : \Ca$ and a natural isomorphism $\Ca(-, x) \cong F$.
\end{defn}

With these definitions, we can now define the notion of \'etale groupoid.
\begin{defn}\label{defn:etale.groupoid}
  An \emph{{\'e}tale pregroupoid} is a pregroupoid $\Ga$ where
  \begin{itemize}
\item The type $\Ga_{0}$ of objects is microlinear.
\item The source map $( (x, y, p) \mapsto x ) : \Ga_{1} \to \Ga_{0}$ from the type of morphisms $\Ga_{1} :\equiv \dsum{x , y : \Ga_{0}} \Ga(x, y)$ to the type of objects is $\Im$-{\'e}tale.
  \end{itemize}

  An \emph{\'etale groupoid} is a groupoid which is equivalent to the Rezk completion of an \'etale pregroupoid.
\end{defn}

In order to prove that \'etale groupoids are microlinear, we will show that the Rezk completion of \'etale pregroupoids are microlinear. To do this, we will show that the Yoneda embedding $\yo : \Ga_{0} \to \type{r}\Ga$ is $\Im$-\'etale. Since $\yo$ is by definition surjective, and since $\Ga_{0}$ is by hypothesis microlinear, it will follow by \cref{thm:microlinear.descends.along.etale} that $\type{r}\Ga$ is \'etale.

The proof that $\yo : \Ga_{0} \to \type{r}\Ga$ is $\Im$-\'etale for an \'etale pregroupoid $\Ga$ is not trivial. It will follow from the following theorem, and only in the case that $\Ga$ is crisp.
\begin{thm}\label{thm:etale.descent}
Suppose that $\modal$ is a modality with surjective units which preserves $\emptyset$, crisp pushouts, and colimits of crisp sequences.  Let $f :: A \to B$ be a crisp, surjective map with inhabited domain $A$, and suppose that the pullback $\fst : A \times_{B} A \to A$ of $f$ along itself is $\modal$-{\'etale}:
  \[
\begin{tikzcd}
	{A \times_B A} & A \\
	A & B
	\arrow["f", two heads, from=1-2, to=2-2]
	\arrow["f"', two heads, from=2-1, to=2-2]
	\arrow["\fst"', from=1-1, to=2-1]
	\arrow["\snd", from=1-1, to=1-2]
	\arrow["\lrcorner"{anchor=center, pos=0.125}, draw=none, from=1-1, to=2-2]
\end{tikzcd}
  \]
  Then $f$ is also $\modal$-\'etale
\end{thm}

We state \cref{thm:etale.descent} in this abstract way so that it applies not only to $\Im$, but also to $\shape$. That \cref{thm:etale.descent} applies to the modality $\Im$ at all is due to Postulate $W$: the type $\Dc$ is presumed to be \emph{tiny}. Results of \cref{sec:tiny} allow us to show that $\Im$ preserves crisp pushouts and colimits of sequences.
\begin{thm}\label{thm:Im.preserves.crisp.colimits}
The modality $\Im$ preserves crisp pushouts and colimits of crisp sequences.
\end{thm}
\begin{proof}
This is a special case of \cref{thm:tiny.null.preservers.crisp.colimits}, since $\Dc$ is tiny.
 \end{proof}

Before proving \cref{thm:etale.descent}, we will use it to show that crisp \'etale groupoids are microlinear.
\begin{thm}\label{thm:etale.groupoid.microlinear}
Crisp \'etale groupoids are microlinear.
\end{thm}
\begin{proof}
  If $\Ha$ is a crisp \'etale groupoid, then by hypothesis there is a crisp \'etale pregroupoid $\Ga$ with $\type{r}\Ga \simeq \Ha$. It will therefore suffice to show that $\type{r}\Ga$ is microlinear. Since $\Ga_{0}$ is by hypothesis microlinear, it suffices to show that $\yo : \Ga_{0} \to \type{r}\Ga$ is $\Im$-{\'etale}.

  Now, for every pair of objects $x, y : \Ga_{0}$, we have that $(\yo(x) = \yo(y)) \simeq \Ga(x, y)$ by the Yoneda lemma. Therefore, we have a pullback square:
  \[
\begin{tikzcd}
	{\Ga_1} & {\Ga_0} \\
	{\Ga_0} & {\type{r}\Ga}
	\arrow["s"', from=1-1, to=2-1]
	\arrow["t", from=1-1, to=1-2]
	\arrow["\yo"', from=2-1, to=2-2]
	\arrow["\yo", from=1-2, to=2-2]
\end{tikzcd}
  \]
  Since $\yo$ is crisp and surjective, and $s$ is $\Im$-\'etale, it follows by \cref{thm:etale.descent} that $\yo$ is is $\Im$-\'etale.
  \end{proof}

  In order to prove \cref{thm:etale.descent}, we will need a few useful lemmas. The first two lemmas that we need show that if all the maps in a crisp diagram $\Da$ are $\Im$-\'etale, then so are all the maps in the cocone $\Da \to \colim \Da$. In all of the following lemmas, we will assume that $\modal$ is a modality with surjective units which preserves $\emptyset$, crisp pushouts, and crisp colimits of sequences.
  \begin{lem}\label{lem:pushout.of.etale.is.etale}
    Suppose that
    \[
\begin{tikzcd}
	A & B \\
	C & D
	\arrow["f"', from=1-1, to=2-1]
	\arrow["g", from=1-1, to=1-2]
	\arrow["h"', from=2-1, to=2-2]
	\arrow["k", from=1-2, to=2-2]
	\arrow["\lrcorner"{anchor=center, pos=0.125, rotate=180}, draw=none, from=2-2, to=1-1]
\end{tikzcd}
    \]
    is a crisp pushout. If $f$ and $g$ are $\modal$-{\'e}tale, then so are $h$ and $k$.
  \end{lem}
  \begin{proof}
    This follows by the fact that $\modal$ preserves crisp pushouts and by Mather's cube theorem, which is proven in HoTT in Theorem 2.2.11 of \cite{Rijke:Thesis}. Consider the cube
    \[
\begin{tikzcd}[sep = small, row sep = tiny]
	& A \\
	B &&& C \\
	&& D \\
	& {\modal A} \\
	{\modal B} &&& {\modal C} \\
	&& {\modal D}
	\arrow[from=2-1, to=5-1]
	\arrow[from=1-2, to=2-1]
	\arrow[from=1-2, to=2-4]
	\arrow[from=2-4, to=3-3]
	\arrow[from=1-2, to=4-2]
	\arrow[from=2-4, to=5-4]
	\arrow[from=4-2, to=5-1]
	\arrow[from=4-2, to=5-4]
	\arrow[from=5-1, to=6-3]
	\arrow[from=5-4, to=6-3]
	\arrow[from=2-1, to=3-3, crossing over]
	\arrow[from=3-3, to=6-3, crossing over]
\end{tikzcd}
    \]
    The top face is a pushout by hypothesis, and since $\modal$ was presumed to preserve crisp pushouts, the bottom face is as well. If $f$ and $g$ are $\modal$-\'etale, then the back two faces are pullbacks. This implies that the front two faces are pullbacks, which shows that $h$ and $k$ are $\modal$-\'etale.
  \end{proof}

  \begin{lem}\label{lem:sequential.colimit.of.etale.is.etale}
    Suppose that
    \[
A_{0} \xto{i_{0}} A_{1} \xto{i_{1}} \cdots
    \]
    is a crisp sequence of types with colimit $A_{\infty}$. If $i_{j}$ is $\modal$-{\'etale} for all $j$, then so are the maps $i_{j,\infty} : A_{j} \to A_{\infty}$.
  \end{lem}
  \begin{proof}
This follows from the analogous descent property for sequences as Mathers's cube theorem is for pushouts, by essentially the same argument.
    \end{proof}

    Next, we need a useful closure property of \'etale maps.
    \begin{lem}\label{lem:etale.2of3}
Suppose that $f : A \to B$ is $\modal$-\'etale and that $g : B \to A$. If $g \circ f$ is $\modal$-\'etale and $\modal f$ is surjective, then $g$ is $\modal$-\'etale.
    \end{lem}
    \begin{proof}
      This follows from the analogous property for pullbacks. In the given situation, we have a diagram
      \[
\begin{tikzcd}
	A & B & C \\
	{\modal A} & {\modal B} & {\modal C}
	\arrow["f", from=1-1, to=1-2]
	\arrow["g", from=1-2, to=1-3]
	\arrow[from=1-1, to=2-1]
	\arrow["{\modal f}"', from=2-1, to=2-2]
	\arrow[from=1-2, to=2-2]
	\arrow["{\modal g}"', from=2-2, to=2-3]
	\arrow[from=1-3, to=2-3]
	\arrow["\lrcorner"{anchor=center, pos=0.125}, draw=none, from=1-1, to=2-2]
\end{tikzcd}
      \]
      in which the left and composite squares are pullbacks, and where $\modal f$ is surjective. It follows that the right square is also a pullback, which means that $g$ is $\modal$-\'etale.
    \end{proof}

    Finally, we are ready to prove \cref{thm:etale.descent}.
    \begin{proof}[Proof of \cref{thm:etale.descent}]
      By Rijke's join construction \cite{Rijke:join.construction}, $B$ is the sequential colimit of the sequence
      \[
A_{0} \xto{i_{0}} A_{1} \xto{i_{1}} A_{2} \xto{i_{2}} \cdots
      \]
      which is inductively defined by defining $A_{0}$ to be $\emptyset$, $i_{0} :A_{0} \to A_{1}$ and $f_{0} : A_{0} \to B$ to be the unique maps, and then defining $A_{n + 1} \equiv A_{n} \ast_{B} A$ and $f_{n + 1} : A_{n + 1} \to X$, and $i_{n+1} : A_{n} \to A_{n+1}$ by the universal property:
      \[
\begin{tikzcd}
	{A_n \times_B A} & A \\
	{A_n} & {A_{n+1}} \\
	&& B
	\arrow["f", from=1-2, to=3-3, bend left]
	\arrow["{f_n}"', from=2-1, to=3-3, bend right]
	\arrow[from=1-2, to=2-2]
	\arrow[from=2-1, to=2-2, "i_{n+1}"']
	\arrow[from=1-1, to=2-1]
	\arrow[from=1-1, to=1-2]
	\arrow["\lrcorner"{anchor=center, pos=0.125, rotate=180}, draw=none, from=2-2, to=1-1]
	\arrow["{f_{n+1}}" description, dashed, from=2-2, to=3-3]
\end{tikzcd}
      \]
      Here, the outer square is a pullback, and the inner square is a pushout defining $A_{n+1}$ as the join of $A_{n}$ and $A$ over $B$. Note that this sequence is crisp, and that $A_{1} \simeq A$. By \cref{lem:sequential.colimit.of.etale.is.etale}, it will suffice to show that each $i_{n}$ is $\modal$-\'etale. We can argue this by induction.

      First, we note that $i_{0} : A_{0} \to A_{1}$ is $\modal$-\'etale since $\modal$ preserves $\emptyset$ and any square
      \[
\begin{tikzcd}
	\emptyset & {A_1} \\
	{\modal \emptyset} & {\modal A_1}
	\arrow[from=1-1, to=2-1, equals]
	\arrow["{i_0}", from=1-1, to=1-2]
	\arrow["{\modal i_0}"', from=2-1, to=2-2]
	\arrow[from=1-2, to=2-2]
\end{tikzcd}
      \]
      is a pullback. Now, suppose that all maps $i_{j}$ for $j < n$ are $\modal$-\'etale, seeking to show that $i_{n}$ is $\modal$-\'etale. Since $i_{n}$ is constructed as a pushout inclusion
      \[
\begin{tikzcd}
	{A_n \times_B A} & A \\
	{A_n} & {A_{n+1}} \\
	\arrow[from=1-2, to=2-2]
	\arrow[from=2-1, to=2-2, "i_{n+1}"']
	\arrow[from=1-1, to=2-1, "\fst"']
	\arrow[from=1-1, to=1-2, "\snd"]
	\arrow["\lrcorner"{anchor=center, pos=0.125, rotate=180}, draw=none, from=2-2, to=1-1]
\end{tikzcd}
      \]
it will suffice to prove that both projections $\fst : A_{n} \times_{B} A \to A_{n}$ and $\snd : A_{n} \times_{B} A \to A$ are $\modal$-\'etale by \cref{lem:pushout.of.etale.is.etale}.
      Consider the following diagram:
      \[
\begin{tikzcd}
	{A \times_B A} & {A_{n} \times_B A} & A \\
	A & {A_n} & B
	\arrow["f", from=1-3, to=2-3]
	\arrow["{f_n}"', from=2-2, to=2-3]
	\arrow["\fst"', from=1-2, to=2-2]
	\arrow["\snd", from=1-2, to=1-3]
	\arrow["i"', from=2-1, to=2-2]
	\arrow["\fst"', from=1-1, to=2-1]
	\arrow["{i\times_B A}", from=1-1, to=1-2]
	\arrow["\snd", bend left, from=1-1, to=1-3]
	\arrow["f"', bend right, from=2-1, to=2-3]
	\arrow["\lrcorner"{anchor=center, pos=0.125}, draw=none, from=1-2, to=2-3]
	\arrow["\lrcorner"{anchor=center, pos=0.125}, draw=none, from=1-1, to=2-2]
\end{tikzcd}
      \]
      Here the map $i : A \to A_{n}$ is the composite of all maps $A \xto{\sim} A_{1} \xto{i_{1}} A_{2}\xto{i_{2}} \cdots \xto{i_{n-1}} A_{n}$. Note that since $A$ is inhabited, each $i_{j} : A_{j} \to A_{j + 1}$ is surjective and therefore $i : A \to A_{n}$ is surjective. The right square is a pullback by definition, and so is the outer square; therefore, the left square is also a pullback. We know by assumption that $\fst : A \times_{B} A \to A$ is $\modal$-\'etale, and by inductive hypothesis that $i : A \to A_{n}$ is $\modal$-\'etale. By symmetry, we also know that $\snd : A \times_{B} A \to A$ is $\modal$-\'etale.

      To show that both $\fst : A_{n} \times_{B} A \to A_{n}$ and $\snd  : A_{n} \times_{B} A \to A$ are $\modal$-\'etale, it therefore suffices by \cref{lem:etale.2of3} to show that $i \times_{B} A : A \times_{B} A \to A_{n} \times_{B} A$ is surjective and $\modal$-\'etale. But this is the pullback of $i : A\to A_{n}$, which as we noted above is surjective and $\modal$-\'etale, so this concludes our proof.
    \end{proof}

\subsection{Deloopings of  infinitesimally linear groups are infinitesimally linear}\label{sec:deloopings.inf.linear}

So far, we have been showing that various notions of smooth spaces are microlinear. This in particular implies that their tangent spaces are $\Rb$-modules by \cref{thm:inf.linear.R.module}. But there is a weaker condition which implies the same thing: infinitesimal linearity (\cref{defn:inf.linear}). Infinitesimal linearity says essentially nothing but the fact that the tangent spaces are $\Rb$-modules.

\cref{thm:inf.linear.R.module} was proven with no truncation conditions on the infinitesimally linear type $X$. We have seen some higher types which are infinitesimally linear since they are microlinear --- for example $X \sslash \Gamma$ where $\Gamma$ is a discrete higher group (\cref{thm:discrete.homotopy.quotient.is.microlinear}). But these example have all had discrete types of identifications. If we restrict ourselves to infinitesimal linearity, we can find examples of higher types whose spaces of identifications are not discrete. In particular, we can show that if $G$ is a crisp infinitesimally linear higher group (for example, a Lie group), then any delooping $\B G$ will be infinitesimally linear.

To prove this, we first need a general lemma about higher groups.
\begin{lem}\label{lem:square.of.higher.groups.pullback}
  Consider a square of homomorphisms of higher groups:
  \[
\begin{tikzcd}
	{\B G_1} & {\B G_3} && {G_1} & {G_3} \\
	{\B G_2} & {\B G_4} && {G_2} & {G_4}
	\arrow["{\B h}"', from=1-1, to=2-1]
	\arrow["{\B f}", from=1-1, to=1-2]
	\arrow["{\B k}", from=1-2, to=2-2]
	\arrow["{\B g}"', from=2-1, to=2-2]
	\arrow["f", from=1-4, to=1-5]
	\arrow["h"', from=1-4, to=2-4]
	\arrow["k", two heads, from=1-5, to=2-5]
	\arrow["g"', from=2-4, to=2-5]
	\arrow["\lrcorner"{anchor=center, pos=0.125}, draw=none, from=1-4, to=2-5]
\end{tikzcd}
  \]
  That is, consider a square of pointed maps between pointed, $0$-connected types on the left. If $k :\equiv \Omega \B k$ is surjective and the looped square on the right is a pullback, then the square on the left is a pullback.
\end{lem}
\begin{proof}
Because $\B G_{4}$ is $0$-connected and $k$ is surjective, $\B k$ is $0$-connected. Since the square on the right is a pullback, $h$ is also surjective and so $\B h$ is $0$-connected. To show that the square on the left is a pullback, it suffices to show that the map $\B f_{\ast} : \fib_{\B h}(\pt_{\B G_{2}}) \to \fib_{\B k}(\pt_{\B G_{4}})$ is an equivalence. But these types are both pointed and $0$-connected, so by the ``fundamental theorem of higher groups'', it suffices to prove that $\Omega \B f_{\ast}$ is an equivalence, or equivalently that $f_{\ast}: \fib_{h}(1_{G_{2}}) \to \fib_{k}(1_{G_{4}})$ is an equivalence. But the square on the right is a pullback, so $f_{\ast}$ is an equivalence.
\end{proof}

\begin{warn}
The surjectivity of $k$ is necessary. Let $k : G_{3} \to G_{4}$ be the doubling map $2 : \Zb \to \Zb$, and let $G_{2} \equiv \ast$ be trivial. Then $G_{1} = \ker k = \ast$ is trivial, but $\fib_{\B k}(\pt_{\B G_{4}}) $ is the quotient $\Zb / 2$ of the action $2 : \Zb \to \Zb$. This is not $0$-connected and so does not deloop $G_{1}$.
\end{warn}

\begin{thm}\label{thm:inf.linear.BG}
Let $G$ be a crisp, infinitesimally linear higher group. Then $\B G$ is infinitesimally linear.
\end{thm}
\begin{proof}
  We need to show that the squares

\[
\begin{tikzcd}
	{(\B G)^{\Db(n + m)}} & {(\B G)^{\Db(n)}} \\
	{(\B G)^{\Db(m)}} & \B G
	\arrow[from=1-1, to=2-1]
	\arrow[from=1-1, to=1-2]
	\arrow[from=1-2, to=2-2]
	\arrow[from=2-1, to=2-2]
\end{tikzcd}
  \]
are pullbacks. By Postulate W, the crisp infinitesimal varieties $\Db(n)$ are tiny in the sense of \cref{defn:tiny}. Therefore, by \cref{cor:tiny.connectedness}, the types $(\B G)^{\Db(n)}$ are $0$-connected. Furthermore, the projection $G^{\Db(n)} \to G$ is surjective since it splits (by precompositiong along the map $\Db(n) \to \ast$). By hypothesis, the looped square is a pullback since $G$ was assumed to be infinitesimally linear. Therefore, \cref{lem:square.of.higher.groups.pullback} applies and we may conclude that this square is a pullback.
\end{proof}

\begin{rmk}
Note that the proof of \cref{thm:inf.linear.BG} made essential use of the surjectivity of the projection $G^{\Db(n)} \to G$. This is why the same argument does not show that the delooping of a crisp microlinear group is microlinear; in the general case, precomposition $G^{V_{2}} \to G^{V_{1}}$ along a map $V_{1} \to V_{2}$ of crisp infinitesimal varieties need not be surjective. I do not know whether $\B G$ of a crisp microlinear (higher) group $G$ is necessarily microlinear.
\end{rmk}

\begin{cor}
Let $G$ be a crisp, infinitesimally linear higher group, and define  $\mathfrak{g} :\equiv T_{1}G$ to be its ``higher Lie algebra''. Then $\mathfrak{g}$ is a higher group with delooping $\B \mathfrak{g} :\equiv T_{\pt_{\B G}}\B G$, and its delooping $\B \mathfrak{g}$ is a (higher) $\Rb$-module.
\end{cor}
\begin{proof}
By definition, $T_{\pt_{\B G}}\B G :\equiv \dsum{v : \Db \to \B G}(v(0) = \pt_{\B G})$ is the fiber of the tangent bundle projection $\type{ev}_{0} : (\B G)^{\Db} \to \B G$. By \cref{cor:tiny.connectedness}, $(\B G)^{\Db} = \B (G^{\Db})$ is $0$-connected and deloops the tangent space of $G$, and the projection $(\B G)^{\Db} \to \B G$ deloops the projection $G^{\Db} \to G$. Therefore, the fiber $T_{\pt_{\B G}}\B G$ of $(\B G)^{\Db} \to \B G$ deloops the fiber $\mathfrak{g}$ of $G^{\Db} \to G$ over $1 : G$.
\end{proof}

    \section{Finiteness, compactness, and proper \'etale groupoids}

    We have seen that \'etale groupoids are microlinear. But not all \'etale groupoids are orbifolds --- there is a finiteness condition in the definition of an orbifold. In an \'etale pregroupoid, this is traditionally handled by asking that the map $(s, t) : \Ga_{1} \to \Ga_{0} \times \Ga_{0}$ be proper (see \cite{Moerdijk-Pronk:Orbifolds}). Using the usual definition of proper (inverse images of compact sets are compact) with the usual definition of compact (every open cover admits a finitely enumerable subcover), we can actually show that the isotropy groups of such an \'etale pregroupoid are \emph{finite}.

    But we don't actually want for the isotropy groups to be finite, because this would imply that they have a constant cardinality over any connected component of our orbifold. This is because there is a function $\type{card} : \type{Fin} \to \Nb$ which sends a finite set $F$ to its cardinality $\type{card}(F)$. If for all $x, y : X$, the type $(x = y)$ were finite, then we would have a function $(x, y) \mapsto \type{card}(x = y) : X \times X \to \Nb$. Because $\Nb$ is discrete, this map would factor through the set $\shape_{0}(X \times X)$ of connected components of $X\times X$. In other words, the cardinality of $(x = y)$ would be constant on each connected component.

  This discussion reveals that being finite is a very strong condition on a set in constructive mathematics. Luckily, and sometimes frustratingly, there are many different, weaker notions of finiteness in constructive mathematics, which we survey in \cref{defn:finite}. In \cref{sec:finiteness}, we will introduce a new notion of finiteness for our own purposes: \emph{properly finite} (see \cref{defn:properly.finite}). A set is properly finite when it is discrete and a subquotient of a finite set.

  Changing our notion of finiteness means that we will also need to change our notion of compactness to match. Dubuc and Penon have investigated a useful notion of compactness which, in the intended models, corresponds on crisp ordinary manifolds with external compactness (\cite{Dubuc-Penon:Compact}). We recall Dubuc and Penon's notion of compactness in \cref{defn:Dubuc.Penon.compact}.

  Using Dubuc-Penon compactness, in \cref{defn:ordinary.proper.etale.groupoid} we will define an ordinary proper \'etale pregroupoid to be an \'etale pregroupoid $\Ga$ where $\Ga_{0}$ and $\Ga_{1}$ are ordinary manifolds, and where the map $(s, t) : \Ga_{1} \to \Ga_{0} \times \Ga_{0}$ is Dubuc-Penon proper in the sense that the inverse image of any Dubuc-Penon compact subset is Dubuc-Penon compact. With this definition of ordinary proper \'etale groupoid, we may then state \cref{thm:ordinary.proper.etale.groupoid.is.orbifold}: the Rezk completion of a crisp, ordinary proper \'etale pregroupoid is an orbifold in the sense of \cref{defn:orbifold}.

  Since we have already shown that the Rezk completion $\type{r}\Ga$ of a crisp \'etale pregroupoid is microlinear, it only remains to show that the types of identifications in $\type{r}\Ga$ are properly finite. The types of identifications in $\type{r}\Ga$ are identifiable with the sets of maps $\Ga(x, y)$ in the pregroupoid $\Ga$. It follows from the assumption that $\Ga$ is proper \'etale that $\Ga(x, y)$ is crystaline and Dubuc-Penon compact. We can therefore argue that $\Ga(x, y)$ is properly finite in two stages: first, in \cref{lem:crystaline.subset.of.manifold.is.discrete}, that crystaline subsets of ordinary manifolds are discrete, and second, in \cref{thm:discrete.compact.subset.of.manifolds.is.properly.finite}, that discrete Dubuc-Penon compact subsets of ordinary manifolds are properly finite.

  In \cref{sec:compact}, we will explore the notion of Dubuc-Penon compactness (\cref{defn:Dubuc.Penon.compact}). In particular, we will show that Dubuc-Penon compact sets are compact in ways more closely resembling ordinary compactness: in \cref{prop:db.compact.is.countably.compact} we show that any Dubuc-Penon compact set is countably compact and \emph{subcountably subcompact}, and in \cref{cor:db.compact.in.manifold.is.refinement.subcompact} that any Dubuc-Penon compact subset of an ordinary manifold is \emph{refinement subcompact}. See \cref{defn:compact} for the definitions of these subtly differing notions of compactness. These results depend in an essential way on the Covering Property appearing in \cref{axiom:SDG}. At the end, we will prove \cref{thm:discrete.compact.subset.of.manifolds.is.properly.finite}: discrete Dubuc-Penon compact subsets of ordinary manifolds are properly finite.

  In \cref{sec:final.theorem}, we will prove \cref{lem:crystaline.subset.of.manifold.is.discrete}: crystaline subsets of ordinary manifolds are discrete. This completes the proof of \cref{thm:ordinary.proper.etale.groupoid.is.orbifold}.

  Finally, in \cref{sec:global.quotient}, we note that the quotient of a microlinear set by the action of a finite group is an orbifold, and that orbifolds are closed under pullback. As a corollary, the inertia orbifold $X^{S^{1}}$ of any orbifold is itself an orbifold.

   \subsection{Notions of finiteness}\label{sec:finiteness}
In constructive mathematics, the notion of ``finiteness'' fractures into a number of inequivalent notions.
    \begin{defn}[Standard]\label{defn:finite}
      Let $X$ be a set. Then:
      \begin{enumerate}
\item $X$ is \emph{finite} if there is an $n : \Nb$ and an equivalence $X \simeq \ord{n}$ with the standard finite ordinal $\ord{n} :\equiv \{0, \ldots, n-1\}$.
              \item $X$ is \emph{subfinite} if there is an $n : \Nb$ and an embedding $X \hookrightarrow \ord{n}$.
              \item $X$ is \emph{finitely enumerable} (also known as Kurotowski finite) if there is an $n : \Nb$ and a surjection $\ord{n} \twoheadrightarrow X$.
              \item $X$ is \emph{subfinitely enumerable} if there is a subfinite set $\tilde{X}$ and a surjection $\tilde{X} \twoheadrightarrow X$.      \end{enumerate}
    \end{defn}

    These notions of finiteness are liminal in the sense that they are only distinct for non-crisp sets (sets which, in some sense, vary continuously in a free variable). Crisply, there is only one notion of finiteness.
    \begin{prop}\label{prop:crisp.subfinitely.enumerable.is.finite}
      Any crisp subfinitely enumerable set is finite.
    \end{prop}
    \begin{proof}
      Let $X$ be a crisp subfinitely enumerable set. Then there is a crisp subset $\tilde{X} \subseteq \ord{n}$ and a crisp surjection $f : \tilde{X} \to X$. Since $\tilde{X}$ is crisp and $\ord{n}$ is crisply discrete, the predicate $i \in \tilde{X}$ for $i :: \ord{n}$ is crisp, and therefore decidable. So, as a decidable subset of a finite set, $\tilde{X}$ is finite.

      This leaves $X$ as the crisp quotient of a finite set. The relation $x \sim y$ defined by $f(x) = f(y)$ on $\tilde{X}$ may therefore also be taken as crisp and therefore decidable. So, as the quotient of a finite set by a decideable relation, $X$ is finite.
    \end{proof}

    We will add an additional notion of finiteness to this list, one which has already appeared in our definition of orbifold (\cref{defn:orbifold}).

    \begin{defn}\label{defn:properly.finite}
A set $X$ is said to be \emph{properly finite} if it is subfinitely enumerable and has decideable equality --- for all $x, y : X$, $(x = y) \vee (x \neq y)$. By Lemma 8.15 of \cite{Shulman:Real.Cohesion}, a set is properly finite if and only if it is discrete and subfinitely enumerable.
    \end{defn}

    \begin{rmk}
      It is folklore that a discrete finitely enumerable set is necessarily finite. This is because such a set is the quotient of a finite set $\ord{n}$ by a decidable equivalence relation $\sim \subseteq \ord{n} \times \ord{n}$, which is a decidable subset of a finite set and therefore itself finite. Therefore, the quotient can also be shown to be finite.

      We might therefore expect that a properly finite subset must necessarily be subfinite. This would be very nice, but I have neither managed to prove it nor construct a counterexample.
    \end{rmk}

    \begin{lem}
Subfinite sets (and so also finite sets) are properly finite.
    \end{lem}
    \begin{proof}
As subsets of discrete sets, subfinite sets are discrete. They are also of course subfinitely enumerable.
    \end{proof}

    \begin{warn}
Though subfinite sets are properly finite, it is not necessarily the case that any subfinitely enumerable set is properly finite. Let $z : \Rb$ and let $Q_{z} = \{0, 1\}/\sim$ where $0 \sim 1$ if and only if $z = 0$ (type theoretically, this is the suspension $\Sigma (z = 0)$ of the proposition $z = 0$). This set is finitely enumerable, and therefore also subfinitely enumerable. But it is not discrete, since the proposition $[0] = [1]$ in $Q_{z}$ is equivalent to $z = 0$, which is not decidable.
    \end{warn}

    The following proposition gives us another useful definition of a subfinitely enumerable set, and therefore also of a properly finite set.
    \begin{prop}[Standard]\label{lem:properly.finite.by.relation}
      A set $X$ is subfinitely enumberable when there exists an $n : \Nb$ and a relation $r : X \times \ord{n} \to \Prop$ such that:
      \begin{enumerate}
\item For all $x : X$, there is some $i : \ord{n}$ with $r(x, i)$.
\item If $r(x, i)$ and $r(x, j)$, then $i = j$.
      \end{enumerate}
      We will refer to such a relation as an \emph{association} of $X$ with $\ord{n}$. If $r \subseteq X \times Y$ is an association of $X$ with $Y$ and $Y$ is subfinitely enumberable, then $X$ is also subfinitely enumerable
    \end{prop}
    \begin{proof}
      The relation itself, considered as a subset of the product $X \times \ord{n}$, defines the $\tilde{X}$ which is a subset of $\ord{n}$ and which surjects onto $X$. If $r \subseteq X \times Y$ and $s \subseteq Y \times \ord{n}$ are associations, then the composite relation $s \circ r \subseteq X \times \ord{n}$ defined by
      $$(s \circ r)(x, i) :\equiv \exists y : Y.\, r(x, y) \wedge s(y, i)$$
      is also an association. For all $x : X$, there is a $y : Y$ with $r(x, y)$, and for this $y$ and $i : \ord{n}$ with $s(y, n)$, so that $(s \circ r)(x, i)$. If there is a $y$ such that $r(x, y)$ and $s(y, i)$ and a $y'$ such that $r(x, y')$ and $r(x, j)$, then $y = y'$ and therefore $i = j$.
    \end{proof}

 We can also show that (subfinitely) enumerable sets may be equally defined as sub(finitely enumerable) sets.
    \begin{lem}[Standard]
$X$ is subfinitely enumerable if and only if there exists a finitely enumerable $\hat{X}$ and an embedding $X \hookrightarrow \hat{X}$.    \end{lem}
    \begin{proof}
      Consider the following square:
      \[
\begin{tikzcd}
	{\tilde{X}} & {\ord{n}} \\
	X & {\hat{X}}
	\arrow[two heads, from=1-1, to=2-1]
	\arrow[two heads, from=1-2, to=2-2]
	\arrow[hook, from=2-1, to=2-2]
	\arrow[hook, from=1-1, to=1-2]
\end{tikzcd}
      \]
      If $X$ is subfinitely enumerable, then $\tilde{X}$ exists and we may define $\hat{X}$ by pushout. Conversely, if $\hat{X}$ exists then we may define $\tilde{X}$ by pullback.
    \end{proof}

    Let's now prove some closure properties of properly finite sets.
    \begin{prop}\label{lem:properly.finite.closed.under.union}
      The following closure properties hold of properly finite sets.
      \begin{enumerate}
              \item Subsets of properly finite sets are properly finite.
              \item The product of two properly finite sets is properly finite.
              \item The pullback of properly finite sets is properly finite.
              \item A finite disjoint union of properly finite sets is properly finite.
              \item If $X$ is subfinitely enumerable and $S$ is discrete and $f : X \to S$, then $\im f \subseteq S$ is properly finite.
\item
Properly finite subsets of a discrete set are closed under finite union.
      \end{enumerate}
    \end{prop}
    \begin{proof}
      We will make use of \cref{lem:properly.finite.by.relation}.
      \begin{enumerate}
\item Suppose that $S \subseteq X$ and $X$ is properly finite as witnessed by the association $r \subseteq X \times \ord{n}$. Then $r$ restricted to $S$ remains an association, so that $S$ is subfinitely enumerable. Furthermore, as the subset of a discrete set, it is also discrete; it is therefore properly finite.
              \item Suppose that $X$ and $Y$ are properly finite as witnessed by the associations $r \subseteq X \times \ord{n}$ and $s \subseteq X \times \ord{m}$. Then the relation $(r \times s)((x, y), (i, j)) :\equiv r(x, i) \wedge s(y, j)$ gives an association $r \times x \subseteq (X \times Y) \times (\ord{n} \times \ord{m})$. The product of discrete types is discrete as well, since $\shape$ is a modality.
        \item The pullback $X \times_{Z} Y$ of sets is a subset of the product $X \times Y$, so this follows by the previous two properties.
        \item Let $X_{i}$ be properly finite sets for $i : \ord{k}$; we will show that $\dsum{i : \ord{k}} X_{i}$ is properly finite. Suppose that $r_{i} \subseteq X_{i} \times \ord{n_{i}}$ s an association witnessing the proper finiteness of $X_{i}$. Define $r((i, x), (i', j)) :\equiv (i = i') \wedge r_{i}(x, j)$ to be a relation $r \subseteq (\dsum{i : \ord{k}} X_{i}) \times (\dsum{i : \ord{k}} \ord{n_{i}})$; we will show that this is an association. First, for any $(i, x)$, there is a $j$ with $r_{i}(x, j)$, so that $r((i, x), (i, j))$. Second, if $r((i, x), (i_{1}, j_{1}))$ and $r((i, x), (i_{2}), j_{2})$, then $i_{1} = i = i_{2}$ and so $r_{i}(x, j_{1})$ and $r_{i}(x, j_{2})$, from which we conclude that $j_{1} = j_{2}$.
              \item If $X$ is subfinitely enumerable and $f : X \to S$, then $\im f$ is also subfinitely enumerable. But if $S$ is discrete, then $\im f$ is also discrete as the subset of a discrete set; therefore, it is properly finite.
              \item We combine the previous two closure properties. Let $X_{i}$ be properly finite subsets of a discrete set $S$ for $i : \ord{k}$. Then $\dsum{i : \ord{k}}X_{i}$ is properly finite and if we define $f : \dsum{i : \ord{k}}X_{i} \to S$ by $f(i, x) = x$, we see that $\im f = \bigcup_{i : \ord{k}}X_{i}$.
      \end{enumerate}
    \end{proof}

    \subsection{Dubuc-Penon compactness}\label{sec:compact}

In their paper \cite{Dubuc-Penon:Compact}, Dubuc and Penon introduce a very creative notion of compactness suitable for synthetic differential geometry.
    \begin{defn}[Dubuc-Penon, \cite{Dubuc-Penon:Compact}]\label{defn:Dubuc.Penon.compact}
      A type $K$ is \emph{Dubuc-Penon compact} if for every $A : \Prop$ and $B : K \to \Prop$,
      \[
( \forall x : K.\, A \vee B(x) ) \to (A \vee \forall x : K.\, B(x)).
      \]
      A map $f : X \to Y$ is \emph{Dubuc-Penon proper} if the inverse image of any Dubuc-Penon compact subtype of $Y$ is Dubuc-Penon compact.
    \end{defn}

    Dubuc and Penon prove in \cite{Dubuc-Penon:Compact} that the sheaves represented by compact ordinary manifolds in our intended models are Dubuc-Penon compact. This will allow us to assume at least that the unit interval is Dubuc-Penon compact.

\begin{axiom}\label{axiom:compact.interval}
The unit interval $[0, 1]$ is Dubuc-Penon compact.
\end{axiom}

    \begin{lem}\label{lem:compact.facts}
      We prove the following basic fact about Dubuc-Penon compact sets.
      \begin{enumerate}
              \item Finitely enumerable sets are Dubuc-Penon compact.
              \item If $K$ is Dubuc-Penon compact, and $f : K \to X$ is surjective, then $X$ is Dubuc-Penon compact.
              \item If $X + Y$ is Dubuc-Penon compact, then so are $X$ and $Y$.
      \end{enumerate}
    \end{lem}
    \begin{proof}
First, finite sets are Dubuc-Penon compact because universal quantification over a finite set is a finite conjunction, and disjunction commutes with finite conjunction. As the images of maps from finite sets, finitely enumerable sets will be compact by the second property.

Suppose that $K$ is Dubuc-Penon compact and $f : K \to X$ is surjective. Let $A  : \Prop$ and $B : X \to \Prop$, and suppose that for all $x : X$, $A$ holds or $B(x)$ holds. Then also for all $k : K$, $A$ holds or $B(f(k))$ holds, so by the compactness of $K$, $A$ holds or for all $k : K$, $B(f(k))$ holds. But since for all $x : X$, there is a $k$ such that $f(k) = x$, this suffices to show that $A$ holds or for all $x : X$, $B(x)$ holds.

Suppose that $X + Y$ is Dubuc-Penon compact, and let $A : \Prop$ and $B : X \to \Prop$ such that for all $x : X$, $A$ holds or $B(x)$ holds. Define $C : X + Y \to \Prop$ by $C(z) \equiv (z \in X) \to B(x)$. For all $z : X + Y$, either $z$ is in $X$ and so $A$ holds or $B(z)$ holds, or $z$ is in $Y$ and so $C(z)$ holds trivially; in either case, $A$ holds or $C(z)$ holds, so by compactness of $X + Y$, conclude that either $A$ holds or for all $z : X + Y$, $C(z)$ holds. But in that latter case, we see that $x : X$, $B(x)$ holds, and so $X$ is Dubuc-Penon compact. Of course, the argument for $Y$ is symmetric.
    \end{proof}

    \begin{warn}
      Although finitely enumerable sets are Dubuc-Penon compact, subfinite sets are not in general Dubuc-Penon compact. Every proposition is subfinite, but a proposition $P$ is Dubuc-Penon compact if and only if it is decideable: $P \vee \neg P$.
    \end{warn}

    \begin{prop}\label{lem:proper.iff.fibers.compact}
A map $f : X \to Y$ is Dubuc-Penon proper if and only if its fibers are Dubuc-Penon compact. As a corollary, Dubuc-Penon proper maps are closed under pullback.
    \end{prop}
    \begin{proof}
Since singletons are clearly compact, if $f : X \to Y$ is proper, then its fibers are compact. Conversely, suppose that the fibers of $f$ are compact, and let $K \subseteq Y$ be a compact subset; we aim to show that $f^{\ast}K :\equiv \{x : X \mid f(x) \in K\}$ is compact. So, let $A : \Prop$ and $B : f^{\ast}K \to \Prop$ and suppose that for all $x \in f^{\ast} K$, $A$ holds or $B(x)$ holds. Then also for all $y \in K$, and for all $x \in f^{\ast}\{y\}$, $A$ holds or $B(x)$ holds. By hypothesis, $f^{\ast}\{y\}$ is compact, so this means that for all $y \in K$, $A$ holds or for all $x \in f^{\ast}\{y\}$ $B(x)$ holds. But then we may appeal to the compactness of $K$ and see that either $A$ holds or for all $y \in K$ and $x \in f^{\ast}\{y\}$, $B(x)$ holds. But this means that either $A$ holds or for all $x \in f^{\ast}K$, $B(x)$ holds.

If $f : X \to Y$ is proper, and $g : A \to X$, then the pullback $g^{\ast}f : A \times_{Y} X \to A$ has equivalent fibers to those of $f$, and so is also proper.
    \end{proof}

    \begin{prop}\label{lem:sum.of.compact.is.compact}
Let $K$ be a Dubuc-Penon compact type and let $F : K \to \Type$ be a family of Dubuc-Penon compact types. Then the type of pairs $\dsum{k : K}F(k)$ is Dubuc-Penon compact. In particular, the product of two Dubuc-Penon compact types is Dubuc-Penon compact.
    \end{prop}
    \begin{proof}
      Let $A : \Prop$ and $B : \dsum{k : K} F(k) \to \Prop$ and suppose that for all $(k, x) : \dsum{k: K}F(k)$, $A$ holds or $B(k, x)$ holds. We may compute:
      \begin{align*}
        \forall (k, x) : \dsum {k : K} F(k).\, (A \vee B(k, x)) &\iff \forall k : K \forall x : F(k).\, (A \vee B(k, x)) \\
                                                                &\Rightarrow \forall k : K.\, (A \vee \forall x : F(k).\, B(k, x)) \\
                                                                &\Rightarrow A \vee (\forall k : K.\, \forall x : F(k). B(k, x)) \\
        &\iff A \vee (\forall (k, x) : \dsum{k : K} F(k).\, B(k, x)) \qedhere
      \end{align*}
    \end{proof}

    In his thesis, \cite{Gago:Thesis}, Gago proved that any positive, real valued function on a Dubuc-Penon compact set valued in $\Rb$ was bounded away from $0$ (\cref{lem:compact.bounded}). We will extract the method which Gago uses in his proof as \cref{prop:compact.open.relation.fatness}. This method makes use of Penon's logical definition of open set, and relies crucially on the Covering Property.
    \begin{defn}
A subtype $U \subseteq X$ is \emph{Penon open} if for all $x \in U$ and $y : X$, either $x \neq y$ or $y \in U$. A subtype $C \subseteq X$ is \emph{Penon closed} if its complement $X - C$ is Penon open.
    \end{defn}

    Penon opens form a topology on any type, and any function is continuous for the Penon topology. Any regular topology is finer than the Penon topology on its set of points; in particular, every open set in an ordinary manifold is Penon open.
    \begin{lem}\label{lem:regular.open.is.Penon.open}
Let $X$ be a regular topological space. Then any open set in $X$ is Penon open.
    \end{lem}
    \begin{proof}
Let $U$ be open in $X$, and let $x \in U$ and $y : X$. By the regularity of $X$, there is are open $V \subseteq U$ and $G$ with $x \in V$, $V \cap G = \empty$, and $U \cup G = X$. Therefore, $y \in U$ or $y \in G$; but if $y \in G$, $y$ cannot equal $x$, so we conclude that either $x \neq y$ or $y \in U$.
    \end{proof}

    \begin{thm}\label{prop:compact.open.relation.fatness}
Let $K$ be Dubuc-Penon compact and let $r : K \times \Rb \to \Prop$ be a relation which is Penon open as a subset of the product. If for all $k$, we have $r(k, x)$, then there is an $\ep > 0$ so that $r(k, y)$ for all $y \in B(x, \ep)$ and $k : K$.
    \end{thm}
    \begin{proof}
      That $r$ is Penon open means that for any $k$ and $x$ so that $r(k, x)$ and any other $q$ and $y$, we have $((k, x) \neq (q, y)$ or $r(q, y)$.
      So, supposing that $r(k, x)$ for all $k$, let $y : \Rb$ and note that for any $k : K$,
      \[
((k, x) \neq (k, y)) \vee r(k, y).
      \]
      Now, $(k, x) \neq (k, y)$ if and only if $x \neq y$, so we conclude that for any $k : K$, $(x \neq y)$ or $r(k, y)$. Therefore, by the compactness of $K$, either $(x \neq y)$ or for all $k : K$, $r(k, y)$. By the Covering Principle, then, either there is an $\ep > 0$ so that $B(x, \ep) \subseteq \{y \mid x \neq y\}$ or there is an $\ep > 0$ so that $B(x, \ep) \subseteq \{y \mid \forall k : K.\, r(k, y)\}$; since the former can't possibly be true, we conclude the latter.
    \end{proof}

    Let's give a careful proof that cartesian products of Penon open sets are open before deriving some corollaries of \cref{prop:compact.open.relation.fatness}.
    \begin{lem}\label{lem:cartesian.product.of.penon.open}
Let $U\subseteq X$ and $V \subseteq Y$ be Penon open subsets. Then $U \times V \subseteq X \times Y$ is Penon open.
    \end{lem}
    \begin{proof}
      Suppose that $(u, v) \in U \times V$ and let $(x, y) : X \times Y$, seeking
      \[
((u, v) \neq (x, y)) \vee ((x, y) \in U \times V).
\]
Expanding these assumptions out a bit, we see that $u \in U$ and $v \in V$, and it will suffice to prove
\[
(u \neq x) \vee (v \neq y) \vee (x \in U \wedge y \in V),
\]
which is equivalently
\[
((u \neq x) \vee (v \neq y) \vee (x \in U)) \wedge ((u \neq x) \vee (v \neq y) \vee (y \in V)).
\]
By hypothesis we have that $(u \neq x) \vee (x \in U)$ and $(v \neq y) \vee (y \in V)$ by the openness of $U$ and $V$. This clearly suffices.
    \end{proof}

    \begin{cor}[Gago, \cite{Gago:Thesis}]\label{lem:compact.bounded}
      Let $K$ be a Dubuc-Penon compact set.
      \begin{enumerate}
\item For every $f : K \to (0, \infty)$, there is an $\ep > 0$ so that $\ep < f(k)$ for all $ k : K$.
              \item For every $f : K \to \Rb$, there is a $B > 0$ so that $-B < f(x) < B$ for all $x : K$.
      \end{enumerate}
\end{cor}
\begin{proof}
To prove the first statement, let $r(k, x) \equiv (x < f(k))$ in \cref{prop:compact.open.relation.fatness}, and conclude that there is some $\delta > 0$ so that for all $ x \in B(0, \delta)$, we have $r(k, x)$ for all $k : K$. Define $\ep \equiv \frac{\delta}{2}$, and conclude that for all $k : K$, we have $\ep < f(k)$.

To prove the second statement, we use Postulate Exp to transform $f$ into a positive valued function. By the first statement, we have $\ep > 0$ so that $\ep < \exp f(k)$ for all $k$, and we also have $\delta > 0$ so that $\delta < \frac{1}{ \exp f(k) }$. By using a smooth approximation to the minimum function, we can ensure that both $\ep$ and $\delta$ are less than $1$. As a result, we see that $\log \ep < f(k) < - \log \delta$ for all $k : K$, and both $\log \ep$ and $\log \delta$ are positive. Then define $B$ to be any number bigger than both $\log \ep$ and $\log \delta$.
\end{proof}

We can use \cref{prop:compact.open.relation.fatness} to show that Dubuc-Penon compact subsets of $\Rb^{n}$ are ``compact'' in suitably weak senses. Let's give names to a few of these senses now.
\begin{defn}\label{defn:compact}
Consider the following notions of compactness. Let $K$ be a topological space.
\begin{enumerate}
        \item
$K$ is (open-cover) \emph{compact} if every open cover admits a finitely enumerable subcover. Explicitly, $K$ is compact if for any open cover $\Ua \subseteq \Oa(K)$, there is a finitely enumerable subset $\Va \subseteq \Ua$ for which $\bigcup_{V \in \Va} V = K$. \footnote{We will not be using this notion in this paper; it is included only for comparison with the subsequent notions. In particular, when I use the term ``compact'' in a proof, it usually means Dubuc-Penon compact. This should be clear from context.}
        \item $K$ is \emph{countably compact} if every countably enumerable open cover admits a finitely enumerable subcover.
        \item $K$ is \emph{subcompact} if every open cover admits a subfinitely enumerable subcover.
  \item $K$ is \emph{subcountably subcompact} if every subcountably enumerable cover admits a subfinitely enumerable subcover.
  \item
$K$ is \emph{refinement subcompact} if every open cover admits a subfinitely enumberable refinement. Explicitly, $K$ is refinement subcompact if for any open cover $\Ua \subseteq \Oa(K)$, there is a subfinitely enumerable cover $\Va \subseteq \Oa(K)$, so that for any $V \in \Va$, there is a $U \in \Ua$ such that $V \subseteq U$.
\end{enumerate}
\end{defn}

The definitions of ``compact'' and ``countably compact'' are standard in constructive mathematics. On the other hand, the definitions of ``subcompact'', ``subcountably subcompact'', and ``refinement subcompact'' are, as far as I know, novel.

\begin{rmk}
  Clearly, any compact set is countably compact, subcompact, subcountably subcompact, and refinement subcompact. Any subcompact set is subcountably subcompact, and refinement subcompact. However, countable compactness and subcountable subcompactness are generally incomparable; the latter applies to more covers but gives a weaker condition on the resulting subcover. This is summarized in the following diagram:
  
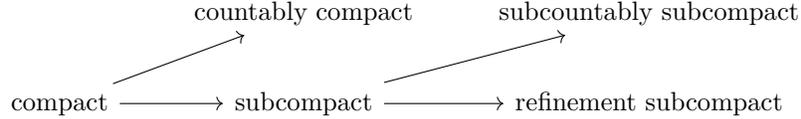
\[
\begin{tikzcd}
	& {\text{countably compact}} &  {\text{subcountably subcompact}}
 \\
 {\text{compact}} & {\text{subcompact}} &	{\text{refinement subcompact}}
 \arrow[from=2-1, to=1-2]
	\arrow[from=2-1, to=2-2]
	\arrow[from=2-2, to=1-3]
	\arrow[from=2-2, to=2-3]
\end{tikzcd}
  \]
\end{rmk}

We can prove that any Dubuc-Penon compact set is both countably compact and subcountably subcompact.
\begin{thm}\label{prop:db.compact.is.countably.compact}
Let $K$ be Dubuc-Penon compact. Then $K$ is subcountably subcompact and countably compact with regard to the Penon topology. That is, for any $I \subseteq \Nb$, any $I$-indexed Penon open cover $U : I \to \Oa(K)$ admits a subfinitely enumerable subcover. More explicitly, there is an $n : \Nb$ so that if $I_{< n} :\equiv \{i : I \mid i < n\}$, we have $K \subseteq \bigcup_{i \in I_{< n}} U_{i}$. In the case that $I = \Nb$, we see that $I_{<n} = \ord{n}$ is actually finite, so that we have a finitely enumerable subcover.
\end{thm}
\begin{proof}
  Let $r(k, x)$ be the relation
  \[
    r(k, x) :\equiv \exists i : I.\, (k \in U_{i}) \wedge \left(x < \frac{1}{i}\right).
  \]
This relation is Penon open, since it may be described as the union of a cartesian product of opens (\cref{lem:cartesian.product.of.penon.open}).
\[
r = \bigcup_{i : I} \left((K \cap U_{i}) \times \left(-\infty, \frac{1}{i}\right)\right).
\]
For all $k : K$, we have $r(k, 0)$; therefore, by \cref{prop:compact.open.relation.fatness} we may conclude that there is an $\ep > 0$ so that $B(0, \ep) \subseteq \{x : \Rb \mid \forall k : K. r(k, x) \}$. In particular, letting $n : \Nb$ being any number greater than $\frac{1}{\ep}$, we see that for all $k : K$ there is an $i : \Nb$ with $k \in U_{i}$ and $\frac{1}{n} < \frac{1}{i}$. This shows that $K = \bigcup_{i = 0}^{n} K \cap U_{i}$, which is a finite union.
\end{proof}

Using \cref{prop:db.compact.is.countably.compact}, we can prove that Dubuc-Penon compact subsets of ordinary manifolds are refinement subcompact.
\begin{cor}\label{cor:db.compact.in.manifold.is.refinement.subcompact}
Let $K$ be a Dubuc-Penon compact subset of a second countable topological space $X$ whose opens are Penon open. Then $K$ is refinement subcompact.
\end{cor}
\begin{proof}
Let $\Ua$ be an open cover of $K$, and let $B : \Nb \to \Oa(X)$ be a countable base for $X$. Define $I :\equiv \{i : \Nb \mid \exists U \in \Ua.\, B_{i} \subseteq U\}$ be the the set of indices of the base opens which are contained in opens of $\Ua$. For any $x \in K$, there is some $U \in \Ua$ with $x \in U$; since $B$ is a base, there is also some $i : \Nb$ so that $x \in B_{i} \subseteq U$. Therefore, $\Ba :\equiv \{B_{i} \mid i \in I\}$ remains an open cover and is a subcountable refinement of $\Ua$. But then, by \cref{prop:db.compact.is.countably.compact}, $\Ba$ admits a subfinitely enumerable subcover by the subcountable subcompactness of $K$.
\end{proof}

Clearly, in the prescence of the law of excluded middle, a space would be refinement subcompact if and only if it were compact --- LEM implies that subfinitely enumerable sets are finite, and we could then choose a subcover out of our refinement since the product of finitely many inhabited types is always inhabited. In fact, we can internalize this observation into a theorem about crisp refinement subcompact sets.
\begin{prop}\label{prop:crisply.refinement.subcompact.is.compact.for.crisp.covers}
Let $K$ be a crisp topological space which is crisply refinement subcompact. Then any crisp open cover $\Ua \subseteq \Oa(X)$ admits a finite subcover.
\end{prop}
\begin{proof}
By hypothesis, $\Ua$ admits a subfinitely enumerable refinement $\Va$, and we may take $\Va$ to be crisp. Since $\Va$ is crisply subfinitely enumerable, by \cref{prop:crisp.subfinitely.enumerable.is.finite}, $\Va$ is finite. But then we may choose for every $v \in \Va$ a $u_{v} \in \Ua$ with $v \subseteq u_{v}$, which gives us a crisp, finitely enumerable subcover of $\Ua$, which is therefore also finite by \cref{prop:crisp.subfinitely.enumerable.is.finite}.
\end{proof}

\begin{rmk}
Putting together \cref{prop:crisply.refinement.subcompact.is.compact.for.crisp.covers} with \cref{cor:db.compact.in.manifold.is.refinement.subcompact} proves that any crisp open cover of any crisp Dubuc-Penon compact ordinary manifold admits a finite subcover. This gives an internal proof of the external theorem of Dubuc and Penon \cite{Dubuc-Penon:Compact} that if the object of the Dubuc topos represented by a manifold $M$ is Dubuc-Penon compact, then $M$ is compact.
\end{rmk}

While we're here, we might as well prove the Heine-Borel property for Dubuc-Penon compact sets: they are both closed and bounded. The converse fails, however, see \cref{warn:not.heine.borel}.
\begin{prop}\label{lem:closed.subset.of.compact.is.compact}
Let $X$ be separated in the sense of Dubuc-Penon \cite{Dubuc-Penon:Compact}, namely if $x \neq y$, then for all $z$, either $x \neq z$ or $z \neq y$. Then if $K \subseteq X$ is Dubuc-Penon compact, it is Penon closed.
\end{prop}
\begin{proof}
We will show that $X - K$ is Penon open. Suppose that $x \not\in K$, and let $z : X$, seeking $x \neq z$ or $z \not\in K$. Now, for all $k \in K$, either $x \neq z$ or $z \neq k$ by the separatedness of $X$. Therefore, by the compactness of $K$, either $x \neq z$ or for all $k \in K$, $z \neq k$. But in the latter case, $z$ cannot be in $K$.
\end{proof}

\begin{lem}\label{lem:R.is.separated}
The smooth reals $\Rb$ are separated in the sense of Dubuc-Penon.
\end{lem}
\begin{proof}
It will suffice to show that if $x \neq 0$, then for any $y : \Rb$, either $x \neq y$ or $y \neq 0$.  This is equivalently asking whether $x - y \neq 0$ or $y \neq 0$. The result then follows since $\Rb$ is a local ring and a field (Postulate K): $(x - y) + y = x$ is invertible, and therefore either $(x - y)$ or $y$ is nonzero.
\end{proof}

\begin{thm}
  Any Dubuc-Penon compact subset of $\Rb$ is closed and bounded.
\end{thm}
\begin{proof}
  By \cref{lem:compact.bounded}, a Dubuc-Penon compact subset $K$ of $\Rb$ is bounded. Furthermore, as a compact subset of a separated space, $K$ is closed.
\end{proof}

\begin{warn}\label{warn:not.heine.borel}
  While Dubuc-Penon compact subsets of $\Rb$ are closed and bounded, the full Heine-Borel theorem does not hold. We can give an example of a closed and bounded subset which is not Dubuc-Penon compact. Let $z : \Rb$ be a real number, and let $S = \{x : \Rb \mid (x = 0) \wedge (z = 0)\}$. If $z = 0$, then $S = \{0\}$, and if $z \neq 0$, then $S = \emptyset$ --- in fact the set $S$ is equivalent to the proposition $(z = 0)$. If $S$ were Dubuc-Penon compact, then $(z = 0)$ or $(z \neq 0)$; this is because
  \[
(S \to ((z = 0) \vee \emptyset)) \simeq ((z = 0) \vee (S \to \emptyset)) \simeq ((z = 0) \vee (z \neq 0)),
\]
where we use the compactness of $S$ and the fact that $S \simeq (z = 0)$. We can conclude that $S$ is \emph{not} Dubuc-Penon compact, since equality of the reals is not decidable. However, $S$ is clearly bounded, and using Postulate K, we can also show that it is closed. Let $x \not\in S$ and $y : \Rb$, seeking $y \neq x$ or $y \not \in S$. The statement $x \not\in S$ means $\neg((x = 0) \wedge (z = 0))$, which by Postulate K is equivalent to $(x \neq 0) \vee (z \neq 0)$; similary, we are seeking to show that $(y \neq x) \vee (y \neq 0) \vee (z \neq 0)$. Of course, in the case that $(z \neq 0)$, we're done. On the other hand, if $(x \neq 0)$, then by the separatedness of $\Rb$ (\cref{lem:R.is.separated}), either $(x \neq y)$ or $(y \neq 0)$.

  The usual proof of the Heine-Borel theorem relies on the fact that closed subsets of compact sets are compact, in order to conclude the compactness of closed bounded subsets from the compactness of closed intervals. This is also false in general. Suppose that closed subsets of compact sets were compact; then, since $\ord{1}$ is compact, every proposition is compact. But a proposition $P$ is Dubuc-Penon compact if and only if $P \vee \neg P$, so this would imply the law of excluded middle.

Both of these counterexamples make essentially use of non-crisp propositions. This is because, by \cref{axiom:LEM}, all crisp propositions $P$ are decidable: $P \vee \neg P$. This leaves room for a theorem proving that crisp, closed and bounded subsets of $\Rb$ are Dubuc-Penon compact, but I do not know how to prove this.
\end{warn}

We can now set about proving the technical lemmas used in \cref{thm:ordinary.proper.etale.groupoid.is.orbifold}. First, we will show in \cref{thm:discrete.compact.subset.of.manifolds.is.properly.finite} that discrete, Dubuc-Penon compact subsets of ordinary manifolds are properly finite. Then, in the next section, we will show in \cref{lem:crystaline.subset.of.manifold.is.discrete} that crystaline subsets of crisp ordinary manifolds are discrete.

      \begin{lem}\label{thm:discrete.compact.subset.of.manifolds.is.properly.finite}
Let $K$ be a discrete, Dubuc-Penon compact subset of an ordinary manifold $M$. Then $K$ is properly finite.
      \end{lem}
      \begin{proof}
        By \cref{cor:db.compact.in.manifold.is.refinement.subcompact}, $K$ is refinement subcompact. Since it is discrete, every singleton is Penon open: for any $x$, either $y \neq x$ or $y = x$ (and so $y \in \{x\}$). Therefore, there is a subfinitely enumerable refinement of the cover of $K$ by its singletons. That is, there is a subfinitely enumerable set $\Sa \subset \Oa(K)$ consisting of subsingletons, such that $\bigcup_{S \in \Sa}S = K$.

        The relation $r(k, S) :\equiv (k \in S)$ gives an association of $K$ with $\Sa$. For $k : K$, there is some $S \in \Sa$ with $k \in S$ since this is a cover. If $k \in S$ and $k \in S'$, then $S = S'$ because they were assumed to be subsingletons and they both contain $k$ and are therefore both the singleton $\{k\}$. It follows by \cref{lem:properly.finite.by.relation} that $K$ is itself subfinitely enumerable. Since it was assumed to be discrete, this makes $K$ properly finite.
      \end{proof}

      \subsection{Crisp, ordinary proper \'etale groupoids are orbifolds}\label{sec:final.theorem}

\begin{lem}\label{lem:crystaline.in.R.discrete}
If a subset $C \subseteq \Rb^{n}$ is crystaline, then it is discrete.
\end{lem}
\begin{proof}
  We will show that every path $\gamma : \Rb \to C$ is constant. Since $C$ is crystaline, for any $t : \Rb$, the composite
  \[
\Dc_{t} \hookrightarrow \Rb \xto{\gamma} C
\]
is constant at $\gamma(t)$. In particular, we have that $\gamma(t + \ep) = \gamma(t)$ for each $\ep^{2} = 0$, since $t + \ep \approx \ep$. But then the coordinate functions $\gamma_{i} : \Rb \to \Rb$ are similarly unmoved by first order displacement, and so by the Principle of Constancy, the coordinate functions $\gamma_{i}$ are constant and so $\gamma$ is constant.
\end{proof}

\begin{lem}\label{lem:pullback.of.crystaline.along.etale.is.crystaline}
Let $f : X\to Y$ be $\Im$-\'etale. If $C$ is crystaline and $c : C \to Y$ is any map, then the pullback $f^{\ast}c : X \times_{Y} C \to X$ is crystaline.
\end{lem}
\begin{proof}
  By Theorem 3.20 of \cite{Cherubini:Thesis}, the pullback of a modally \'etale map is modally \'etale; therefore, the map $c^{\ast} f : X \times_{Y}C \to C$ is $\Im$-{\'etale}, and so the naturality square
  \[
\begin{tikzcd}
	{X \times_Y C} & C \\
	{\Im(X \times_Y C)} & {\Im C}
	\arrow[from=1-1, to=1-2]
	\arrow[from=1-2, to=2-2]
	\arrow[from=2-1, to=2-2]
	\arrow[from=1-1, to=2-1]
	\arrow["\lrcorner"{anchor=center, pos=0.125}, draw=none, from=1-1, to=2-2]
\end{tikzcd}
  \]
  is a pullback. By hypothesis, the unit $C \to \Im C$ is an equivalence, so we conclude that the unit $X \times_{Y} C \to \Im(X \times_{Y}C)$ is an equivalence.
\end{proof}

\begin{prop}[Shulman, Theorem 11.1 \cite{Shulman:Real.Cohesion}]\label{lem:R.is.connected}
The reals $\Rb$ are connected in the sense that if $X \cup Y = \Rb$ and both $X$ and $Y$ are nonempty, then $X \cap Y$ is nonempty.
\end{prop}
\begin{proof}
Suppose that $X \cap Y = \emptyset$. Then $\Rb \simeq X + Y$ is a disjoint union, and can define a function $f : \Rb \to \ord{2}$ defined by $f(x) = 0$ if $x \in X$ and $f(x) = 1$ if $x \in Y$. But $\ord{2}$ is crisply discrete, and so is in particular discrete, so $f$ factors through the shape $\shape \Rb$, which is the point. In other words, $f$ is constant, and so for all $x : \Rb$, $x \in X$, or for all $x : \Rb$, $x \in Y$. That is, $X = \Rb$ or $Y = \Rb$. Suppose that $X = \Rb$; then since $X \cap Y = \emptyset$, we can conlude that $Y = \emptyset$. But this contradicts our assumption that $Y$ was nonempty. Similarly, if $Y = \Rb$, then $X = \emptyset$, a contradiction. In either case, $X \cap Y$ cannot be empty.
\end{proof}

\begin{lem}\label{lem:crystaline.subset.of.manifold.is.discrete}
If a subset $C \subseteq M$ of a crisp ordinary manifold $M$ is crystaline, then it is discrete.
\end{lem}
\begin{proof}
  Let $\gamma : \Rb \to C$; we will show that $\gamma$ is constant by showing that $\gamma(t) = \gamma(0)$ for all $t : \Rb$. First, let's show that it suffices to prove that $\gamma(t) \approx \gamma(0)$. Consider the following pullback:
  \[
\begin{tikzcd}
	{\Dc_{\gamma(0)} \cap C} & C \\
	{\Dc_{\gamma(0)}} & M
	\arrow[from=1-1, to=2-1]
	\arrow[hook, from=2-1, to=2-2]
	\arrow[hook, from=1-2, to=2-2]
	\arrow[from=1-1, to=1-2]
	\arrow["\lrcorner"{anchor=center, pos=0.125}, draw=none, from=1-1, to=2-2]
\end{tikzcd}
  \]
  By \cref{lem:crisp.Penon.manifold.disk}, the inclusion $\Dc_{\gamma(0)} \hookrightarrow M$ is $\Im$-\'etale, and so by \cref{lem:pullback.of.crystaline.along.etale.is.crystaline} the pullback $\Dc_{\gamma(0)} \cap C$ is crystaline. But $\Dc_{\gamma(0)}$ embedds into $\Rb^{n}$ since ordinary manifolds are Penon manifolds, so $\Dc_{\gamma(0)} \cap C$ is a crystaline subset of $\Rb^{n}$ and therefore by \cref{lem:crystaline.in.R.discrete} is discrete. Therefore, equality in $\Dc_{\gamma(0)} \cap C$ is decidable; but since for any $x \in \Dc_{\gamma(0)}$, $x \approx \gamma(0)$, we can conclude that if $x$ is also in $C$ then $x = \gamma(0)$. That is, $\Dc_{\gamma(0)} \cap C = \{\gamma(0)\}$, and so if we prove that $\gamma(t) \approx \gamma(0)$, this will imply that $\gamma(t) = \gamma(0)$.

  Since $M$ was assumed to be crisp, we can take a crisp, countable open cover $M = \bigcup_{i : \Nb} U_{i}$. Any chart is infinitesimally stable, because $M$ is regular. Therefore, for any chart $\phi_{i} : \Rb^{n} \to M$ (with $U_{i} = \phi_{i}(\Rb^{n})$) in this cover, $\phi_{i}\inv C \subseteq \Rb^{n}$ is crystaline by \cref{cor:crisp.open.is.etale} and \cref{lem:pullback.of.crystaline.along.etale.is.crystaline}. The restriction $\gamma_{|U_{i}} : \gamma\inv(U_{i}) \to \phi_{i}\inv C$ is therefore constant.

  Now, let $t : \Rb$, seeking to prove that $\gamma(t) \approx \gamma(0)$. Since we are trying to prove a negative statement (namely, $\neg \neg(\gamma(t) = \gamma(0))$), we are free to use the law of excluded middle and double negation elimination. Let $N$ be a number so that $t$ and $0$ are both in $(-N, N)$, and therefore also $[-N, N]$. By \cref{axiom:compact.interval}, $[-N, N]$ is Dubuc-Penon compact and so by \cref{prop:db.compact.is.countably.compact} there is an $n : \Nb$ such that $[-N, N] \subseteq V_{1}, \ldots, V_{n}$. Therefore, there is some $i$ and $j$ so that $0 \in V_{i}$ and $t \in V_{j}$; let $W = \bigcup_{j\neq k \neq i} V_{k}$ be the rest of this finite cover. Now, either $W$ is empty or it isn't. If $W$ is empty, then $(-N, N) \subseteq V_{i} \cup V_{j}$ and so $V_{i} \cap V_{j}$ is nonempty by \cref{lem:R.is.connected}. But then we can assume that $x \in V_{i} \cap V_{j}$, so that $\gamma(t) = \gamma(x) = \gamma(0)$. On the other hand, if $W$ is nonempty, then either $V_{i} \cap V_{j}$ is nonempty or both of $W \cap V_{i}$ and $W \cap V_{j}$ are nonempty. In either case, we can assume we have inhabitants, and conclude that $\gamma(t) = \gamma(0)$.
\end{proof}

\begin{warn}
  We scraped through \cref{lem:crystaline.subset.of.manifold.is.discrete} by the skin of our teeth. It feels like it should be easier to prove. The real line is connected, and so we would expect that a locally constant function $f : \Rb \to X$ (valued in a set) should be constant. One way we might try to prove this general theorem is by showing that every open cover of $\Rb$ admits a \emph{chain} from $x$ to $y$: opens $U_{1}, \ldots, U_{n}$ in the cover with $x \in U_{1}$ and $y \in U_{n}$ and $U_{i} \cap U_{i+1}$ inhabited. We could then give $f(x) = f(y)$ by $f(x) = f(x_{1}) = \cdots = f(x_{n-2}) = f(y)$ with $x_{i} \in U_{i} \cap U_{i + 1}$, appealing to the constancy of $f$ on each $U_{i}$. There are a number of issues, however. First, the usual proof of this relies on showing that the set of all points $y$ for which there is a chain from $x$ is clopen and inhabited. However, I do not know how to prove that a clopen, inhabited subset $U$ of $\Rb$ is all of $\Rb$ --- the usual argument makes use of the classical fact that $U \cup (\Rb - U) = \Rb$, but this does not hold constructively. Second, even in the binary case it is not clear that if $U \cup V = \Rb$ with $U$ and $V$ open and inhabited, then $U \cap V$ must be inhabited (as opposed to \cref{lem:R.is.connected} which shows that $U \cap V$ is nonempty). This property is called \emph{overt connectedness} by Taylor \cite{Taylor:Real.Analysis}, and is proven for crisp subsets of $\Rb$ as Theorem 11.3 of \cite{Shulman:Real.Cohesion}.

Another approach to proving \cref{lem:crystaline.subset.of.manifold.is.discrete} would be to prove that the union of discrete subsets is discrete, or more narrowly that the union of countably many discrete subsets is discrete. However, discrete subsets are not even closed under binary union. Let $z : \Rb$ and consider the quotient $Q_{z} = \{0, 1\}/\sim$ where $0 \sim 1$ iff $z = 0$ (in more type theoretic language, this is the suspension $\Sigma(z =0)$ of the proposition $(z = 0)$). By definition, $[0] = [1]$ in $Q_{z}$ if and only if $z = 0$, so $Q_{z}$ is not discrete; if it were, then it would have decidable equality and so we could decide whether or not $(z = 0)$. But $\{[0]\} \cup \{[1]\} = Q_{z}$ since the quotient map $[-] : \{0, 1\} \to Q_{z}$ is surjective, and singletons are discrete. There is a nice topological interpretation of this counterexample: the projection $\fst : \dsum{z : \Rb} Q_{z} \to \Rb$ is the codiagonal map $\Rb \vee \Rb \to \Rb$ from the wedge of $\Rb$ with itself, pointed at $0$. The failure of the discreteness of $Q_{z}$ reflects the failure of this map to be a covering map; see also Remark 9.9 of \cite{Jaz:Good.Fibrations}.

However, the use of compactness --- and therefore of \cref{axiom:compact.interval} which asserts the Dubuc-Penon compactness of the unit interval --- is likely inessential.

\end{warn}

Finally, we can introduce the definition of an ordinary proper \'etale pregroupoid, and prove that the Rezk completions of crisp ordinary proper \'etale pregroupoids are orbifolds.
   \begin{defn}\label{defn:ordinary.proper.etale.groupoid}
        An \emph{ordinary proper \'etale pregroupoid} is a pregroupoid $\Ga$ satisfying the following conditions:
        \begin{enumerate}
          \item The type $\Ga_{0}$ of objects and $\Ga_{1}$ of morphisms are ordinary manifolds.
                \item The source map $s : \Ga_{1} \to \Ga_{0}$ is $\Im$-{\'etale} (which, by \cref{lem:etale.means.etale.for.ordinary.manifolds}, means that it is a local diffeomorphism in the usual sense).
                \item The map $(s, t) : \Ga_{1} \to \Ga_{0} \times \Ga_{0}$ sending a morphism to its source and target is Dubuc-Penon proper.
          \end{enumerate}
      \end{defn}

      \begin{thm}\label{thm:ordinary.proper.etale.groupoid.is.orbifold}
The Rezk completion of a crisp ordinary proper \'etale pregroupoid is an orbifold in the sense of \cref{defn:orbifold}.
      \end{thm}
      \begin{proof}
        Since a crisp ordinary proper \'etale pregroupoid $\Ga$ is in particular a crisp \'etale pregroupoid, its Rezk completion $\type{r}\Ga$ is microlinear by \cref{thm:etale.groupoid.microlinear}. Therefore, it remains to show that the types of identifications in $\type{r}\Ga$ are properly finite. By the Yoneda lemma, $(\yo(x) = \yo(y))$ is equivalent to $\Ga(x, y) \subseteq \Ga_{1}$, so it will suffice to show that $\Ga(x, y)$ is properly finite. Since $(s, t) : \Ga_{1} \to \Ga_{0} \times \Ga_{0}$ is Dubuc-Penon proper and singletons are Dubuc-Penon compact, the inverse image $\Ga(x, y) \simeq \fib_{(s, t)}(x, y)$ is Dubuc-Penon compact.
As a subset of the crystaline set $\dsum{z : \Ga_{0}} \Ga(x, z)$ --- which is the fiber of $s : \Ga_{1}\to \Ga_{0}$ over $x$ --- the set $\Ga(x, y)$ is crystaline. Therefore, it is discrete by \cref{lem:crystaline.subset.of.manifold.is.discrete}.
Then, by \cref{thm:discrete.compact.subset.of.manifolds.is.properly.finite}, we may conclude that $\Ga(x, y)$ is properly finite.

      \end{proof}

    \subsection{Global quotient orbifolds, and pullback of orbifolds}\label{sec:global.quotient}

    We finally have the definitions of microlinear types and properly finite sets so that \cref{defn:orbifold} is a fully precise definition of orbifold. Let's quickly show that the good orbifolds --- the quotients of the actions of finite groups on microlinear types --- are orbifolds in this sense. Then, we will show that orbifolds are closed under pullback, and use this fact to conclude that the \emph{inertia orbifold} of an orbifold is, in fact, an orbifold.

    \begin{thm}\label{thm:good.orbifold}
The quotient $X \sslash \Gamma$ of a microlinear set $X$ by the action of a crisp, finite group $\Gamma$ is an orbifold in the sense of \cref{defn:orbifold}. These are the \emph{global quotient} orbifolds.
    \end{thm}
    \begin{proof}
By \cref{thm:discrete.homotopy.quotient.is.microlinear}, $X \sslash \Gamma$ is microlinear; it remains to show that its types of identifications are properly finite. Since $q : X \to X \sslash \Gamma$ is surjective, we may consider the types of identifications $q(x) = q(y)$ for $x,\, y : X$. By \cref{lem:identifications.in.homotopy.quotient}, this type is equivalent to the type $\dsum{\gamma : \Gamma} (\gamma x = y)$. Since $X$ is a set, the type $(\gamma x = y)$ is a proposition, and so $\dsum{\gamma : \Gamma} (\gamma x = y)$ is a subset of $\Gamma$. Therefore, $q(x) = q(y)$ is subfinite, and so properly finite.
    \end{proof}

    \begin{prop}\label{prop:orbifold.pullback}
Orbifolds are closed under pullback.
    \end{prop}
    \begin{proof}
      Suppose that
      \[
\begin{tikzcd}
	A & B \\
	C & D
	\arrow["k"', from=1-1, to=2-1]
	\arrow["h", from=1-1, to=1-2]
	\arrow["f", from=1-2, to=2-2]
	\arrow["g"', from=2-1, to=2-2]
	\arrow["\lrcorner"{anchor=center, pos=0.125}, draw=none, from=1-1, to=2-2]
\end{tikzcd}
      \]
      is a pullback square commuting via $S : \dprod{a : A} fha = gka$, and that $B$, $C$ and $D$ are orbifolds. Since microlinear types are closed under pullback, $A$ is also microlinear; it remains to show that it has properly finite identification types. We have an equivalence
      \[
(a = a') \simeq \left(\dsum{p : ha = ha'} \dsum{q : ka = ka'} ( f_{\ast}p \bullet S_{a'} = S_{a} \bullet g_{\ast}q ) \right)
\]
given by the computation of identitication types of pair types. The first two factors are properly finite by assumption; the third is a proposition. Therefore, $(a = a')$ is a subset of a product of properly finite sets and is therefore properly finite by \cref{lem:properly.finite.closed.under.union}.
    \end{proof}

We can use \cref{prop:orbifold.pullback} to show that the \emph{inertia orbifolds} are orbifolds.
\begin{defn}
If $X$ is an orbifold, its \emph{inertia orbifold} is the type $\dsum{x : X} (x = x)$, or equivalently the type $X^{S^{1}}$, the free loop type of $X$.
\end{defn}

    \begin{cor}
The inertia orbifold $X^{S^{1}}$ of an orbifold is an orbifold in the sense of \cref{defn:orbifold}.
    \end{cor}
    \begin{proof}
      We have a pullback
      \[
\begin{tikzcd}
	{X^{S^1}} & X \\
	X & {X\times X}
	\arrow["\Delta", from=1-2, to=2-2]
	\arrow["\Delta"', from=2-1, to=2-2]
	\arrow[from=1-1, to=2-1]
	\arrow[from=1-1, to=1-2]
	\arrow["\lrcorner"{anchor=center, pos=0.125}, draw=none, from=1-1, to=2-2]
\end{tikzcd}
      \]
      So, by \cref{prop:orbifold.pullback}, $X^{S^{1}}$ is an orbifold.
    \end{proof}

\section{Conclusion}

We began this paper by seeing that examples of orbifolds can be constructed explicitly and intuitively in homotopy type theory using the yoga of higher groups. Orbifolds constructed in this way have meaningful points, and we can study these orbifolds in terms of their points --- just as we might study a particular set by understanding its elements.

We then showed that the axioms and notions of synthetic differential geometry generalize smoothly to higher types when interpreted in homotopy type theory. In particular, we saw that a variety of locally discrete higher types --- types with discrete types of identifications --- are microlinear, using the exact definition which has become standard in SDG.

In particular, we saw a definition of orbifold which closely reflects the intuitive idea of an orbifold as a smooth space whose points have finite automorphism groups. Since finiteness has many constructive incarnations, we had to take a detour to understand finiteness and compactness in SDG. But, in the end, we saw that any proper \'etale groupoid in the ordinary, external sense (which internally is a crisp ordinary proper \'etale pregroupoid) presents an orbifold in the sense of the new definition.

\appendix

\section{Tiny Types}\label{sec:tiny}

A remarkable feature of synthetic differential geometry is the \emph{tinyness} of the infinitesimals and infinitesimal varieties. This is Postulate W of \cref{axiom:SDG}. The importance of tiny objects for SDG was first realized by Lawvere \cite{Lawvere:SDG}, and their elementary theory was worked out by Yetter \cite{Yetter:tiny}.

In this subsection, we will develop just enough of the theory of tiny types for our purposes in this paper --- in particular, enough to prove that localization at the type $\Dc$ of infinitesimals preserves crisp colimits (\cref{thm:Im.preserves.crisp.colimits}). A great deal more can be done with tiny types in synthetic differential geometry; for example, the construction of differential form classifiers, and from them the construction of classifying types for principal bundles with connection. We leave the further development of the theory of tiny types to future work.

A type $T$ is tiny if the functor $X \mapsto X^{T}$ has a \emph{right} adjoint, in addition to its usual left adjoint $X \mapsto T \times X$. This formulation is not quite correct: the adjunction can only exist \emph{externally} (see \cref{warn:internal.tiny}). To refer to the external world internally, we will use crisp types and the $\flat$ comodality.

Tiny types in a type theory with $\flat$ have been defined before by Licata, Orton, Pitts, and Spitters in Figure 1 of \cite{LOPS:Universes}. However, their definition is only coherent for set level objects, which suits their purposes because they interpret the type theory in a $1$-topos. Their axioms are also not propositional when applied to higher types. We will use a different definition which is coherent for higher types as well, and where being tiny is a propostion.

\begin{defn}\label{defn:tiny}
  A crisp type $T$ is \emph{tiny} when the following structure exists crisply:
  \begin{enumerate}
    \item For any crisp type $X$, a type $X^{\frac{1}{T}}$ and a map $\xi : (X^{\frac{1}{T}})^{T} \to X$.
    \item For any crisp types $X$ and $Y$, the map
          \[
          \Xi :\equiv \omega \mapsto [v \mapsto \xi(\omega \circ v) ] :(X \to Y^{\frac{1}{T}}) \to (X^{T} \to Y)
          \]
          is a $\flat$-equivalence: $\flat \Xi$ is an equivalence. That is,
          \[
\flat\Xi : \flat (X \to Y^{\frac{1}{T}}) \simeq \flat(X^{T} \to Y)
          \]
\end{enumerate}
\end{defn}

This definition is coherent because the assignment $X \mapsto X^{T}$ is $\infty$-functorial, and an $\infty$-functor $L : \Ca \to \Da$ has a right adjoint if and only if the slice category $L / d$ has a terminal object for each object $d \in \Da$ --- with no functoriality in $d$ assumed. \cref{defn:tiny} gives, roughly speaking, a terminal object $\xi : \left(Y^{\frac{1}{T}}\right)^{T} \to Y$ in $(-)^{T}/Y$; that $\flat \Xi$ is an equivalence says that for every object $X^{T} \to Y$ of $(-)^{T}/Y$, there is a unique map $X \to Y^{\frac{1}{T}}$ so that the triangle commutes. Since we want all of these statements to be external, we put them under a $\flat$.

\begin{warn}\label{warn:internal.tiny}
  By Theorem 1.4 of \cite{Yetter:Tiny} (or, a suitable adaptation to $\infty$-toposes), $T$ should be tiny in \emph{any} context, even one which is not crisp. However, as mentioned in the warning immediately following that theorem, it is not generally the case that the adjoint $X^{\frac{1}{T}}$ is stable under base-change. For this reason, we restrict $X$ to be crisp, since $X^{\frac{1}{T}}$ is stable under crisp base change.
\end{warn}

\begin{rmk}
In \cite{Lawvere:Data.Types}, Lawvere says that ``this possibility [of the existence of tiny types] does not seem to have been contemplated by combinatory logic; the formalism should be extended to enable treatment of so basic a situation.''. \cref{defn:tiny} does not constitute such an extension of the formalism of type theory. Rather, it is more of a hack, using the Shulman's modalities to internalize the external. Working with tiny objects as in \cref{defn:tiny} is not very different then working with them externally --- which is how they must be worked with in an internal logic without externalizing modalities, since the defining adjunctions only exist externally.
A real solution to Lawvere's challenge could be a novel type theory for tiny objects.
\end{rmk}

\begin{lem}\label{lem:naturality.of.tiny.adjunction.domain}
  Let $f :: X' \to X$ be a crisp map. Then the square
  \[
\begin{tikzcd}
	{(X \to Y^{\frac{1}{T}})} & {(X^T \to Y)} \\
	{(X' \to Y^{\frac{1}{T}})} & {(X'^T \to Y)}
	\arrow["{- \circ f}"', from=1-1, to=2-1]
	\arrow["\Xi", from=1-1, to=1-2]
	\arrow["{- \circ f^T}", from=1-2, to=2-2]
	\arrow["\Xi"', from=2-1, to=2-2]
\end{tikzcd}
  \]
  commutes.
\end{lem}
\begin{proof}
 Let $\omega : X \to Y^{\frac{1}{T}}$ and $v : T \to X'$, and compute:
  \begin{align*}
    \Xi(\omega) \circ f^{T}(v) &\equiv \xi(\omega \circ f \circ v) \\
    &\equiv \Xi(\omega \circ f)(v).
\end{align*}

\end{proof}

Since mapping out of a tiny type has a right adjoint, it commutes with all colimits. This is the sense in which tiny types are ``tiny'': that $X \mapsto (T \to X)$ commutes with all colimits is a very strong compactness property. Of course, the adjoint only exists for crisp types and the adjunction only holds for crisp maps, so we can only hope to commute with crisp colimits.
\begin{prop}\label{prop:tiny.commutes.colimits}
  Let $T$ be at tiny type. Then the functor $X \mapsto X^{T}$ preserves all crisp colimits, but in particular the following:
  \begin{enumerate}
    \item If $f, g :: A \to B$, then
          \[
\type{coeq}(f, g)^{T} \simeq \type{coeq}(f^{T}, g^{T})
          \]
          where $f^{T},g^{T} :: A^{T} \to B^{T}$ are given by post-composition.
    \item If the square on the left is a crisp pushout, then so is the square on the right:
          \[
\begin{tikzcd}
	A & B && {A^T} & {B^T} \\
	C & D && {C^T} & {D^T}
	\arrow[from=1-1, to=2-1]
	\arrow[from=1-1, to=1-2]
	\arrow[from=1-2, to=2-2]
	\arrow[from=2-1, to=2-2]
	\arrow["\lrcorner"{anchor=center, pos=0.125, rotate=180}, draw=none, from=2-2, to=1-1]
	\arrow[from=1-4, to=2-4]
	\arrow[from=1-4, to=1-5]
	\arrow[from=1-5, to=2-5]
	\arrow[from=2-4, to=2-5]
	\arrow["\lrcorner"{anchor=center, pos=0.125, rotate=180}, draw=none, from=2-5, to=1-4]
\end{tikzcd}
          \]
    \item If $A_{0} \to A_{1} \to \cdots$ is a crisp sequence with colimit $A_{\infty}$, then $A_{\infty}^{T}$ is the colimit of the sequence $A_{0}^{T} \to A_{1}^{T} \to \cdots$.
\end{enumerate}
\end{prop}
\begin{proof}
  In general, the arguments will go by showing that both $\colim D_{i}^{T}$ and $(\colim D_{i})^{T}$ have the same universal property in the $\infty$-category of crisp types with mapping types from $X$ to $Y$ being $\flat(X \to Y)$. This will give an equivalence $\flat(\colim D_{i}^{T} \simeq (\colim D_{i})^{T})$, from which we may conclude that $\colim D_{i}^{T} \simeq (\colim D_{i})^{T}$ as types (and therefore that they have the same universal property using the full mapping types $X \to Y$).

 We will prove the case of coequalizers; the other cases are similar. To that end, let $f, g :: A \to B$ be crisp maps. For any crisp type $Y$, we have equivalences:
  \begin{align*}
    \flat(\type{coeq}(f, g)^{T} \to Y) &\simeq \flat(\type{coeq}(f, g) \to Y^{\frac{1}{T}}) \\
                                       &\simeq \dsum{z : \flat(B \to Y^{\frac{1}{T}})} (\mbox{let $h^{\flat} :\equiv z$ in $\flat( h \circ f = h \circ g )$})\\
    &\simeq \dsum{z : \flat(B^{T} \to Y)} (\mbox{let $h^{\flat} :\equiv \Xi\inv(z)$ in $\flat( h \circ f = h \circ g )$})\\
    &\simeq \dsum{z : \flat(B^{T} \to Y)} (\mbox{let $h^{\flat} :\equiv \Xi\inv(z)$ in $( h^{\flat} \circ f^{\flat} = h^{\flat} \circ g^{\flat} )$})\\
    &\equiv \dsum{z : \flat(B^{T} \to Y)} (\Xi\inv(z) \circ f^{\flat} = \Xi\inv(z) \circ g^{\flat})\\
    &\simeq \dsum{z : \flat(B^{T} \to Y)} (\Xi\inv(z \circ ( f^{T} )^{\flat}) = \Xi\inv(z \circ (g^{T})^{\flat}))\\
    &\simeq \dsum{z : \flat(B^{T} \to Y)} (z \circ ( f^{T} )^{\flat} = z \circ (g^{T})^{\flat})\\
    &\simeq \flat(\type{coeq}(f^{T}, g^{T}) \to Y)
  \end{align*}
  By subsituting in $\type{coeq}(f, g)^{T}$ and $\type{coeq}(f^{T}, g^{T})$ in for $Y$, respectively, we can give an equivalence between them.
\end{proof}

\begin{lem}\label{lem:mapping.out.sum}
  Let $T$ be any type. If $A$ is $T$-null (the inclusion $A \to A^{T}$ of constants is and equivalence), then for any family $ B: A \to \Type$, we have
  \[
(\dsum{a : A} B(a))^{T} \simeq \dsum{a : A} B(a)^{T}.
  \]
\end{lem}
\begin{proof}
  If $A$ is $T$-null, then the fibers $\dsum{a : A} (f = \term{const}_{a})$ of the inclusion of constants $\term{const} : A \to A^{T}$ are contractible. Therefore, we have the following equivalences:
  \begin{align*}
    (T \to \dsum{a : A} B(a)) &\simeq \dsum{f : T \to A} (\dprod{t : T} B(ft)) \\
                              &\simeq \dsum{a : A} \dsum{f : T \to A} \dsum{f = \term{const}_{a}} (\dprod{t : T} B(ft)) \\
                              &\simeq \dsum{a : A} (\dprod{t : T} B(\term{const}_{a} t)) \\
    &\equiv \dsum{a : A} B(a)^{T}.
\end{align*}
\end{proof}

\begin{thm}\label{thm:tiny.commute.with.localization}
  Suppose that $T$ is a tiny type. Let $I$ be a crisply discrete, $T$-null type, and let $f_{i} :: P_{i} \to Q_{i}$ be a crisp family of $T$-null, sequentially compact types indexed by $i :: I$. (More formally, we have a crisp function $f :: I \to \dsum{X, Y : \Type}(X \to Y)$, and since $I$ is crisply discrete, we may assume any element of $I$ to be crisp) If $X$ is any crisp type, then
  \[
( L_{f}X )^{T} \simeq L_{f} (X^{T})
\]
where $L_{f}$ is the localization at the family $f$. As a corollary, we have
\[
\trunc{X}_{n}^{T} \simeq \trunc{X^{T}}_{n}
\]
for any $n : \Nb$.
\end{thm}
\begin{proof}
  We will use the construction of $L_{f} X$ as in Section 7.2 of Rijke's thesis \cite{Rijke:Thesis}. This construction proceeds as follows: for any type $A$, define $\type{QL}_{f} A$ to be the pushout
  \[
\begin{tikzcd}
	{\sum_{i : I} \left((P_i \times A^{P_i}) +_{(P_i \times A^{Q_i})} (Q_i \times A^{Q_i} \right)} & {\sum_{i : I} Q_i \times A^{P_i}} \\
	A & {\type{QL}_f A}
	\arrow[from=1-1, to=2-1]
	\arrow[from=1-1, to=1-2]
	\arrow[from=1-2, to=2-2]
	\arrow[from=2-1, to=2-2]
\end{tikzcd}
  \]
  where the top horizontal map is given by pushout-product and the left vertical map is given by evaluation (see Definition 7.2.6 of \cite{Rijke:Thesis} for details). We then define $L_{f} X$ to be the colimit of the sequence $X \to \type{QL}_{f} X \to \type{QL}_{f}\type{QL}_{f} X \to \cdots$. By \cref{prop:tiny.commutes.colimits}, $(L_{f} X)^{T}$ is equivalently the colimit of the sequence $X^{T} \to (\type{QL}_{f} X)^{T} \to (\type{QL}_{f} \type{QL}_{f} X)^{T} \to \cdots$. Therefore, it will suffice to show that for any crisp type $A$, we have $(\type{QL}_{f} A)^{T} \simeq \type{QL}_{f}(A^{T})$, natural for crisp maps in $A$.

  For this, we will appeal to \cref{lem:mapping.out.sum} and \cref{prop:tiny.commutes.colimits} again to see that we have equivalences:
  \[
\begin{tikzcd}
	{\left(\sum_{i : I} \left((P_i \times A^{P_i}) +_{(P_i \times A^{Q_i})} (Q_i \times A^{Q_i} \right)\right)^T} & {\left(\sum_{i : I} Q_i \times A^{P_i}\right)^T} \\
	{\sum_{i : I} \left((P_i \times (A^T)^{P_i}) +_{(P_i \times (A^T)^{Q_i})} (Q_i \times (A^T)^{Q_i} \right)} & {\sum_{i : I} Q_i \times (A^T)^{P_i}}
	\arrow[from=1-1, to=1-2]
	\arrow[from=2-1, to=2-2]
	\arrow[from=1-1, to=2-1, equals]
	\arrow[from=1-2, to=2-2, equals]
\end{tikzcd}
  \]
  We are making use of the fact that $I$, $P_{i}$ and $Q_{i}$ are all $T$-null, and that $(-)^{T}$ commutes with crisp pushouts.

\end{proof}

\begin{cor}\label{cor:tiny.connectedness}
  Let $T$ be a tiny type. If $X$ is crisp and $k$-connected, then $X^{T}$ is $k$-connected. In particular, for $k$-commutative higher groups $G$,
  \[
    \B^{k+1}(G^T) \simeq (\B^{k+1} G)^{T}.
  \]
\end{cor}
\begin{proof}
By \cref{thm:tiny.commute.with.localization}, we have $\trunc{X^{T}}_{k} \simeq \trunc{X}_{k}^{T} \simeq \ast^{T} \simeq \ast$.
\end{proof}

\begin{thm}\label{thm:tiny.null.preservers.crisp.colimits}
  Let $I$ be a crisply discrete type and let $T :: I \to \Type$ be a family of tiny types. Let $L_{T}$ be the modality given by nullifiying the $T_{i}$. Then $L_{T}$ commutes with crisp colimits: in particular, with pushouts and colimits of sequences.
  \[
L_{T}(\colim D_{j}) \simeq \colim L_{T} D_{j}.
  \]
\end{thm}
\begin{proof}
The argument is the same for any expressible colimit (coequalizers, pushouts, colimits of sequences, etc.). Both $L_{T}(\colim D_{j})$ and $\colim L_{T} D_{j}$ are universal for cones under $D$ mapping into $L_{T}$-modal types. Since $L_{T}(\colim D_{j})$ is an $L_{T}$-modal type admitting a cone under $D$,  we have a map $\colim L_{T} D_{j} \to L_{T}(\colim D_{j})$. To show that this map is an equivalence, it will suffice to show that $\colim L_{T}D_{j}$ is $L_{T}$-modal; then we can discharge the universal property of $L_{T}(\colim D_{j})$ to construct an inverse.

Being $L_{T}$-modal means being $T_{i}$-null for all $i : I$. That is, we need to show that $\type{const} : \colim L_{T}D_{j} \to (\colim L_{T} D_{j})^{T}$ is an equivalence. But by \cref{prop:tiny.commutes.colimits}, we have $(\colim L_{T}D_{j})^{T} \simeq \colim ((L_{T} D_{j})^{T})$, and $L_{T}D_{j}$ is $T$-null by construction.
\end{proof}

\printbibliography

\end{document}